\documentclass[11pt, twoside, leqno]{article}

\usepackage{amssymb}
\usepackage{amsmath}
\usepackage{amsthm}
\usepackage{color}
\usepackage{mathrsfs}

\usepackage{indentfirst}

\usepackage{txfonts}

\allowdisplaybreaks

\def\red{\color{red}}

\def\rr{{\mathbb R}}
\def\rn{{\mathbb{R}^n}}

\def\zz{{\mathbb Z}}

\def\nn{{\mathbb N}}

\def\cd{{\mathcal D}}

\def\cf{{\mathcal F}}

\def\cp{{\mathcal P}}

\def\cs{{\mathcal S}}

\def\cx{{\mathcal X}}
\def\cy{{\mathcal Y}}

\def\fz{\infty }
\def\az{\alpha}

\def\dz{\delta}

\def\lz{\lambda}

\def\tz{\theta}

\def\lf{\left}
\def\r{\right}

\def\la{\langle}
\def\ra{\rangle}
\def\hs{\hspace{0.35cm}}
\def\ls{\lesssim}

\def\noz{\nonumber}
\def\wz{\widetilde}

\def\gfz{\genfrac{}{}{0pt}{}}

\def\dsum{\displaystyle\sum}
\def\dint{\displaystyle\int}
\def\dfrac{\displaystyle\frac}
\def\dsup{\displaystyle\sup}

\def\dlim{\displaystyle\lim}

\def\loc{{\mathop\mathrm{\,loc\,}}}
\def\supp{\mathop\mathrm{\,supp\,}}

\def\BMO{\mathop\mathrm{\,BMO\,}}
\def\bmo{\mathop\mathrm{\,bmo\,}}

\def\Xint#1{\mathchoice
{\XXint\displaystyle\textstyle{#1}}%
{\XXint\textstyle\scriptstyle{#1}}%
{\XXint\scriptstyle\scriptscriptstyle{#1}}%
{\XXint\scriptscriptstyle\scriptscriptstyle{#1}}%
\!\int}
\def\XXint#1#2#3{{\setbox0=\hbox{$#1{#2#3}{\int}$ }
\vcenter{\hbox{$#2#3$ }}\kern-.6\wd0}}

\def\dashint{\Xint-}

 \def\la{{\langle}}
 \def\ra{{\rangle}}
\def\({\left(}
\def \){ \right)}

\def\lz{{\lambda}}

 \def\BB{{\mathbb B}}

\newtheorem{theorem}{Theorem}[section]
\newtheorem{lemma}[theorem]{Lemma}
\newtheorem{corollary}[theorem]{Corollary}

\newtheorem{proposition}[theorem]{Proposition}

\theoremstyle{definition}
\newtheorem{remark}[theorem]{Remark}
\newtheorem{definition}[theorem]{Definition}
\renewcommand{\appendix}{\par
   \setcounter{section}{0}%
   \setcounter{subsection}{0}%
   \setcounter{subsubsection}{0}%
   \gdef\thesection{\@Alph\c@section}%
   \gdef\thesubsection{\@Alph\c@section.\@arabic\c@subsection}%
   \gdef\theHsection{\@Alph\c@section.}%
   \gdef\theHsubsection{\@Alph\c@section.\@arabic\c@subsection}%
   \csname appendixmore\endcsname
 }

\numberwithin{equation}{section}

\begin{document}

\arraycolsep=1pt

\title{\bf\Large Bilinear Decomposition and Divergence-Curl Estimates on Products
Related to Local Hardy Spaces and Their Dual Spaces
\footnotetext{\hspace{-0.35cm} 2020 {\it
Mathematics Subject Classification}. Primary
42B30; Secondary 42B35, 42B15, 46E30, 42C40.
\endgraf {\it Key words and phrases.} (local) Hardy space, dual space,
Orlicz space, bilinear decomposition, divergence-curl estimate, renormalization of wavelets, atom.
\endgraf This project is supported
by the National Natural Science Foundation of China (Grant Nos.
11761131002, 11971058, 11671185 and 11871100).}}
\author{Yangyang Zhang, Dachun Yang\,\footnote{Corresponding
author, E-mail: \texttt{dcyang@bnu.edu.cn}/{\red April 28, 2020}/Final version.}
\  and Wen Yuan}
\date{}
\maketitle

\vspace{-0.7cm}

\begin{center}
\begin{minipage}{13cm}
{\small {\bf Abstract}\quad
Let $p\in(0,1)$, $\alpha:=1/p-1$ and, for any $\tau\in [0,\infty)$,
$\Phi_{p}(\tau):=\tau/(1+\tau^{1-p})$.
Let $H^p(\mathbb R^n)$, $h^p(\mathbb R^n)$ and $\Lambda_{n\alpha}(\mathbb{R}^n)$ be,
respectively, the Hardy space,
the local Hardy space and the inhomogeneous Lipschitz space
on $\mathbb{R}^n$. In this article, applying the
inhomogeneous renormalization of wavelets,
the authors establish a bilinear decomposition
for multiplications of elements in
$h^p(\mathbb R^n)$ [or $H^p(\mathbb R^n)$] and $\Lambda_{n\alpha}(\mathbb{R}^n)$,
and prove that these bilinear decompositions are sharp in some sense.
As applications, the authors also obtain some estimates of the product of elements in
the local Hardy space $h^p(\mathbb R^n)$ with $p\in(0,1]$ and its
dual space, respectively, with zero $\lfloor n\alpha\rfloor$-inhomogeneous curl and
zero divergence, where $\lfloor n\alpha\rfloor$ denotes the largest integer
not greater than $n\alpha$. Moreover, the authors find new structures
of $h^{\Phi_p}(\mathbb R^n)$ and $H^{\Phi_p}(\mathbb R^n)$ by showing that
$h^{\Phi_p}(\mathbb R^n)=h^1(\mathbb R^n)+h^p(\mathbb R^n)$ and
$H^{\Phi_p}(\mathbb R^n)=H^1(\mathbb R^n)+H^p(\mathbb R^n)$
with equivalent quasi-norms,
and also prove that the dual spaces of both $h^{\Phi_p}(\mathbb R^n)$ and $h^p(\mathbb R^n)$
coincide. These results give a complete picture on the
multiplication between the local Hardy space and
its dual space.}
\end{minipage}
\end{center}

\vspace{0.16cm}

\tableofcontents

\vspace{0.2cm}

\section{Introduction}\label{s1}

Motivated by developments in the geometric
function theory and the nonlinear elasticity (see, for instance, \cite{AIKM00,Ba77,Mu88,Mu90}),
Bonami et al. \cite{BIJZ07} initiated the study
of the bilinear decomposition of the product of Hardy spaces and their
dual spaces, which plays key roles
in improving the estimates of many nonlinear quantities
such as div-curl products, weak Jacobians (see, for instance,  \cite{BGK12,BFG,CLMS})
and commutators (see, for instance,  \cite{Ky13,LCFY}).
These works
\cite{BGK12,BFG,CLMS,Ky13,LCFY} inspire many new ideas in various research areas of mathematics
such as the compensated compactness theory in the
nonlinear partial differential equations and the study of the  existence and the regularity
for solutions to partial differential equations
where the uniform ellipticity is lost; see \cite{BIJZ07,He,IS,Le02,Mur78,Ta79} and their references.

The first significant result in this direction was made
by Bonami et al. \cite{BIJZ07}.
It was proved in \cite{BIJZ07} the following \emph{linear} decomposition:
\begin{align}\label{hrn}
H^1(\rn)\times\BMO(\rn)\subset L^1(\rn)+H_w^\Phi(\rn),
\end{align}
where $H_w^\Phi(\rn)$ denotes the weighted Orlicz Hardy
space associated to the weight function
$w(x):= 1/{\log(e+|x|)}$ for any $x\in\rn$
and to the Orlicz function
\begin{align}\label{111}
\Phi(\tau):= \tau/ {\log(e + \tau)},\ \ \forall\,\tau\in [0,\fz).
\end{align}
Precisely, for any given $f\in H^1(\rn)$, there
exist two bounded linear operators
$S_f:\, \BMO(\rn)\to L^1(\rn)$
and
$T_f:\, \BMO(\rn)\to H_w^\Phi(\rn)$
such that, for any $g\in \BMO(\rn)$,
$f\times g=S_f g+T_f g$.
This result was essentially improved by Bonami et al. in \cite{BGK12}, where they found two
bounded \emph{bilinear} operators $S$ and $T$ such that
the aforementioned decomposition \eqref{hrn} still holds true.
Precisely, via the wavelet multiresolution analysis, Bonami et al. \cite{BGK12}
proved the following
\emph{bilinear} decomposition:
\begin{align}\label{11z}
H^1(\rn)\times\BMO(\rn)\subset L^1(\rn)+H^{\rm log}(\rn),
\end{align}
where $H^{\rm log}(\rn)$ denotes
the Musielak--Orlicz Hardy space related to the Musielak--Orlicz function
\begin{align}\label{222}
\theta(x,\tau):=\frac{\tau}{\log(e+|x|)+\log(e+\tau)}
,\ \ \forall\,x\in\rn,\ \forall\,\tau\in [0,\fz)
\end{align}
(see \cite{Ky14}), which is smaller than  $H_w^\Phi(\rn)$ in \eqref{hrn}.
By proving that the dual space of $ H^{\rm log}(\rn)$
is the generalized ${\rm{BMO}}$ space ${\rm{BMO}^{\rm log}(\rn)}$
introduced by Nakai and Yabuta  \cite{NY85}, which also characterizes the set of
multipliers of $\BMO(\rn)$ (see also the recent survey \cite{n17} on this subject of Nakai),
Bonami et al. in \cite{BGK12} deduced that
$H^{\rm log}(\rn)$ in \eqref{11z} is sharp in some sense.
Moreover, Bonami et al. in \cite{bfgk}
proved that every atom of $H^{\rm log}(\rn)$ can be written
as a finite linear combination of product distributions in the
space $H^1(\rn)\times\BMO(\rn)$ and, in dimension one,
Bonami and Ky in \cite{BK} proved that $H^{\rm log}(\rr)$ is indeed the smallest space satisfying \eqref{11z}.
Recently, in \cite{bckly,blyy}, a bilinear decomposition theorem for multiplications of elements in
$H^p(\mathbb R^n)$
and its dual space $\mathfrak{C}_{\alpha}(\mathbb{R}^n)$
was established when $p\in(0,1)$ and $\alpha:=1/p-1$,
and the sharpness of this bilinear decomposition was also obtained therein.
The div-curl lemma is a prominent
tool in the study of nonlinear partial differential equations via the method
of compensated compactness (see, for instance, \cite{Mur78,Ta79}),
which has been investigated in
\cite{bckly,BFG,BGK12,CLMS}. When $p\in(1,\infty)$, with the Hardy space $H^1({\mathbb {R}^n})$
as the target space, some
estimates of the div-curl product of functions in
$L^p({\mathbb {R}^n})$ and $L^{p'}({\mathbb {R}^n})$ were established in \cite{CLMS},
where $1/p+1/p'=1$.
Using the local Hardy space $h^1({\mathbb {R}^n})$ as the target space,
Dafni in \cite{Da95} obtained
some nonhomogeneous estimates of div-curl products of functions in
$L^p({\mathbb {R}^n})$ and $L^{p'}({\mathbb {R}^n})$ with $p\in(1,\infty)$,
which was further developed, and applied to the divergence-curl decomposition
of the corresponding Hardy space on $\rn$ or its domains in Chang et al. \cite{cdy16,cdy10,cdy09,cds05}.
When $p\in(0,1]$,
as an application of the obtained bilinear decomposition, in \cite{bckly,BFG,BGK12}, some
estimates of the div-curl products of elements in the
Hardy space $H^{p}(\mathbb R^n)$ and its dual space were also established.
As for the local Hardy space, Cao et al. \cite{cky1} established an estimate of
div-curl products of functions in the local Hardy space $h^1(\rn)$ and $\bmo(\rn)$,
and no other estimates of div-curl products of elements in
the local Hardy space $h^p(\rn)$ with $p\in (0,1)$ and its dual space are known so far.

On the other hand, for the local Hardy space,
Bonami et al. \cite{bf} established some \emph{linear} decomposition
of the product of the local Hardy space and its dual space. Let $p\in(0,1)$ and $\alpha:=1/p-1$.
Precisely, it was proved in \cite{bf}
the following \emph{linear} decomposition:
$$
h^p(\rn)\times\Lambda_{n\az}(\rn)\subset L^1(\rn)+h^{p}(\rn),
$$
where $h^p(\mathbb R^n)$ denotes the local Hardy space introduced by Goldberg
\cite{g} and
$\Lambda_{n\alpha}(\mathbb{R}^n)$
the inhomogeneous Lipschitz space.
Moreover, Bonami et al. \cite{bf} introduced the local Hardy-type space $h_*^\Phi(\rn)$, where
$\Phi$ is as in \eqref{111},
and, using $h_*^\Phi(\rn)$ as the target space, Bonami et al. proved the
following \emph{linear} decomposition:
$$
h^1(\rn)\times\bmo(\rn)\subset L^1(\rn)+h_\ast^{\Phi}(\rn).
$$
Recently,
Cao et al. \cite{cky1} obtained the \emph{bilinear} decomposition of product distributions in the
local Hardy space $h^p(\rn)$ and its dual space for $p\in(\frac{n}{n+1},1]$.
Precisely, Cao et al. \cite{cky1} proved that
\begin{align}\label{cao1}
h^1(\rn)\times\bmo(\rn)\subset L^1(\rn)+h_\ast^{\Phi}(\rn),
\end{align}
\begin{align}\label{cao2}
h^1(\rn)\times\bmo(\rn)\subset L^1(\rn)+h^{\log}(\rn)
\end{align}
and, for any $p\in(\frac{n}{n+1},1)$ and $\alpha:=1/p-1$,
\begin{align}\label{cao}
h^p(\rn)\times\Lambda_{n\az}(\rn)\subset L^1(\rn)+h^{p}(\rn),
\end{align}
where $h^{\rm log}(\rn)$ (see \cite{yy2})
denotes the local Hardy space of Musielak--Orlicz
type associated to the Musielak--Orlicz function
$\tz$ as in \eqref{222}.
Cao et al. \cite{cky1} also established an estimate in $h_\ast^{\Phi}(\mathbb R^n)$ of div-curl products
of functions in $h^1(\rn)$ and $\bmo(\mathbb{R}^n)$, however, there exists a gap
in the proof of this div-curl estimate in \cite{cky1} because, in this proof, Cao et al. used
the boundedness on $h^1(\rn)$ of the Riesz transforms, which is not true.
The sharpness of the bilinear decompositions
\eqref{cao1} and \eqref{cao2} was recently obtained in \cite{zyy} by proving that \eqref{cao1} is sharp,
while \eqref{cao2} is not sharp. Moreover,
there exists a gap in the proof of the bilinear decomposition \eqref{cao} in \cite[Theorem 1.1(i)]{cky1}
and hence \eqref{cao} is questionable.
Thus, it is a quite natural question to find
a suitable local Hardy-type space which can
give a sharp bilinear decomposition of the product of the
local Hardy space $h^p(\rn)$ and its dual space
for any given $p\in(0,1)$, and to establish
some estimates of the div-curl product of elements in the local
Hardy space $h^{p}(\mathbb R^n)$ with $p\in(0,1]$ and its dual space.
Since the Riesz transforms are not bounded on $h^p(\rn)$ with $p\in(0,1]$,
the approach used in \cite{BGK12,bckly} to obtain these estimates of
div-curl products of elements in $H^p(\rn)$ and its dual space is no longer feasible for
$h^p(\rn)$ and its dual space. In this article, a main motivation for us to establish bilinear decompositions
of product spaces $h^p(\rn)\times \Lambda_{n\az}(\rn)$,
$H^p(\rn)\times \Lambda_{n\az}(\rn)$ and $H^1(\rn)\times \bmo(\rn)$
is to overcome this obstacle.

Let $p\in(0,1)$ and $\alpha:=1/p-1$.
In this article, we consider the local
Orlicz Hardy space $h^{\Phi_p}(\rn)$ and the
Orlicz Hardy space $H^{\Phi_p}(\rn)$ associated with the
Orlicz function
\begin{align}\label{333}
\Phi_{p}(\tau):=
\tau/(1+\tau^{1-p}),\ \ \forall\,\tau\in[0,\infty).
\end{align}
One of the main targets of this article is to establish the bilinear decompositions
$h^p(\rn)\times \Lambda_{n\az}(\rn)\subset L^1(\rn)+
h^{\Phi_p}(\rn)$ (see Theorem \ref{mainthm1} below),
$H^p(\rn)\times \Lambda_{n\az}(\rn)\subset L^1(\rn)+
H^{\Phi_p}(\rn)$ (see Theorem \ref{mainthm4} below)
and
$H^1(\rn)\times \bmo(\rn)\subset L^1(\rn)+
H_\ast^{\Phi}(\rn)$ (see Theorem \ref{mainthm7} below)
and, moreover, to prove that these bilinear decompositions are
sharp in some sense [see Remarks \ref{rem-add}, \ref{432}(iii) and \ref{442}(iii) below].
Borrowing some ideas from \cite{bckly,BGK12}, the main strategy for us to establish these
bilinear decompositions
is based on a technique of the renormalization of products of functions
(or distributions) via wavelets introduced by Coifman et al.  \cite{CDM95,Do95}, which
enables us to write the considered product into a sum of four bilinear operators.
However, since the approach to define these four bilinear operators
used in  \cite{bckly} is no longer feasible
for the local Hardy space $h^p(\rn)$, we employ a modified
and hence different method to introduce four different bilinear
operators to overcome this obstacle [see \eqref{eq-pi1} through \eqref{eq-pi4} below].
Then the problem can be reduced to the study of each bilinear operator $\Pi_i(f,g)$ for any $i\in\{1,\,2,\,3,\,4\}$.
The main difficulty is to deal with $\Pi_2$.
We do this by borrowing some ideas from \cite{FYL} and \cite{bckly},
and, therefore, reducing the estimation of the bilinear operator $\Pi_2(a,\,g)$ to that of $a P_B^sg$,
where $a$ denotes a (local) $(p,2,d)$-atom with integer
$d\geq2\lfloor n\az\rfloor$, and $P_B^sg$ the minimizing polynomial
of $g$ on the ball $B$ with degree $\le s$.
Here and thereafter, the \emph{symbol $\lfloor \beta\rfloor$}
for any $\beta\in\rr$ denotes the largest integer not greater than $\beta$.
As $P_B^sg$ is a polynomial, the term $a P_B^sg$ can still enjoy the high order vanishing
moment condition by requiring the wavelets
$\psi_I^\lambda$ to have the
sufficiently high order vanishing moment condition (see Proposition \ref{ap}, \ref{apW} and \ref{apW11} below).

Recall that, in \cite{bckly}, the following
\emph{bilinear} decomposition was proved:
for any given $p\in(0,1)$,
$$H^p(\rn)\times \mathfrak C_{\az}(\rn)\subset L^1(\rn)+H^{\phi_p}(\rn),$$
where $\mathfrak{C}_{\alpha}(\mathbb{R}^n)$ denotes
the Campanato space on $\mathbb{R}^n$ and $H^{\phi_p}(\rn)$ the Musielak--Orlicz Hardy
space associated with
the Musielak--Orlicz function $\phi_p$ defined
by setting, for any $x\in \mathbb{R}^n$ and $t\in [0,\,\infty)$,
\begin{align}\label{356}
\phi_{p}(x,\,t):=
\begin{cases}
\dfrac{t}{1+[t(1+|x|)^n]^{1-p}}& \qquad \textup{when}\ n(1/p-1)\notin \mathbb N,\vspace{0.2cm}\\
\dfrac{t}{1+[t(1+|x|)^n]^{1-p}[\log(e+|x|)]^p}& \qquad \textup{when}\ n(1/p-1)\in \mathbb N.
\end{cases}
\end{align}
Compared with
the results in \cite{bckly,cky1}, the Orlicz functions $\Phi_p$ and $\Phi$
in the target spaces $h^{\Phi_p}(\rn)$, $H^{\Phi_p}(\rn)$ and $H_\ast^{\Phi}(\rn)$
with $p\in(0,1)$ of the bilinear decompositions obtained
in this article look much simpler than the corresponding $\phi_p$ appearing in $H^{\phi_p}(\rn)$ of \cite{bckly},
or $\theta$ appearing in $H^{\rm log}(\rn)$ of \cite{cky1}. This is
because the function in $\Lambda_{n\az}(\rn)$ is \emph{bounded} but
the function in $\mathfrak{C}_{\alpha}(\mathbb{R}^n)$
and $\bmo(\rn)$ is not necessarily bounded. This reveals
the essential difference existing between the homogeneous
(or called the global) Hardy space and its dual space,
as well as between the inhomogeneous
(or called the local) Hardy space and its dual space.

Another contribution of this article is to characterize
the pointwise multiplier on the inhomogeneous Lipschitz space $\Lambda_{n\alpha}(\mathbb{R}^n)$,
with $\az\in(0,\infty)$,
by means of the dual space of $h^{\Phi_p}(\rn)$ (see Theorem \ref{thm-main2} below)
and the dual space of $H^{\Phi_p}(\rn)$ [see Theorem \ref{mainthm4}(ii) below].
Recall that the pointwise multiplier on $\BMO(\rn)$ was characterized
by Nakai and Yabuta \cite{NY85} (see also the recent survey \cite{n17} of Nakai).
We mention here that the class of pointwise multiplies of
$\Lambda_{n\alpha}(\mathbb{R}^n)$ for any $\az\in(0,\infty)$
was already characterized (see, for instance, \cite[(2.115)]{gg}),
which, however, is not connected with the dual space of $h^{\Phi_p}(\rn)$
or the dual space of $H^{\Phi_p}(\rn)$.
The obtained  pointwise multiplier conclusions on $\Lambda_{n\alpha}(\mathbb{R}^n)$ in this article imply that
the corresponding bilinear decomposition obtained in this article
is sharp in some sense [see Remarks \ref{rem-add} and \ref{432}(iii) below].
As an application of the target space $h^{\Phi_p}(\rn)$ obtained in the bilinear decomposition,
applying the bilinear decomposition
$H^p(\rn)\times \Lambda_{n\az}(\rn)\subset L^1(\rn)+
H^{\Phi_p}(\rn)$ (see Theorem \ref{mainthm4} below), we
obtain, in Theorem \ref{div} below, some estimates of products
of elements in the local Hardy space $h^p(\mathbb R^n)$ with $p\in(0,1)$
and its dual space, respectively, with zero $\lfloor n\alpha\rfloor$-inhomogeneous curl
(see Definition \ref{a2.15} below) and
zero divergence. Observe that the notion of the
zero $\lfloor n\alpha\rfloor$-inhomogeneous curl,
which is introduced in this article, perfectly
matches the locality (also called the inhomogeneity)
of both the local Hardy space $h^p(\mathbb R^n)$
and its dual space $\Lambda_{n\alpha}(\mathbb{R}^n)$.
Using the atomic characterization of the local Hardy space $h^1(\rn)$ and
the bilinear decomposition $H^1(\rn)\times \bmo(\rn)\subset L^1(\rn)+
H_\ast^{\Phi}(\rn)$ (see Theorem \ref{mainthm7} below),
we seal the gap appeared in \cite{cky1} of the estimate in $h_\ast^{\Phi}(\mathbb R^n)$
of div-curl products of functions in $h^1(\rn)$ and $\bmo(\mathbb{R}^n)$, and obtain the same estimate
under a weaker assumption on the curl (see Theorem \ref{sssa} below).
In addition, we find new structures of $h^{\Phi_p}(\mathbb R^n)$ and $H^{\Phi_p}(\mathbb R^n)$ 
by showing that, for any $p\in(0,1)$,
$h^{\Phi_p}(\rn)=h^1(\rn)+h^p(\rn)$ and $H^{\Phi_p}(\rn)=H^1(\rn)+H^p(\rn)$ with equivalent quasi-norms
(see Theorems \ref{FINN} and \ref{FIN2} below).
We also prove that the dual space of $h^{\Phi_p}(\mathbb R^n)$ coincides with the
dual space of $h^p(\mathbb R^n)$ (see Corollary \ref{11p} below).
The relationships $h^{\Phi_p}(\mathbb R^n)\subsetneqq h^{\phi_p}(\mathbb R^n)$
is also clarified (see Theorem \ref{cheng} below). These results,
together with the results in \cite{cky1,zyy}, give a complete story on the bilinear
decomposition of the multiplication between the local Hardy space and its dual space.

To be precise, this article is organized as follows.

In Section \ref{s2}, we first recall some notions concerning
the local Hardy space $h^p(\rn)$, the local Orlicz Hardy space
$h^{\Phi_p}(\rn)$ and their dual spaces with $p\in(0,1)$, where $\Phi_p$
is as in \eqref{333}. Then we prove that $aP_B^sg$, with $a$ being a
local $(p,2,d)$-atom, $s:=\lfloor n(1/p-1)\rfloor$ and
integer $d\geq2s$, is an element of $h^{\Phi_p}(\rn)$ (see Proposition \ref{ap} below),
which plays a key role in the proof of the bilinear decomposition of
products of elements in the local Hardy space and its dual space (see Theorem \ref{mainthm1} below).

In Section \ref{s3}, we characterize
the pointwise multipliers of $\Lambda_{n\alpha}(\mathbb{R}^n)$
for any $\az\in(0,\infty)$ (see Theorem \ref{thm-main2} below).
As a consequence of Theorem \ref{thm-main2}, we show that elements in $\cs(\rn)$
belong to the class of pointwise multipliers of $\Lambda_{n\alpha}(\mathbb{R}^n)$,
which guarantees that the product space $h^p(\rn)\times \Lambda_{n\alpha}(\mathbb{R}^n)$ is well defined.

In Section \ref{s4},
we first recall some basic
notions of the multiresolution analysis and
the renormalization of the products in $L^2(\rn)\times L^2(\rn)$
from \cite{CDM95,Do95}.
Based on the boundedness results
of the aforementioned four bilinear operators $\{\Pi_i\}_{i=1}^4$ [see \eqref{eq-pi1} through
\eqref{eq-pi4} below], we establish a bilinear decomposition for multiplications of elements in
$h^p(\mathbb R^n)$ and its dual space $\Lambda_{n\alpha}(\mathbb{R}^n)$
(see Theorem \ref{mainthm1} below).
We also establish the bilinear decompositions of the
product space $H^p(\rn)\times \Lambda_{n\alpha}(\rn)$ and the
product space $H^1(\rn)\times \bmo(\rn)$, where $\az:=1/p-1$, which are later used
in this article to achieve some
estimates of div-curl products of elements in $h^p(\rn)$ with $p\in(0,1]$ and its dual space
(see Theorems \ref{div} and \ref{sssa} below).
Moreover, we show that these bilinear decompositions
are sharp in some sense [see Remarks \ref{rem-add}, \ref{432}(iii) and \ref{442}(iii) below].

In Section \ref{s5},
we find new structures
of the spaces $h^{\Phi_p}(\mathbb R^n)$ and $H^{\Phi_p}(\mathbb R^n)$ by showing that, for any $p\in(0,1)$,
$h^{\Phi_p}(\rn)=h^1(\rn)+h^p(\rn)$ and $H^{\Phi_p}(\rn)=H^1(\rn)+H^p(\rn)$
with equivalent quasi-norms (see Theorems \ref{FINN} and \ref{FIN2} below).
Then we prove that the dual space of $h^{\Phi_p}(\mathbb R^n)$ coincides with the
dual space of $h^p(\mathbb R^n)$ (see Corollary \ref{11p} below).
The relationships $h^{\Phi_p}(\mathbb R^n)\subsetneqq h^{\phi_p}(\mathbb R^n)$
is also clarified in this section (see Theorem \ref{cheng} below).

In Section \ref{s6}, to match the inhomogeneity of the local Hardy
space and its dual space, we first introduce the notion of the
zero $\az$-inhomogeneous curl with $\az\in\zz_+$ (see Definition \ref{a2.15}
below). Applying the target space $h^{\Phi_p}(\mathbb R^n)$ and
the bilinear decomposition $H^p(\rn)\times \Lambda_{n\az}(\rn)\subset L^1(\rn)+
H^{\Phi_p}(\rn)$ (see Theorem \ref{mainthm4} below), we obtain
an estimate of div-curl products of elements in the local Hardy space $h^p(\rn)$ with $p\in(0,1)$
and its dual space (see Theorem \ref{div} below).
Using the atomic characterization of the local Hardy space $h^1(\rn)$ and
a bilinear decomposition of the product space $H^1(\rn)\times\bmo(\rn)$, in Theorem \ref{sssa} below,
we seal the gap existed in the estimation related to $h_\ast^{\Phi}(\mathbb R^n)$
of div-curl products of functions in $h^1(\rn)$ and $\bmo(\mathbb{R}^n)$
in \cite{cky1} with a weaker assumption on the curl.

Finally, we make some conventions on notation.
For any $x\in\rn$ and $r\in(0,\infty)$, let $B(x,r):=\{y\in\rn:|x-y|<r\}$.
For any $r\in(0,\infty)$, $f\in L^1_\loc(\rn)$ and $x\in\rn$, let
$$
\dashint_{B(x,\,r)}f(y)\,dy:=\frac1{|B(x,r)|}\int_{B(x,r)}f(y)\,dy,
$$
here and thereafter, $L^1_\loc(\rn)$ denotes the set of all locally 
integrable functions on $\rn$.
For any set $E$, we use $\mathbf1_{E}$ to denote its \emph{characteristic function}.
We also use $\vec{0}_n$ to denote the \emph{origin}
of $\rn$. Let $\mathcal{S}(\rn)$ denote the collection of all
\emph{Schwartz functions} on $\rn$, equipped
with the classical well-known topology
determined by a sequence of norms, and $\mathcal{S}'(\rn)$ its \emph{topological dual}, namely,
the set of all bounded linear functionals on $\mathcal{S}(\rn)$
equipped with the weak-$\ast$ topology.
Let $\mathbb{N}:=\{1,\,2,...\}$ and $\mathbb{Z}_+:=\mathbb{N}\bigcup\{0\}$.
Throughout this article, all cubes have their edges parallel to the coordinate axes.
We always denote by $Q$ for a cube of $\rn$, which is not necessary to be closed or open,
by $x_Q$ its \emph{center} and by $\ell(Q)$ its \emph{side length}.
For any cube $Q:=Q(x_Q,\ell(Q))\in\mathcal{Q}$,
with $x_Q\in\rn$ and $\ell(Q)\in(0,\infty)$,
and $\alpha\in(0,\infty)$, let $\alpha Q:=Q(x_Q,\alpha \ell(Q))$.
For any $\varphi\in\mathcal{S}(\rn)$ and $t\in(0,\infty)$,
let $\varphi_t(\cdot):=t^{-n}\varphi(t^{-1}\cdot)$. For any $s\in\mathbb{R}$,
we denote by $\lfloor s\rfloor$ the \emph{largest integer not greater than} $s$.
We always use $C$ to denote a \emph{positive constant}, which is independent of main parameters,
but it may vary from line to line.
Moreover, we use $C_{(\gamma,\ \beta,\ \ldots)}$ to denote a positive constant depending on the indicated
parameters $\gamma,\ \beta,\ \ldots$. If, for any real functions $f$ and $g$, $f\leq Cg$, we then write
$f\lesssim g$ and, if $f\lesssim g\lesssim f$, we then write $f\sim g$. For any $\alpha:=(\alpha_1,
\ldots, \alpha_n)\in\zz_+^n$ and $x:=(x_1,\ldots,x_n)\in\rn$, define $|\alpha|:=\alpha_1+ \cdots+\alpha_n$, $\partial^\alpha:=\partial_{x_1}^{\alpha_1}\cdots\partial_{x_n}^{\alpha_n}$ with
$\partial_{x_j}:=\frac{\partial}{\partial x_j}$ for any $j\in\{1,\ldots,n\}$, and
$x^\alpha:=x_1^{\alpha_1}\cdots x_n^{\alpha_n}$.

\section{Local Hardy-type spaces and their dual spaces}\label{s2}

In this section, we present some notions and basic properties on the local
Hardy space $h^p(\rn)$, the inhomogeneous Lipschitz space  $\Lambda_{n\alpha}(\mathbb{R}^n)$,
the local Orlicz Hardy space $h^{\Phi_p}(\rn)$
and its dual space $\mathcal{L}_{\loc}^{\Phi_p}(\rn)$.
The dual results of these function spaces come
from \cite{g, yy2}.

\begin{definition}\label{hp}
\begin{itemize}
\item[(i)]
Let $\varphi\in\mathcal{S}(\rn)$ and $f\in\cs'(\rn)$.
The \emph{local radial maximal function $m(f,\varphi)$} of $f$ associated to $\varphi$
is defined by setting, for any $x\in\rn$,
$$
m(f,\varphi)(x):=\sup_{s\in(0,1)}|f\ast\varphi_s(x)|,
$$
where, for any $s\in(0,\infty)$ and $x\in\mathbb R^n,\varphi_s(x):=s^{-n}\varphi(x/s)$.
\item[(ii)]
Let $p\in(0,\infty)$.
Then the \emph{local Hardy space $h^{p}(\rn)$}
is defined by setting
$$
h^{p}(\rn):=\lf\{f\in\cs'(\rn):\ \|f\|_{h^{p}(\rn)}:=\lf\|m(f,\varphi)\r\|_{L^p(\rn)}<\infty\r\},
$$
where $\varphi\in\mathcal{S}(\rn)$ satisfies
$
\int_{\rn}\varphi(x)\,dx\neq0.
$
\end{itemize}
\end{definition}

We refer the reader to \cite{g} for more properties on  $h^p(\rn)$.
The dual space of $h^{p}(\rn)$ turns out to be the
local Campanato space
which was first introduced by Campanato in \cite{Ca63,Ca64}.
Recall that,
for any $d\in\zz_+$, the
\emph{symbol $\cp_d(\rn)$} denotes the set of all polynomials with order at most $d$. In what follows,
for any $g\in L_{\loc}^1(\rn)$ and ball $B\subset\rn$,
we always use
$P_B^dg$ to denote the \emph{minimizing polynomial}
of $g$ on the ball $B$ with degree not greater than
$d$, which means $P_B^dg$ is the unique polynomial $P\in\cp_d(\rn)$
such that, for any polynomial $R\in\cp_d(\rn)$, $\int_B[g(x)-P(x)]R(x)\,dx=0$
(see \cite[Definition 6.1]{ns}).
It is well known that, if $f$ is locally integrable, then $P_B^df$ uniquely exists (see, for instance,
\cite{tw}).

\begin{definition}\label{campa}
Let $\alpha\in[0,\infty)$, $r\in[1,\infty)$ and $d\in\mathbb{Z}_+$ satisfy
$d\ge\lfloor n\alpha\rfloor$.
The \emph{local Campanato space} $\mathcal{L}_{\loc}^{\alpha,r,d}(\rn)$
is defined to be the set of all locally integrable functions $g$
such that
\begin{align*}
\|g\|_{\mathcal{L}_{\loc}^{\alpha,r,d}(\rn)}&:=
\sup_{\mathrm{ball\,}B\subset\rn,|B|<1}
\frac{1}{|B|^\alpha}\lf[\dashint_{B}|g(x)-P_B^d g(x)|^r\,dx\r]^{\frac{1}{r}}\\
&\quad+\sup_{\mathrm{ball\,}B\subset\rn,|B|\geq1}
\frac{1}{|B|^\alpha}\lf[\dashint_{B}|g(x)|^r\,dx\r]^{\frac{1}{r}}
<\infty,
\end{align*}
where $P_B^dg$ for any ball $B\subset\rn$ denotes the minimizing polynomial of $g$ on $B$ with degree
not greater than $d$.

When  $d:= \lfloor n\az\rfloor$ and $r:=1$, we denote $\mathcal{L}_{\loc}^{\az,r,d}(\rn)$
simply by $\mathcal{L}_{\loc}^{\az}(\rn)$.
\end{definition}

Also, let $\alpha\in[0,\infty)$, $r\in[1,\infty)$ and $d\in\mathbb{Z}_+$ satisfy
$d\ge\lfloor n\alpha\rfloor$. An equivalent norm of the local Campanato space is as follows (see, for instance, \cite[p.\,292]{GR85}):
For any $g\in\mathcal{L}_{\loc}^{\alpha,r,d}(\rn)$,
\begin{align}\label{f-encs}
\|g\|_{\mathcal{L}_{\loc}^{\alpha,r,d}(\rn)}&\sim
\sup_{\mathrm{ball\,}B\subset \rn,|B|<1}\inf_{P\in \mathcal{P}_d(\rn)}
\frac{1}{|B|^\az}\lf\{\dashint_B|g(x)-P(x)|^r\,dx\r\}^{\frac{1}{r}}\\ \noz
&\quad
+\sup_{\mathrm{ball\,}B\subset\rn,|B|\geq1}
\frac{1}{|B|^\alpha}\lf[\dashint_{B}|g(x)|^r\,dx\r]^{\frac{1}{r}},
\end{align}
where the
positive equivalence constants are independent of $g$.

We give several remarks on the relations between local Campanato spaces and some
other related function spaces.

\begin{remark}\label{r11}
\begin{enumerate}
\item[(i)] The space $\mathcal{L}_{\loc}^{0,1,0}(\rn)$
is just the space $\bmo(\rn)$ in \cite{g}, where the \emph{space $\bmo(\rn)$}
is defined to be the set of all measurable functions $f$ such that
\begin{equation*}
\|f\|_{\bmo(\rn)}:=\sup_{\mathrm{ball\,}B\subset\rn,|B|<1}
\frac{1}{|B|}\int_{B}|f(x)-f_B|\,dx
+\sup_{\mathrm{ball\,}B\subset\rn,|B|\geq1}\frac{1}{|B|}\int_{B}|f(x)|\,dx<\infty
\end{equation*}
with $f_B:=\frac{1}{|B|}
\int_{B}f(x)\,dx$.
\item[(ii)] When $p\in(0,1]$, $\alpha:=1/p-1$ and $d\in\mathbb{Z}_+$ satisfies
$d\ge\lfloor n\alpha\rfloor$, from \cite[Theorem 5]{g}
(see also \cite[Theorem 7.5]{yy}), we deduce
that
$$
(h^p(\rn))^*=\mathcal{L}_{\loc}^{\alpha,r,d}(\rn).
$$
\end{enumerate}
\end{remark}

Next, we recall the notion of inhomogeneous Lipschitz spaces
(see for instance \cite[Definition 1.4.2]{g14}).
For any $h\in \rn$, we define the \emph{difference operator} $D_h$
by setting, for any function $f:\ \rn \to \mathbb{C}$ and $x\in\rn$,
$$
D_h(f)(x) := f(x+h)-f(x).
$$
In general, for any $k\in\nn$, the $(k+1)$-order difference operator $D^{k+1}_h$
is defined by setting, for any continuous function $f$,
$$
D^{k+1}_h (f): = D^k_h(D_h(f)).
$$

\begin{definition}\label{def-Zygmund}
Let $\alpha\in(0, \infty)$.
Then the \emph{inhomogeneous Lipschitz space}
$\Lambda_\az(\rn)$ is defined to be the set of all measurable functions $f$ such that
there exists a continuous function $g$ satisfying that $f=g$ almost
everywhere and
\begin{align*}
\|f\|_{\Lambda_\az(\rn)}:=\|f\|_{L^\infty(\rn)}+ \sup_{x\in\rn}\sup_{h\in\rn\backslash\{\vec{0_n}\}}
\frac{|D^{\lfloor\alpha\rfloor+1}_h (g)(x)|}{|h|^{\alpha}}<\infty.
\end{align*}
\end{definition}

The following identification of local Campanato spaces and
inhomogeneous Lipschitz spaces can be found in \cite[Theorem 5]{g}
or \cite{Ca63,Ca64}.

\begin{lemma}\label{GR-DDY}
Let $\az\in(0,\infty)$ and $r\in[1,\infty)$.
Then $f\in\Lambda_{n\az}(\rn)$ if and only if
$f\in\mathcal{L}_{\loc}^{\alpha,r,\lfloor n\az\rfloor}(\rn)$. Moreover,
$$\|f\|_{\Lambda_{n\alpha}(\rn)}\sim \|f\|_{\mathcal{L}_{\loc}^{\alpha,r,\lfloor n\az\rfloor}(\rn)}$$
with the
positive equivalence constants independent of $f$.
\end{lemma}

Now we recall the notions of both Orlicz functions and Orlicz spaces (see, for instance, \cite{mmz}).

\begin{definition}\label{d1.1}
A function $\Phi:\ [0,\infty)\ \to\ [0,\infty)$ is called an \emph{Orlicz function} if it is
non-decreasing and satisfies $\Phi(0)= 0$, $\Phi(\tau)>0$ whenever $\tau\in(0,\infty)$ and $\lim_{\tau\to\infty}\Phi(\tau)=\infty$.
\end{definition}

An Orlicz function $\Phi$  is said to be
of \emph{lower} (resp., \emph{upper}) \emph{type} $p$ with
$p\in(-\infty,\infty)$ if
there exists a positive constant $C_{(p)}$, depending on $p$, such that, for any $\tau\in[0,\infty)$
and $s\in(0,1)$ [resp., $s\in [1,\infty)$],
\begin{equation*}
\Phi(s\tau)\le C_{(p)}s^p \Phi(\tau).
\end{equation*}
A function $\Phi:\ [0,\infty)\ \to\ [0,\infty)$ is said to be of
\emph{positive lower} (resp., \emph{upper}) \emph{type} $p$ if it is of lower
(resp., upper) type $p$ for some $p\in(0,\infty)$.

\begin{definition}\label{d1.2}
Let $\Phi$ be an Orlicz function with positive lower type $p_{\Phi}^-$ and positive upper type $p_{\Phi}^+$.
The \emph{Orlicz space $L^\Phi(\rn)$} is defined
to be the set of all measurable functions $f$ such that
 $$\|f\|_{L^\Phi(\rn)}:=\inf\lf\{\lambda\in(0,\infty):\ \int_{\rn}\Phi\lf(\frac{|f(x)|}{\lambda}\r)\,dx\le1\r\}<\infty.$$
\end{definition}

Now we recall the notions of local Orlicz Hardy spaces
and local Orlicz Campanato spaces  (see, for instance, \cite{yy,yy2}).

\begin{definition}\label{defn-hmo-1}
Let $\Phi$ be an Orlicz function with positive lower type $p_{\Phi}^-$ and positive upper
type $p_{\Phi}^+\in(0,1]$.
The \emph{local Orlicz Hardy
space $h^\Phi(\rn)$} is defined to be the set of all $f\in\cs'(\rn)$
such that $\|f\|_{h^\Phi(\rn)}:=
\|m(f,\varphi)\|_{L^\Phi(\rn)}<\infty$,
where $\varphi\in\mathcal{S}(\rn)$ satisfies $\int_{\rn}\varphi(x)\,dx\neq0$ and $m(f,\varphi)$ is as
in Definition \ref{hp}(i).
\end{definition}

\begin{definition}\label{defn-mocs-1}
Let $d\in\zz_+$ and $\Phi$ be an Orlicz function with positive lower type $p_{\Phi}^-$ and positive upper
type $p_{\Phi}^+\in(0,1]$.
The {\em local Orlicz Campanato space}
$\mathcal{L}_{\loc}^{\Phi,1,d}(\rn)$ is defined to be the set
of all locally integrable functions $g$ on $\rn$ such that
\begin{align*}
\|g\|_{\mathcal{L}_{\loc}^{\Phi,1,d}(\rn)}:&=
\dsup_{\mathrm{ball\,}B\subset\rn,|B|<1} \frac1{\|\mathbf1_B\|_{L^{\Phi}(\rn)}} \int_B |g(x)-P_B^d g(x)|\,dx\\
&\quad+\dsup_{\mathrm{ball\,}B\subset\rn,|B|\geq1} \frac1{\|\mathbf1_B\|_{L^{\Phi}(\rn)}} \int_B |g(x)|\,dx
<\infty,
\end{align*}
where $P_B^dg$ for any ball $B\subset\rn$ denotes the minimizing polynomial of $g$ on $B$ with degree
not greater than $d$.

When $d= \lfloor n(1/p^{-}_{\Phi}-1)\rfloor$, we denote
$\mathcal{L}_{\loc}^{\Phi,1,d}(\rn)$ simply by $ \mathcal{L}_{\loc}^{\Phi}(\rn)$.
\end{definition}

The following duality result is just \cite[Theorem 7.5]{yy2}.

\begin{lemma}\label{lem-dualMHC}
Let $\Phi$ be an Orlicz function
with positive lower type $p_{\Phi}^-$ and positive upper type $p_{\Phi}^+\in(0,1]$.
Let $d\in\zz_+$ be such that $d\ge \lfloor n(\frac{1}{p_{\Phi}^-}-1)\rfloor$.
Then
$(h^{\Phi}(\rn))^*=\mathcal{L}_{\loc}^{\Phi,1,d}(\rn).$
\end{lemma}

\begin{lemma}\label{444}
Let $p\in(0,1)$ and $\Phi_{p}$ be as in \eqref{333}.
\begin{itemize}
\item[\rm (i)] For any $f\in L^{\Phi_p}(\rn)$,
\begin{equation}\label{555}
\|f\|_{L^{\Phi_p}(\rn)}\le \min \,\left\{\|f\|_{L^{1}(\rn)},\, \|f\|_{L^p(\rn)}\right\}
\end{equation}
and, for any ball $B\subset \rn$,
\begin{equation*}
\|\mathbf1_B\|_{L^{\Phi_p}(\rn)}\sim \min \,\left\{\|\mathbf1_B\|_{L^{1}(\rn)},\, \|\mathbf1_B\|_{L^p(\rn)}\right\}
\end{equation*}
with positive equivalence constants independent of $B$.

Moreover, $h^1(\rn)\subset h^{\Phi_p}(\rn)$, $h^p(\rn)\subset h^{\Phi_p}(\rn)$
and, for any $f\in h^1(\rn)\cup h^p(\rn)$,
\begin{equation*}
\|f\|_{h^{\Phi_p}(\rn)}\le \min \,\left\{\|f\|_{h^{1}(\rn)},\, \|f\|_{h^p(\rn)}\right\}.
\end{equation*}
\item[\rm (ii)]
For any $\tau\in[0,\infty)$,
\begin{align*}
\Phi_p(s\tau)\le
\begin{cases}
s^p \Phi_p(\tau)&\qquad \textup{when}\; s\in(0,1),
\vspace{0.2cm}\\
s \Phi_p(\tau)&\qquad \textup{when}\;
s\in [1,\infty)
\end{cases}
\end{align*}
and,
for any $\{t_j\}_{j\in\nn}\subset[0,\,\fz)$,
\begin{align}\label{xx4}
\Phi_p\lf(\sum_{j\in\nn}t_j\r)
\leq \sum_{i\in\nn} \Phi_p(t_i).
\end{align}
\end{itemize}
\end{lemma}
\begin{proof}
Let $p\in(0,1)$.
We first prove (i).
Observe that, for any $\tau\in[0,\infty)$,
$$
\Phi_{p}(\tau)=\dfrac{\tau}{1+\tau^{1-p}}\le\tau
\ \
\text{and}
\ \
\Phi_{p}(\tau)=\dfrac{\tau^p}{1+\tau^{p-1}}\le\tau^p.
$$
From this and Definition \ref{d1.2}, we deduce that, for any
$f\in h^1(\rn)\cup h^p(\rn)$,
$$
\|f\|_{L^{\Phi_p}(\rn)}\le \min \,\left\{\|f\|_{L^{1}(\rn)},\, \|f\|_{L^p(\rn)}\right\}.
$$
This finishes the proof of \eqref{555}.

On the other hand, it is easy to see that,
for any $\tau\in(0,1)$,
$$\tau/2\le\dfrac{\tau}{1+\tau^{1-p}}=\Phi_{p}(\tau)$$
and, for any $\tau\in[1,\infty)$,
$$\tau^p/2\le\dfrac{\tau^p}{1+\tau^{p-1}}=\Phi_{p}(\tau),$$
which, combined with Definition \ref{d1.2} and \eqref{555}, implies that
$$
\min \,\left\{\|\mathbf1_B\|_{L^{1}(\rn)},\, \|\mathbf1_B\|_{L^p(\rn)}\right\}
\sim\|\mathbf1_B\|_{L^{\Phi_p}(\rn)}.
$$

From \eqref{555}, Definitions \ref{defn-hmo-1} and \ref{hp}(ii),
it easily follows that $h^1(\rn)\subset h^{\Phi_p}(\rn)$ and $h^p(\rn)\subset h^{\Phi_p}(\rn)$. Moreover,
for any $f\in h^1(\rn)\cup h^p(\rn)$
\begin{equation*}
\|f\|_{h^{\Phi_p}(\rn)}\le \min \,\left\{\|f\|_{h^{1}(\rn)},\, \|f\|_{h^p(\rn)}\right\},
\end{equation*}
which completes the proof of (i).

Now we prove (ii).
By the definition of $\Phi_p$, we easily conclude that,
for any $\tau\in[0,\infty)$,
\begin{align*}
\Phi_p(s\tau)=
\begin{cases}
\dfrac{(s\tau)^p}{1+(s\tau)^{p-1}}\le
s^p \Phi_p(\tau)&\qquad \textup{when}\; s\in(0,1),
\vspace{0.2cm}\\
\dfrac{s\tau}{1+(s\tau)^{1-p}}\le s\Phi_p(\tau)&\qquad \textup{when}\;
s\in [1,\infty).
\end{cases}
\end{align*}
From this, we deduce
that,
for any sequence $\{t_j\}_{j\in\nn}\subset[0,\,\fz)$,
\begin{align*}
\Phi_p\lf(\sum_{j\in\nn}t_j\r)
=\dsum_{i\in\nn} \lf[\frac{t_i}{\sum_{j\in\nn}t_j}\Phi_p\lf(\sum_{j\in\nn}t_j \r)\r]
\leq \sum_{i\in\nn} \Phi_p(t_i).
\end{align*}
Thus, (ii) holds true. This finishes the proof Lemma \ref{444}.
\end{proof}

\begin{remark}\label{rem-add1}
Let $p\in(0,1)$ and $d\in\zz_+$ be such that $d\ge \lfloor n(1/p-1)\rfloor$.
Then $\Phi_p$
in \eqref{333} is an Orlicz function with positive upper type $p_{\Phi}^+=1$
and positive lower type $p_{\Phi}^-=p$.
From this, Definitions \ref{defn-hmo-1}
and \ref{defn-mocs-1}, it follows that
the local Orlicz Hardy space $h^{\Phi_p}(\rn)$
and the local Orlicz Campanato space $\mathcal{L}_{\loc}^{\Phi_p,1,d}(\rn)$
are well defined.
Moreover, by Lemma \ref{lem-dualMHC}, we know that the dual space of $h^{\Phi_p}(\rn)$ is $\mathcal{L}_{\loc}^{\Phi_p,1,d}(\rn)$.
\end{remark}

Now we recall the notion of local atoms (see, for instance, \cite{g,t}).

\begin{definition}\label{17-defa}
Let $p\in(0,\,1)$, $r\in(1,\infty]$ and $d\in\zz_+$.
Then a measurable function $a$ on $\rn$ is called a \emph{local} $(p,\,r,\,d)${\it-atom}
if there exists a ball $B\subset\rn$ such that
\begin{itemize}
\item[(i)] $\supp a:=\{x\in\rn:\ a(x)\neq0\}\subset B$;
\item[(ii)] $\|a\|_{L^r(\rn)}\le |B|^{1/r-1/p}$;
\item[(iii)] if $|B|<1$, then
$\int_{\rn}a(x)x^\alpha\,dx=0$ for any
$\alpha:=(\alpha_1,\ldots,\alpha_n)\in\zz_+^n$ with $|\alpha|\le d$,
here and thereafter, for any $x:=(x_1,\ldots,x_n)\in\rn$,
$x^\alpha:=x_1^{\alpha_1}\cdots x_n^{\alpha_n}$.
\end{itemize}
\end{definition}

The following lemma is just \cite[p.\,54,\, Lemma 4.1]{Lu95} (see also \cite[p.\,83]{tw}).

\begin{lemma}\label{lem-ep}
Let $d\in\zz_+$. Then there exists
a positive constant $C$ such that, for any $g\in L_{\loc}^{1}(\rn)$ and ball $B\subset\rn$,
$$
\sup_{x\in B}\lf|P_B^dg(x)\r|\le\frac{C}{|B|}\int_B|g(x)|\,dx,
$$
where $P_B^dg$ for any ball $B\subset\rn$ denotes the minimizing polynomial of $g$ on $B$ with degree
not greater than $d$.
\end{lemma}

Borrowing some ideas from the proof of \cite[Proposition 2.24]{bckly}, we have the following
technical lemma, which
plays a vital role in the proof of Theorem \ref{mainthm1} below.

\begin{proposition}\label{ap}
Let $p\in (0,\,1)$, $\Phi_{p}$ be as in \eqref{333},
$s:=\lfloor n(1/p-1)\rfloor$ and $d\in\zz_+\cap [2s,\,\fz)$.
Assume that $g\in L^\infty(\rn)$ and $a$ is a local $(p,\,2,\,d)$-atom as in Definition
\ref{17-defa}
supported in a ball $B\subset\rn$. Then
$$\|aP_B^s g\|_{h^{\Phi_p}(\rn)}\le C \|g\|_{L^\infty(\rn)},$$
where $P_B^sg$ for any ball $B\subset\rn$ denotes the minimizing polynomial of $g$ on $B$ with degree
not greater than $s$, and the positive constant $C$ is independent of $a$ and $g$.
\end{proposition}

\begin{proof}
Let $p\in (0,\,1)$, $s:=\lfloor n(1/p-1)\rfloor$
and $g\in L^\infty(\rn)$. Let $d\in\zz_+\cap [2s,\,\fz)$, $a$ be a local $(p,\,2,\,d)$-atom as in Definition
\ref{17-defa}
supported in a ball $B\subset\rn$.
Without loss of generality, we may assume that $\|g\|_{L^\infty(\rn)}=1$.
From this and Lemma \ref{lem-ep}, it follows that
\begin{equation}\label{2w}
\sup_{x\in B}|P_B^s g(x)|\lesssim\dashint_B |g(y)|\,dy\lesssim1.
\end{equation}
By the definition of $h^{\Phi_p}(\rn)$
and \eqref{333}, to show $\|aP_B^s g\|_{h^{\Phi_p}(\rn)}\lesssim1$,
it suffices to prove that there exists a positive constant $c$, independent of $a$
and $g$, such that
\begin{equation}\label{1w}
\int_{\rn}\Phi_p\lf(\frac{m(a P_B^s g,\varphi)(x)}{c}\r)\,dx=
\int_{\rn}\frac{m(a P_B^s g,\varphi)(x)/c}{1+[m(a P_B^s g,\varphi)(x)/c]^{1-p}}\,dx
\leq1,
\end{equation}
where
$m(\cdot,\varphi)$ is as in Definition \ref{defn-hmo-1}.
Moreover, from an elementary calculation, we deduce that \eqref{1w} holds true provided that
\begin{equation}\label{3w}
\int_{\rn}[m(a P_B^s g,\varphi)(x)]^p\,dx
\leq c^p.
\end{equation}
Next, we show \eqref{3w}.
By \eqref{2w} and the assumption that $a$ is a local $(p,2, d)$-atom supported in the ball
$B$, we
find that $\widetilde{a}:= \frac{1}{\widetilde{C}}a P_B^s g$ is a local $(p,2, d-s)$-atom for
some $\widetilde{C}\in(0,\infty)$.
From this, $d-s\ge\lfloor n(1/p-1)\rfloor$ and
the atomic characterization of $h^{p}(\rn)$ (see \cite[Lemma 5]{g} or \cite[Theorem 5.1]{t}),
it follows that
$$
\int_{\rn}[m(a P_B^s g,\varphi)(x)]^p\,dx=
\|a P_B^s g\|_{h^{p}(\rn)}\lesssim\|\widetilde{a}\|_{h^{p}(\rn)}\lesssim1,
$$
which implies that \eqref{3w} holds true and hence \eqref{1w} holds true.
This finishes the proof of Proposition \ref{ap}.
\end{proof}

\section{Pointwise multipliers of local Campanato spaces $\Lambda_{n(1/p-1)}(\mathbb{R}^n)$ 
with $p\in (0,1)$}\label{s3}

In this section, we
characterize pointwise multipliers of $\Lambda_{n\alpha}(\mathbb{R}^n)$
for any $\az\in(0,\infty)$ (see Theorem \ref{thm-main2} below).
As a consequence of Theorem \ref{thm-main2}, we show that elements in $\cs(\rn)$ belong
to the class of pointwise multipliers of $\Lambda_{n\alpha}(\mathbb{R}^n)$, which guarantees that the
product space $h^p(\rn)\times \Lambda_{n\alpha}(\mathbb{R}^n)$ is well defined.

Recall that, for any quasi-Banach space $X$ equipped with a
quasi-norm $\|\cdot\|_X$, a function $g$ defined on $\rn$ is called a \emph{pointwise multiplier} on $X$
if there exists a positive constant $C$ such that, for any $f\in X$, $\|gf\|_{X}\le C\|f\|_X$
(see, for instance, \cite{n17,ns}).

\begin{theorem} \label{thm-main2}
Let $p\in(0,\,1)$, $\az:=1/p-1$
and $\Phi_p$ be as in \eqref{333}. Let $\mathcal{L}_{\loc}^{\Phi_p}(\rn)$ be as in Definition
\ref{defn-mocs-1} with $\Phi$ replaced by $\Phi_p$.
Then the following assertions are equivalent:
\begin{itemize}
\item[{\rm(i)}]
$g\in L^\infty(\rn)\cap \mathcal{L}_{\loc}^{\Phi_p}(\rn)$;

\item[{\rm(ii)}] $g$ is a pointwise multiplier of
$\Lambda_{n\alpha}(\rn)$ and, for any  $f\in \Lambda_{n\alpha}(\rn)$,
$$\|gf\|_{\Lambda_{n\alpha}(\rn)}
\le C \lf[\|g\|_{L^\infty(\rn)}+
\|g\|_{\mathcal{L}_{\loc}^{\Phi_p}(\rn)}\r]
\|f\|_{\Lambda_{n\alpha}(\rn)},$$
where $C$ is a positive constant independent of $f$ and $g$.
\end{itemize}
\end{theorem}

\begin{proof}
Let $p\in(0,\,1)$, $\az:=1/p-1$ and $s:=\lfloor n(1/p-1)\rfloor$.
We first show (i)$\,\Rightarrow\,$(ii).
Let $g\in L^\infty(\rn)\cap \mathcal{L}_{\loc}^{\Phi_p}(\rn)$, $f\in \Lambda_{n\alpha}(\rn)$
and $B\subset\rn$ be a ball satisfying $|B|<1$.
Then we have
\begin{align}\label{2r}
\lf|gf-P_B^sfP_B^s g\r|
&\le \lf|g\r|\lf|f-P_B^sf\r|\,+ \lf|P_B^sf\r|\lf|g-P_B^sg\r|,
\end{align}
where $P_B^sg$ and $P_B^sf$ for any ball $B\subset\rn$ denote the minimizing polynomials of $g$ and $f$ on $B$ with degree
not greater than $s$, respectively.
From Lemma \ref{lem-ep}, it follows that
\begin{align}\label{1r}
\sup_{x\in B}|P_B^sf(x)|\ls
\dashint_B |f(y)|\,dy\ls \|f\|_{L^\infty(\rn)}.
\end{align}
Moreover, by \eqref{555} with $f:=\mathbf1_B$,
we conclude that
$\|\mathbf1_B\|_{L^{\Phi_p}(\rn)}
\le\|\mathbf1_B\|_{L^{p}(\rn)}.$
From this, \eqref{2r}, \eqref{1r}, Definition \ref{campa} and Lemma \ref{GR-DDY}, we deduce that
\begin{align}\label{r5}
&\dashint_B \lf|g(x)f(x)-P_B^sf(x)P_B^s g(x)\r|\, dx\\ \noz
&\quad\ls \|g\|_{L^\infty(\rn)}\dashint_B \lf|f(x)-P_B^sf(x)\r|\,dx\, +
 \|f\|_{L^\infty(\rn)} \dashint_B {|g(x)-P_B^sg(x)|\,dx}\\\noz
&\quad\ls |B|^\alpha \lf[\|g\|_{L^\infty(\rn)}
\|f\|_{\mathcal{L}_{\loc}^{\alpha}(\rn)}
+ \|f\|_{L^\infty(\rn)} \frac{1}{|B|^{1/p}}\int_B |g(x)-P_B^sg(x)|\,dx\r]\\\noz
&\quad\ls |B|^\alpha \lf[\|g\|_{L^\infty(\rn)}
\|f\|_{\mathcal{L}_{\loc}^{\alpha}(\rn)}
+ \|f\|_{L^\infty(\rn)}
\frac{1}{\|\mathbf1_B\|_{L^{\Phi_p}(\rn)}}\int_B |g(x)-P_B^sg(x)|\,dx\r]\\\noz
&\quad\ls |B|^\alpha
\lf[\|g\|_{L^\infty(\rn)}+\|g\|_{\mathcal{L}
_{\loc}^{\Phi_p}(\rn)}\r]
\|f\|_{\Lambda_{n\alpha}(\mathbb{R}^n)}.
\end{align}
By Lemma \ref{GR-DDY} again, we obtain,
for any ball $B\subset\rn$ with $|B|\geq1$,
\begin{align}\label{r4}
\dashint_B \lf|g(x)f(x)\r|\, dx&
\leq |B|^\alpha\frac{1}{|B|^{1/p}}\int_B \lf|f(x)\r|\,dx\, \|g\|_{L^\infty(\rn)}
\leq |B|^\alpha\|f\|_{\mathcal{L}_{\loc}^{\alpha}(\rn)}\|g\|_{L^\infty(\rn)}\\\noz
&\sim|B|^\alpha \|g\|_{L^\infty(\rn)}
\|f\|_{\Lambda_{n\alpha}(\mathbb{R}^n)}.
\end{align}
Observe that $P_B^sfP_B^s g\in\cp_{2s}(\rn)$.
From this, \eqref{f-encs}, \eqref{r5}, \eqref{r4} and Definition \ref{campa},
it follows that
$gf\in\mathcal{L}_{\loc}^{\alpha,1,2s}(\rn)$ and
$$
\|gf\|_{\mathcal{L}_{\loc}^{\alpha,1,2s}(\rn)}
\ls
[\|g\|_{L^\infty(\rn)}+\|g\|_{\Lambda_{n\alpha}(\mathbb{R}^n)}]
\|f\|_{\mathcal{L}_{\loc}^{\alpha}(\rn)},$$
which, together with Lemma \ref{GR-DDY}, implies that
\begin{align*}
\|gf\|_{\Lambda_{n\alpha}(\mathbb{R}^n)}
\ls[\|g\|_{L^\infty(\rn)}+\|g\|_{\Lambda_{n\alpha}(\mathbb{R}^n)}]
\|f\|_{\mathcal{L}_{\loc}^{\alpha}(\rn)}.
\end{align*}
This finishes the proof that (i)$\,\Rightarrow\,$(ii).

As for (ii)$\,\Rightarrow\,$(i), since the pointwise
multiplier space of $\Lambda_{n\alpha}(\mathbb{R}^n)$ is itself
(see, for instance, \cite[(2.115)]{gg}), we only need to prove that
$\Lambda_{n\alpha}(\mathbb{R}^n)\subset L^\infty(\rn)\cap \mathcal{L}_{\loc}^{\Phi_p}(\rn)$
and $\|\cdot\|_{L^{\infty}(\rn)}
+\|\cdot\|_{\mathcal{L}_{\loc}^{\Phi_p}(\rn)}
\lesssim \|\cdot\|_{\Lambda_{n\alpha}(\mathbb{R}^n)}$.
Let $g\in\Lambda_{n\alpha}(\mathbb{R}^n)$.
It is easy to know that
$\|g\|_{L^{\infty}(\rn)}
\le\|g\|_{\Lambda_{n\alpha}(\mathbb{R}^n)}$. Next, we show that $\|g\|_{\mathcal{L}_{\loc}^{\Phi_p}(\rn)}
\lesssim\|g\|_{\Lambda_{n\alpha}(\mathbb{R}^n)}$.

From Lemmas \ref{GR-DDY} and \ref{444}(i), and Definition \ref{campa},
we deduce that, for any ball $B\subset\rn$ satisfying $|B|<1$,
\begin{align}\label{rrr}
\frac{1}{\|\mathbf1_B\|_{L^{\Phi_p}(\rn)}}\int_B |g(x)-P_B^sg(x)|\,dx
&\lesssim\frac{1}{\|\mathbf1_B\|_{L^{p}(\rn)}}\int_B |g(x)-P_B^sg(x)|\,dx\\ \noz
&\lesssim\|g\|_{\mathcal{L}_{\loc}^{\az,1,s}(\rn)}\sim\|g\|_{\Lambda_{n\alpha}(\mathbb{R}^n)}.
\end{align}
By Lemma \ref{444}(i) again, we conclude that, for any ball $B\subset\rn$ satisfying $|B|\geq1$,
\begin{align}\label{r9}
\frac{1}{\|\mathbf1_B\|_{L^{\Phi_p}(\rn)}}\int_B |g(x)|\,dx
\lesssim \frac{1}{|B|}\int_B \lf|g(x)\r|\,dx\lesssim
\|g\|_{L^\infty(\rn)}\lesssim\|g\|_{\Lambda_{n\alpha}(\mathbb{R}^n)}.
\end{align}
Combining \eqref{rrr} and \eqref{r9}, we obtain
$\|g\|_{\mathcal{L}_{\loc}^{\Phi_p}(\rn)}
\lesssim \|g\|_{\Lambda_{n\alpha}(\mathbb{R}^n)}$, which completes
the proof that (ii)$\,\Rightarrow\,$(i) and hence of Theorem \ref{thm-main2}.
\end{proof}

\begin{lemma}\label{schwartz}
Let $p\in(0,\,1)$ and $\Phi_{p}$ be as in \eqref{333}. Then $ \mathcal{S}(\rn)$ embeds continuously into $\mathcal{L}_{\loc}^{\Phi_p}(\rn)$.
\end{lemma}

\begin{proof}
Let $p\in(0,\,1)$, $\az:=1/p-1$ and $g\in \mathcal{S}(\rn)$.
By the fact that $g$ is bounded, and Lemma \ref{444}(i),
we conclude that, for any ball $B\subset\rn$ satisfying $|B|\geq1$,
\begin{align}\label{r2}
\frac{1}{\|\mathbf1_B\|_{L^{\Phi_p}(\rn)}}\int_B |g(y)|\,dy
\lesssim \frac{1}{|B|}\int_B \lf|g(y)\r|\,dy\lesssim
\sup_{y\in\rn}\lf|g(y)\r|.
\end{align}
Let $s:=\lfloor n\az\rfloor$ and $B:=B(c_B,\,r_B)\subset\rn$ with $c_B\in \rn$ and $r_B\in(0,\,1)$.
For any $y\in\rn$, let
$$
P(y):=\sum_{|\beta|\le s}
\frac{\partial^\beta g(c_B)}{\beta!}\lf(y-c_B\r)^\beta,
$$
here and thereafter, for any $\beta:=(\beta_1,\ldots,\beta_n)\in\zz_+^n$
and $x\in\rn$,
$|\beta|:=\beta_1+\cdots+\beta_n$ and
$\partial^\beta:=(\frac{\partial}{\partial x_1})^{\beta_1}
\cdots(\frac{\partial}{\partial x_n})^{\beta_n}$.
From the Taylor remainder theorem, we deduce that
\begin{align*}
\sup_{y\in B}\lf|g(y)-P(y)\r|\ls\sup_{y\in B}\sum_{\beta\in\zz_+^n,|\beta|=s+1}
\lf|\frac{\partial^\beta g(\xi(y))}{\beta!}(y-c_B)^\beta\r|
\ls\sum_{\beta\in\zz_+^n,|\beta|=s+1}
\sup_{y\in\rn}\lf|\partial^\beta g(y)\r|r_B^{s+1},
\end{align*}
where $\xi(y):=y+\theta(c_B-y)$ for some $\theta\in[0,1]$.
By this and
Lemma \ref{444}(i), we conclude that
\begin{align*}
\frac1{\|\mathbf1_B\|_{L^{\Phi_p}(\rn)}}
\int_B \lf|g(y)-P(y)\r|\,dy&\ls\sum_{\beta\in\zz_+^n,|\beta|=s+1}
\sup_{y\in\rn}\lf|\partial^\beta g(y)\r|\frac{r_B^{s+1}}{|B|^{1/p-1}}\\
&\ls\sum_{\beta\in\zz_+^n,|\beta|=s+1}
\sup_{y\in\rn}\lf|\partial^\beta g(y)\r|.
\end{align*}
Combining this and \eqref{r2}, we obtain
\begin{align}\label{h3}
\|g\|_{\mathcal{L}_{\loc}^{\Phi_p}(\rn)}
\lesssim \sum_{\beta\in\zz_+^n,|\beta|=s+1}
\sup_{y\in\rn}\lf|\partial^\beta g(y)\r|+\sup_{y\in\rn}\lf|g(y)\r|,
\end{align}
which implies that $g\in\mathcal{L}_{\loc}^{\Phi_p}(\rn)$ and hence
$\mathcal{S}(\rn)\subset\mathcal{L}_{\loc}^{\Phi_p}(\rn)$.
Moreover, from \eqref{h3},
it easily follows that the embedding $\mathcal{S}(\rn)\subset\mathcal{L}_{\loc}^{\Phi_p}(\rn)$ is continuous.
This finishes the proof of Lemma \ref{schwartz}.
\end{proof}

\begin{corollary}\label{cor-pwm}
Let $\alpha\in(0,\infty)$. Then, for any $g\in \mathcal{S}(\rn)$,
$g$ is a pointwise multiplier of $\Lambda_{n\alpha}(\mathbb{R}^n)$.
\end{corollary}

\begin{proof}
Let $\alpha\in(0,\infty)$.
By Theorem \ref{thm-main2}, it suffices to prove that
$g\in L^\infty(\rn)\cap \mathcal{L}_{\loc}^{\Phi_p}(\rn)$,
where the number $p$ satisfies that $\az=1/p-1>0$ and hence $p=\frac{1}{1+\alpha}\in(0,1)$.
It is obvious that Schwartz functions are bounded. The fact that Schwartz functions are in
$\mathcal{L}_{\loc}^{\Phi_p}(\rn)$
has been proved in Lemma \ref{schwartz}, which completes the proof of Corollary
\ref{cor-pwm}.
\end{proof}

\section{Bilinear decompositions}\label{s4}

In this section, we establish the bilinear decompositions of the
product spaces $h^p(\rn)\times \Lambda_{n\alpha}(\mathbb{R}^n)$ and $H^p(\rn)\times \Lambda_{n\alpha}(\rn)$
with $p\in(0,1)$ and $\az:=1/p-1$, and the
product space $H^1(\rn)\times \bmo(\rn)$, and we prove that these bilinear
decompositions are sharp in some sense.

\subsection{Bilinear decomposition of $h^p(\rn)\times \Lambda_{n\alpha}(\mathbb{R}^n)$
with $p\in (0,1)$ and $\az:=1/p-1$}\label{s4.1}

In this subsection, we establish the bilinear decomposition of the
product space $h^p(\rn)\times \Lambda_{n\alpha}(\mathbb{R}^n)$, where $p\in(0,1)$ and $\az:=1/p-1$,
and obtain the sharpness of this decomposition.

Let us begin with the following notion of
the multiresolution analysis
of $L^2(\rn)$ (see, for instance, \cite[p.\,21]{Me92} or \cite[Remark 1.52]{gg}).

\begin{definition}\label{MRA}
Let $\{V_j\}_{j\in\zz}$ be an increasing sequence of closed linear subspaces in $L^2(\rn)$.
Then $\{V_j\}_{j\in\zz}$ is called a {\it multiresolution analysis} (for short, MRA) of
$L^2(\rn)$ if it has the following
properties:

\begin{itemize}
\item [{\rm(a)}] $\bigcap_{j\in\zz} V_j=\{{\bf 0}\}$ and $\overline{\bigcup_{j\in\zz} V_j}=L^2(\rn)$,
where ${\bf 0}$ denotes the zero element of $L^2(\rn)$;

\item [{\rm(b)}] for any $j\in\zz$ and $f\in L^2(\rn)$, $f(\cdot)\in V_j$ if and only if $f(2\cdot)\in V_{j+1}$;

\item [{\rm(c)}] for any $f\in L^2(\rn)$ and $k\in\zz^n$, $f(\cdot)\in V_0$ if and only if $f(\cdot-k)\in V_{0}$;

\item [{\rm(d)}] there exists a function $\phi\in L^2(\rn)$ (called a {\it scaling function} or {\it father wavelet})
such that
$\{\phi(\cdot-k)\}_{k\in\zz^n}$ is a \emph{Riesz basis} of $V_0$, that is,
for every sequence $\{\az_k\}_{k\in\zz^n}$ of scalars,
$$\lf\|\sum_{k\in\zz^n} \az_k \phi(\cdot-k)\r\|_{L^2(\rn)}\sim \lf(\sum_{k\in\zz^n}|\az_k|^2\r)^{1/2},$$
where the
positive equivalence constants are independent of $\{\az_k\}_{k\in\zz^n}$.
\end{itemize}
\end{definition}

Notice that, in many literature, the definition of MRA is restricted to the one-dimension case, however,
the extension from one-dimension to higher dimension is classical via the tensor product
method (see \cite[p.\,921]{Da88} or \cite[Section 3.9]{Me92}).

Let $H$ be a Hilbert space and $V$ a subspace of $H$.
Recall that the \emph{orthogonal complement} of $V$
is defined to be the set of all elements $x\in H$ such that,
for any $y\in V$, $\la x,y\ra=0$, where $\la\cdot,\cdot\ra$
denotes the inner product of $H$.
Let $\{M_j\}_{j\in\nn}$ be a sequence
of subspaces of $H$. Then $H$ is said to be the \emph{direct sum} of $\{M_j\}_{j\in\nn}$
if
\begin{itemize}
\item[(i)] for any $i,\ j\in\nn$ with $i\neq j$ and, for any $x\in M_i$ and $y\in M_j$,
$\la x,y\ra=0$;
\item[(ii)] for any $f\in H$, there exists a sequence $\{x_j\}_{j\in\nn}\subset H$ such that,
for any $j\in\nn$, $x_j\in M_j$ and $f=\sum_{j\in\nn}x_j$ in $H$.
\end{itemize}
The \emph{symbol}
$H=\bigoplus_{j\in\nn} M_j$
denotes that
$H$ is the direct sum of $\{M_j\}_{j\in\nn}$ (see \cite[p.\,124]{kf} for more details).

For any $j\in\zz$, let $\{V_j\}_{j\in\zz}$ be as in Definition \ref{MRA} and
$W_j$ the orthogonal complement of
$V_{j}$ in $V_{j+1}$. It is easy to see that
\begin{align}\label{f-VW}
V_{j+1}=\bigoplus_{i=-\fz}^j W_i
\ \ \quad\text{and}\ \ \quad
L^2(\rn)=\bigoplus_{i=0}^{\fz} W_i\bigoplus V_{0}.
\end{align}
Let
$E:=\{0,\,1\}^n\setminus\{(\overbrace{0,\,\ldots,\,0}^{n {\textrm{ times}}})\}$ and
$\mathcal{D}$ be the class of all
dyadic cubes
$I_{j,k}:=\{x\in\rn: \, 2^j x-k\in[0,1)^n\}$ with $j\in\zz$ and $k\in\zz^n$ in $\rn$.
Assume that $\mathcal{D}_0$ is the class of all
dyadic cubes in $\rn$ with side lengths not greater than $1$, namely, $I_{j,k}\in\mathcal{D}_0$
if and only if
$I_{j,k}:=\{x\in\rn: \, 2^j x-k\in[0,1)^n\}$ with $j\in\zz_+$ and $k\in\zz^n$.

For any given $d\in\nn$, according to \cite[Sections~3.8 and 3.9]{Me92}, we know that
there exist $\{V_j\}_{j\in\zz}$ as in
Definition \ref{MRA}, and families of
\begin{equation}\label{www}
\text{father}\ \text{wavelets}\
\{\phi_I\}_{I\in\mathcal{D}}\ \text{and}\ \text{mother}\ \text{wavelets}\
\{\psi_I^{\lz}\}_{I\in\mathcal{D},\,\lz\in E}
\end{equation}
having the following properties:
\begin{enumerate}
\item[\rm (P1)] for any $j\in\zz$, the family $\{\phi_I\}_{|I|=2^{-jn}}$ forms an orthonormal basis of $V_j$
and the family $\{\psi_I^\lz\}_{|I|=2^{-jn},\lz\in E}$ an orthonormal basis of $W_j$. Moreover,
the family $\{\psi_I^{\lz}\}_{I\in\mathcal{D}_0,\,\lz\in E}\bigcup\{\phi_I\}_{I\in\mathcal{D}_0,|I|=1}$
forms an orthonormal basis of $L^2(\rn)$;

\item[{\rm (P2)}] there exists a positive constant $m\in(1,\infty)$ such that,
for any $I\in \mathcal{D}$ and $\lz\in E$,
\begin{equation}\label{220}
\supp \phi_I\subset mI \ \ \ \ \text{and} \ \ \ \ \supp \psi^\lz_I\subset mI,
\end{equation}
where $mI$ denotes the $m$ dilation of $I$ with the same center as $I$;

\item[{\rm (P3)}] for any $\az\in\zz_+^n$ with $|\az|\le d$, there exists a positive constant $C$ such that,
for any $I\in \mathcal{D}$, $\lz\in E$ and $x\in\rn$, it holds true that
\begin{equation*}
\lf|\partial^\az\phi_I(x)\r|+\lf|\partial^\az\psi^\lz_I(x)\r|\le C \ell_I^{-n/2-|\az|},
\end{equation*}
where $\ell_I$ denotes the side length of $I$;

\item[{\rm (P4)}] for any $I\in \mathcal{D}$, $\lz\in E$ and $\nu\in\zz_+^n$ with $|\nu|\le d$,
it holds true that
\begin{equation*}
\int_{\rn}x^{\nu}\psi_I^\lz(x)\,dx=0
\end{equation*}
and, for any $I\in \mathcal{D}$,
$$\int_{\rn}\phi_I(x)\,dx\neq0;$$

\item[{\rm (P5)}] For any $I$, $I'\in \mathcal{D}$ satisfying $|I|\le|I'|$ and $\lz\in E$,
\begin{equation}\label{f-woc}
\int_{\rn}\psi_I^\lz(x) \phi_{I'}(x)\,dx=0.
\end{equation}
Indeed, let $W_j$ and $V_{j'}$ be the linear subspaces of $L^2(\rn)$ defined as in \eqref{f-VW},
respectively, with
$|I|=2^{-jn}$ and $|I'|=2^{-j'n}$. Since $|I|\le|I'|$, we deduce $j'\le j$, which, combined
with \eqref{f-VW}, shows that $W_j \perp V_{j'}$. By this and the above property (P1), we
prove \eqref{f-woc}.
\end{enumerate}

Let us point out that the constants $m$ and $C$ in the above properties (P2) and (P3) depend on
the \emph{regularity constant $d$} (see \cite{Da88} or \cite[p.\,96]{Me92}).
Notice that, even in the one-dimensional case, there does not exist a wavelet basis in $L^2(\rr)$
whose elements are
both infinitely differentiable and have compact supports (see, for instance, \cite[Theorem 3.8]{HW96}).

Observe that the family $\{\phi_I\}_{I\in\mathcal{D}_0,|I|=1}
\bigcup\{\psi_I^{\lz}\}_{I\in\mathcal{D}_0,\,\lz\in E}$
forms an orthonormal basis of $L^2(\rn)$.
Thus, for any $f\in L^2(\rn)$, we have
\begin{equation}\label{eq-crf}
f=\sum_{I\in\mathcal{D}_0,|I|=1}\langle f,\,\phi_I \rangle \phi_I+
\sum_{I\in \mathcal{D}_0}\sum_{\lz\in E} \langle f,\,\psi_I^\lz \rangle \psi_I^\lz
\end{equation}
in $L^2(\rn)$, where $\langle \cdot,\,\cdot \rangle$ denotes the inner product in $L^2(\rn)$.
A function $f$ in $L^2(\rn)$ is said to have a \emph{finite wavelet expansion} if the coefficients $\langle f,\,\psi_I^\lz \rangle$ and $\langle f,\,\phi_I \rangle$ in
\eqref{eq-crf} have only
finite non-zero terms.

For any $j\in\zz$, let
$P_j$ and $Q_j$ be the  orthogonal projectors of $L^2(\rn)$, respectively, onto $V_j$ and $W_j$. As
in \cite{Do95} (see also \cite[Section~4]{BGK12}),
for any $j\in\zz$ and $f\in L^2(\rn)$, we have
\begin{equation}\label{eq-Pj}
P_jf=\sum_{|I|=2^{-jn}} \langle f,\,\phi_I \rangle\phi_I,
\ \ \ \ Q_jf=\sum_{
|I|=2^{-jn}} \sum_{\lz\in E} \langle f,\,\psi_I^\lz \rangle \psi_I^\lz
\end{equation}
and
\begin{equation}\label{eq-Qj}
f=\sum_{j\in\zz}Q_jf=P_0f+\sum_{j=0}^{\infty}Q_jf,
\end{equation}
where the equality holds true in $L^2(\rn)$.

\begin{lemma}\label{1q}
Let $f$, $g\in L^2(\rn)$ be such that at least one of them has a finite wavelet expansion.
Then it holds true that
\begin{equation*}
fg=\sum_{j=0}^{\infty} (P_j f)(Q_j g)+\sum_{j=0}^{\infty} (Q_j f)(P_j g)
+\sum_{j=0}^{\infty} (Q_j f)(Q_j g)+P_0fP_0g
\qquad \textup{in} \;\; L^1(\rn).
\end{equation*}
\end{lemma}
\begin{proof}
Assume that $g$ has a finite wavelet expansion and $f\in L^2(\rn)$. Then there exists an $M\in\nn$
such that $Q_j g=0$ for any $j>M$. From this, \eqref{eq-Qj} and the fact that, for any $k\in\zz_+$,
$Q_k=P_{k+1}-P_{k}$, we deduce that
$$
g=P_0g+\sum_{j=0}^{M}Q_jg=P_{M+1}g.
$$
Combining this and the fact that $f=\lim_{j\to\infty}P_jf$ in $L^2(\rn)$,
we obtain $fg=\lim_{j\to\infty}P_jfP_jg$ in $L^1(\rn)$.
By this and the fact that, for any $j\in\zz_+$, $P_{j+1}=P_j+Q_j$, we conclude that
\begin{align*}
fg&=\lim_{k\to\infty}(P_kfP_kg-P_0fP_0g)+P_0fP_0g\\
&=\lim_{k\to\infty}\sum_{j=0}^{k-1}(P_{j+1}fP_{j+1}g-P_jfP_jg)+P_0fP_0g\\
&=\lim_{k\to\infty}\sum_{j=0}^{k-1}[(P_jf+Q_jf)(P_jg+Q_jg)-P_jfP_jg]+P_0fP_0g\\
&=\sum_{j=0}^{\infty} (P_j f)(Q_j g)+\sum_{j=0}^{\infty} (Q_j f)(P_j g)
+\sum_{j=0}^{\infty} (Q_j f)(Q_j g)+P_0fP_0g,
\end{align*}
where the equality holds true in $L^1(\rn)$ and the existence of the above three limits are guaranteed
by the fact that, for any $j>M$, $Q_jg=0$ and $P_jg=g$. This
finishes the proof of Lemma \ref{1q}.
\end{proof}

From Lemma \ref{1q}, we deduce that, for any $f$, $g\in L^2(\rn)$ satisfying
that at least one of them has a finite wavelet expansion,
\begin{equation}\label{eq-L2L2}
fg=\sum^4_{i=1}\Pi_i(f,\,g)
\qquad \textup{in} \;\; L^1(\rn),
\end{equation}
where
\begin{equation}\label{eq-pi1}
\Pi_1(f,\,g):=\sum_{I,\,I'\in \mathcal{D}_0,|I|=|I'|} \sum_{\lz\in E} \langle f,\,\phi_I \rangle
\langle g,\,\psi_{I'}^\lz \rangle \phi_I \psi_{I'}^\lz\qquad \textup{in} \;\; L^1(\rn),
\end{equation}
\begin{equation}\label{eq-pi2}
\Pi_2(f,\,g):=\sum_{I,\,I'\in \mathcal{D}_0,
|I|=|I'|} \sum_{\lz\in E} \langle f,\,\psi_I^\lz \rangle
\langle g,\,\phi_{I'} \rangle \psi_I^\lz \phi_{I'}\qquad \textup{in} \;\; L^1(\rn),
\end{equation}
\begin{align}\label{eq-pi3}
\Pi_3(f,\,g):&=\dsum_{\gfz{I,\,I'\in \mathcal{D}_0}{
|I|=|I'|}} \dsum_{\gfz{\lz,\,\lz'\in E}{
(I,\,\lz)\ne (I',\,\lz')}} \langle f,\,\psi_I^\lz \rangle
\langle g,\,\psi_{I'}^{\lz'} \rangle \psi_I^\lz \psi_{I'}^{\lz'}\\\noz
&\quad+
\dsum_{\gfz{I,\,I'\in \mathcal{D}_0}
{|I|=|I'|=1,I\neq I'}}  \langle f,\,\phi_I \rangle
\langle g,\,\phi_{I'} \rangle \phi_I \phi_{I'}\\\noz
&=:\Pi_{3,1}(f,\,g)+\Pi_{3,2}(f,\,g)
\qquad \textup{in} \;\; L^1(\rn),
\end{align}
and
\begin{align}\label{eq-pi4}
\Pi_4(f,\,g)&:=\sum_{I\in \mathcal{D}_0} \sum_{\lz\in E} \langle f,\,\psi_I^\lz \rangle
\langle g,\,\psi_{I}^\lz \rangle \lf(\psi_{I}^\lz\r)^2
+
\dsum_{\gfz{I\in \mathcal{D}_0}
{|I|=1}}  \langle f,\,\phi_I \rangle
\langle g,\,\phi_I \rangle \lf(\phi_I\r)^2\qquad\textup{in} \;\; L^1(\rn).
\end{align}
From their definitions in \eqref{eq-pi1} through \eqref{eq-pi4}, it easily follows
that the four operators $\{\Pi_i\}_{i=1}^4$ are bilinear operators
for any $f$, $g\in L^2(\rn)$ satisfying
that at least one of them has a finite wavelet expansion.

\begin{lemma}\label{rem-w1}
Let $\{\Pi_i\}_{i=1}^4$, $\Pi_{3,1}$ and $\Pi_{3,2}$
be as in \eqref{eq-pi1} through \eqref{eq-pi4}. Then
$\{\Pi_i\}_{i=1}^3$, $\Pi_{3,1}$ and $\Pi_{3,2}$
can be extended to bounded bilinear operators
from $L^2(\rn)\times L^2(\rn)$ to $H^1(\rn)$
and $\Pi_4$ to a bounded bilinear
operator from $L^2(\rn)\times L^2(\rn)$ to $L^1(\rn)$.
\end{lemma}

\begin{proof}
Repeating the
argument similar to that used in the proof of \cite[Lemma 4.2]{BGK12}
(see also \cite[Lemma 10.2.2]{ylk}), we find that $\{\Pi_i\}_{i=1}^3$,
$\Pi_{3,1}$ and $\Pi_{3,2}$
can be extended to bounded bilinear operators
from $L^2(\rn)\times L^2(\rn)$ to $H^1(\rn)$.

By the H\"older inequality and
the fact that, for any $\lz\in E$ and $I\in\mathcal{D}_0$,
$\|(\psi_{I}^\lz)^2\|_{L^1(\rn)}=1$ and $\|(\phi_{I}^\lz)^2\|_{L^1(\rn)}=1$, we know that
\begin{align*}
\|\Pi_4(f,g)\|_{L^1(\rn)}
&\le \sum_{I\in \mathcal{D}_0} \sum_{\lz\in E} |\langle f,\,\psi_I^\lz \rangle|\,
|\langle g,\,\psi_{I}^\lz \rangle| \|(\psi_{I}^\lz)^2\|_{L^1(\rn)}\\
&\quad+
\dsum_{I\in \mathcal{D}_0,
|I|=1}  |\langle f,\,\phi_I \rangle||
\langle g,\,\phi_I \rangle| \|\lf(\phi_I\r)^2\|_{L^1(\rn)}
\\
&\le \lf(\sum_{I\in \mathcal{D}_0} \sum_{\lz\in E} |\langle f,\,\psi_I^\lz \rangle|^2
+\dsum_{I\in \mathcal{D}_0,
|I|=1}  |\langle f,\,\phi_I \rangle|^2\r)^{1/2}\\
&\quad\times
\lf(\sum_{I\in \mathcal{D}_0} \sum_{\lz\in E} |\langle g,\,\psi_I^\lz \rangle|^2
+\dsum_{I\in \mathcal{D}_0,
|I|=1}  |\langle g,\,\phi_I \rangle|^2
\r)^{1/2}\\
&= \|f\|_{L^2(\rn)}\|g\|_{L^2(\rn)}.
\end{align*}
This implies that $\Pi_4$ can be extended to a bounded bilinear
operator from $L^2(\rn)\times L^2(\rn)$ to $L^1(\rn)$, which completes
the proof of Lemma \ref{rem-w1}.
\end{proof}

\begin{remark}\label{rem-w3}
Assume that $f$ has a finite wavelet expansion, $g\in L_\loc^2(\rn)$ (the set of all locally $L^2(\rn)$
integrable functions), and
the coefficients $\langle f,\,\psi_I^\lz \rangle$ and $\langle f,\,\phi_I \rangle$ in
\eqref{eq-crf} have only
finite non-zero terms.
Then we may as well assume that there exists an $R\in\mathcal{D}$ satisfying that,
for any $I\in\mathcal{D}_0$ such that $|I|=1$ and $\langle f,\,\phi_I \rangle\neq0$,
or for any $I\in\mathcal{D}_0$ and $\lz\in E$ such that  $\langle f,\,\psi_I^\lz \rangle\neq0$,
$I\subset R$ holds true.
Then we know that $f$ is supported in the cube $mR$, where $m$ is as in \eqref{220}.
Take $\eta$ to be a smooth cut-off function such that $\supp\eta\subset 9m R$ and $\eta\equiv 1$ on $5mR$.
Since $f$ has a finite wavelet expansion and $\eta g\in L^2(\rn)$,
it follows from Lemma \ref{rem-w1} that, for any $i\in\{1,2,3,4\}$,
every $\Pi_i(f,\eta g)$ is well defined in $L^1(\rn)$.
We claim that, for any $i\in\{1,2,3,4\}$,
\begin{equation}\label{eq-x10}
\Pi_i(f,g)=\Pi_i(f, \eta g).
\end{equation}
The proof of this claim is similar to that of \cite[Remark 4.4]{bckly}.
For the convenience of the reader, we present some details here.
Indeed, as for $i=1$,
since $f$ has a finite wavelet expansion and $\eta g\in L^2(\rn)$, it follows from Lemma \ref{rem-w1} that
\begin{align*}
\Pi_1(f,\,\eta g)=\dsum_{\gfz{I,\,I'\in \mathcal{D}_0}{|I|=|I'|}} \dsum_{\lz\in E} \langle f,\,\phi_I \rangle
\langle \eta g,\,\psi_{I'}^\lz \rangle \phi_I \psi_{I'}^\lz \ \ \ \ \ \text{in} \ L^1(\rn).
\end{align*}
Based on the above properties (P2) and (P5)
of wavelets,
the factor $\langle f,\,\phi_I \rangle
\phi_I \psi_{I'}^\lz$ in the above sum is non-zero only when
$(mI)\cap (mI')\neq\emptyset$, $|I|\leq|R|$
and $mI\cap mR\neq\emptyset$, which implies that
$ mI'\subset 5mR$. From this and the fact
that $\eta\equiv 1$ on $5mR$, it follows that
we can remove the function $\eta$ in
the pairing $\langle \eta g,\,\psi_{I'}^\lz \rangle$ and hence obtain
\begin{align*}
\Pi_1(f,\,\eta g)=\dsum_{\gfz{I,\,I'\in \mathcal{D}_0}{|I|=|I'|}} \dsum_{\lz\in E} \langle f,\,\phi_I \rangle
\langle  g,\,\psi_{I'}^\lz \rangle \phi_I \psi_{I'}^\lz=\Pi_1(f,g) \ \ \ \ \ \text{in} \ L^1(\rn).
\end{align*}
Thus, \eqref{eq-x10} makes sense when $i=1$.
When $i\in\{2,3,4\}$, the proof is similar and we omit the details. Thus, the claim \eqref{eq-x10} holds true.
\end{remark}

The next lemma gives a finite atomic decomposition of elements in local
Hardy spaces that have finite wavelet expansions, which is just \cite[Theorem 2.7]{cky1}.

\begin{lemma}\label{lem-hp}
Let $p\in(0,\,1]$, $s\in\zz_+$ with $s\ge \lfloor n(1/p-1)\rfloor$
and $d\in\nn$ with $
d> n(2/p+1/2)$.
Let $\{\psi_I^{\lz}\}_{I\in\mathcal{D}_0,\,\lz\in E}\bigcup\{\phi_I\}_{I\in\mathcal{D}_0,|I|=1}$ be the wavelets
as in \eqref{www}
with the regularity parameter $d$.
Assume that $f\in h^p(\rn)$ has a finite wavelet expansion, namely,
\begin{equation}\label{eqn 3.7}
f=\sum_{I\in \mathcal{D}_0,|I|=1}\langle f,\,\phi_I \rangle \phi_I+
\sum_{I\in \mathcal{D}_0}\sum_{\lz\in E} \langle f,\,\psi_I^\lz \rangle \psi_I^\lz,
\end{equation}
where
the coefficients $\langle f,\,\psi_I^\lz \rangle$ and $\langle f,\,\phi_I \rangle$ in
\eqref{eqn 3.7} have only
finite non-zero terms.
Then $f$ has a finite atomic decomposition satisfying
$f=\sum_{l=1}^L \mu_l a_l$, where $L\in\nn$, $\mu_l\in\mathbb{C}$ for any $l\in\{1,\,\ldots,\,L\}$,
and the following properties hold true:
\begin{enumerate}
\item[\rm (i)] for any $l\in\{1,\,\ldots,\,L\}$,
$a_l$ is a local $(p,2,s)$-atom
supported in some cube $mR_l$ with $m$ as in \eqref{220} and $R_l\in\mathcal{D}_0$,
which can be written into the following form:
\begin{equation}\label{eqn 3.x4}
a_l=\sum_{\gfz{I\in \mathcal{D}_0,I\subset R_l}{|I|=1}} c_{(I,\,l)} \phi_I+
\sum_{I\in \mathcal{D}_0, {I\subset R_l}}\sum_{\lz\in E} c_{(I,\,\lz,\,l)} \psi_I^\lz,
\end{equation}
where $\{c_{(I,\,\lz,\,l)}\}_{I\subset R_l,\,\lz\in E}\subset\rr$ and
$\{c_{(I,\,l)}\}_{I\subset R_l}\subset\rr$;

\item[\rm (ii)] there exists a positive constant $ C$,
independent of $\{\mu_l\}_{l=1}^L$, $\{a_l\}_{l=1}^L$
and $f$, such that
$$\lf\{\sum_{l=1}^L |\mu_l|^p\r\}^{\frac{1}{p}}\le  C \|f\|_{h^p(\rn)};$$

\item[\rm (iii)] for any $l\in\{1,\,\ldots,\,L\}$,
the wavelet expansion of $a_l$ in \eqref{eqn 3.x4} is also finite which is
extracted from that of $f$ in \eqref{eqn 3.7}.
\end{enumerate}
\end{lemma}

The following lemma is just \cite[Theorem 2.8(ii)]{cky1}.

\begin{lemma}\label{thm}
Let $p\in(0,\,1)$ and $\az:=1/p-1$.
Let $\{\psi_I^{\lz}\}_{I\in\mathcal{D}_0,\,\lz\in E}\bigcup\{\phi_I\}_{I\in\mathcal{D}_0,|I|=1}$ be the wavelets
as in \eqref{www}
with the regularity parameter $d\in\nn$ satisfying $
d> n(2/p+1/2)$.
Then there exists a positive constant $C$
such that, for any $g\in\Lambda_{n\alpha}(\mathbb{R}^n)$,
$$
\sup_{I\in \mathcal{D}_0}
\lf\{\frac{1}{|I|^{2\az+1}}\lf[
\sum_{J\subset I,|J|=1}\lf|\langle f,\,\phi_J \rangle\r|^2 +
\sum_{J\in \mathcal{D}_0,J\subset I}\sum_{\lz\in E} \lf|\langle f,\,\psi_J^\lz \rangle\r|^2
\r]\r\}^{\frac{1}{2}}\leq C\|g\|_{\Lambda_{n\alpha}(\mathbb{R}^n)}.
$$
\end{lemma}

\begin{proposition}\label{prop-pi1}
Let $p\in(0,1)$ and $\az:=1/p-1$. Assume that the regularity parameter $d\in\nn$ appearing in (P3)
and (P4) of wavelets satisfies that
$
d> n(2/p+1/2)$.
Then   the bilinear operator $\Pi_1$, defined as in \eqref{eq-pi1},
can be extended to a bilinear operator
bounded from $h^p(\rn)\times\Lambda_{n\alpha}(\mathbb{R}^n)$ to $h^1(\rn)$.
\end{proposition}

\begin{proof}
Assume that $f\in h^p(\rn)$ has a finite wavelet expansion and $g\in\Lambda_{n\alpha}(\rn)$.
Let $s\in\zz_+$ be such that $s\ge \lfloor n(1/p-1)\rfloor$.
By Lemma \ref{lem-hp}, we know that $f=\sum_{l=1}^L \mu_l a_l$ has a
finite atomic decomposition with the same notation as therein.
Assume that every local $(p,2,s)$-atom $a_l$ is supported in the cube
$mR_l$ with $m$ as in \eqref{220} and $R_l\in\mathcal{D}_0$ as in Lemma \ref{lem-hp}.
For any $l\in\{1,\dots, L\}$, define
\begin{align*}
b_l:=\sum_{I\in\mathcal{D}_0,I\subset 5mR_l}\sum_{\lz\in E} \langle g,\,\psi_I^\lz \rangle \psi_I^\lz.
\end{align*}
Applying the above property (P1) of wavelets and Lemma \ref{thm}, we conclude that, for any $l\in\{1,\dots, L\}$,
\begin{align*}
\|b_l\|_{L^2(\rn)}\ls
\lf(\sum_{I\in\mathcal{D}_0,I\subset 5mR_l}\sum_{\lz\in E}\lf|\langle g,\,\psi_I^\lz \rangle\r|^2\r)^{\frac{1}{2}}
\ls |R_l|^{\az+1/2} \|g\|_{\Lambda_{n\alpha}(\mathbb{R}^n)}.
\end{align*}
By the fact that, for any $l\in\{1,\dots, L\}$, $b_l\in L^2(\rn)$ and $a_l$ is a local $(p,2,s)$-atom,
and Lemmas \ref{rem-w1} and \ref{1q}, we know that, for any $l\in\{1,\dots, L\}$,
\begin{align}\label{eq-x6}
\|\Pi_1(a_l,\,b_l)\|_{h^1(\rn)}\ls\|\Pi_1(a_l,\,b_l)\|_{H^1(\rn)}\ls
\|a_l\|_{L^2(\rn)}\|b_l\|_{L^2(\rn)}
\ls \|g\|_{\Lambda_{n\alpha}(\mathbb{R}^n)}.
\end{align}
From the above property (P1)
of wavelets, it follows that, for any $l\in\{1,\dots, L\}$,
\begin{align}\label{eq-x5}
 \Pi_1(a_l,\,b_l)
&=  \dsum_{\gfz{I,\,I'\in \mathcal{D}_0}{|I|=|I'|}} \sum_{\lz\in E} \langle a_l,\,\phi_I \rangle
\langle b_l,\,\psi_{I'}^\lz \rangle \phi_I \psi_{I'}^\lz
=  \dsum_{\gfz{I,\,I'\in \mathcal{D}_0}{|I|=|I'|, I'\subset 5mR_l}} \sum_{\lz\in E} \langle a_l,\,\phi_I \rangle
\langle g,\,\psi_{I'}^\lz \rangle \phi_I \psi_{I'}^\lz.
\end{align}
Moreover, according to Lemma \ref{lem-hp}(iii), the wavelet expansion of $a_l$ has only finite terms.
By this, properties (P2) and (P5) of wavelets, we know that, for any $l\in\{1,\dots, L\}$
and $I,\,I'\in \mathcal{D}_0$, the factor $\langle a_l, \phi_I\rangle
 \phi_I \psi_{I'}^\lz$ in the above sum is non-zero only when
 $(mI)\cap (mI')\neq\emptyset$, $|I|\leq|R_l|$
and $mI\cap mR_l\neq\emptyset$, which implies that
$ I'\subset 5mR_l$.
Therefore, the restriction in the last term of \eqref{eq-x5} can be removed and hence we have,
for any $l\in\{1,\dots, L\}$,
\begin{align*}
 \Pi_1(a_l,\,b_l)
=  \sum_{I,\,I'\in \mathcal{D}_0,|I|=|I'|} \sum_{\lz\in E} \langle a_l,\,\phi_I \rangle
\langle g,\,\psi_{I'}^\lz \rangle \phi_I \psi_{I'}^\lz=\Pi_1(a_l,g).
\end{align*}
From this, the fact that $\Pi_1$ is bilinear, \eqref{eq-x6} and Lemma \ref{lem-hp}(i), we deduce that
\begin{align}\label{eq-x8}
\|\Pi_1(f,\,g)\|_{h^1(\rn)}
&=\lf\|\sum_{l=1}^L \mu_l \Pi_1(a_l,\,g) \r\|_{h^1(\rn)}
=\lf\|\sum_{l=1}^L \mu_l \Pi_1(a_l,\,b_l) \r\|_{h^1(\rn)}\\\noz
&\le  \sum_{l=1}^L |\mu_l|\|\Pi_1(a_l,\,b_l)\|_{h^1(\rn)}
\ls \lf(\sum_{l=1}^L|\mu_l|^p\r)^{1/p} \|g\|_{\Lambda_{n\alpha}(\mathbb{R}^n)}\\\noz
&\ls \|f\|_{h^p(\rn)} \|g\|_{\Lambda_{n\alpha}(\mathbb{R}^n)}. \notag
\end{align}

For any $f\in h^p(\rn)$, observe that the regularity parameter $d\in\nn$ satisfies that
$d> n(2/p+1)$.
By this and \cite[Theorem 1.64]{gg}, we know that the family
of wavelets, $\{\psi_I^{\lz}\}_{I\in\mathcal{D}_0,\,\lz\in E}\bigcup\{\phi_I\}_{I\in\mathcal{D}_0,|I|=1}$,
is an unconditional basis of
$h^p(\rn)$ (see \cite[Definition 1.56]{gg} for the precise definition), which implies that
there exists a sequence $\{f_k\}_{k\in\nn}\subset h^p(\rn)$ having finite wavelet expansions such that
$\lim_{k\to \fz}f_k=f$ in $h^p(\rn)$.
Thus, using \eqref{eq-x8}, we extend the definition of $\Pi_1$
by setting, for any $f\in h^p(\rn)$ and $g\in \Lambda_{n\alpha}(\mathbb{R}^n)$,
\begin{align*}
\Pi_1(f,\,g):=\lim_{k\to \fz} \Pi_1(f_k,\,g) \qquad\textup{in}\;\; h^1(\rn).
\end{align*}
Estimate \eqref{eq-x8} ensures that the above definition is independent of the
choice of the sequence $\{f_k\}_{k\in\nn}$ and hence is
well defined. Based
on this extension, we deduce from \eqref{eq-x8} again that
\begin{align*}
\|\Pi_1(f,\,g)\|_{h^1(\rn)}=\lim_{k\to \fz}
\|\Pi_1(f_k,\,g)\|_{h^1(\rn)}\ls\lim_{k\to \fz}
\|f_k\|_{h^p(\rn)}\|g\|_{\Lambda_{n\alpha}(\mathbb{R}^n)}\ls
\|f\|_{h^p(\rn)}\|g\|_{\Lambda_{n\alpha}(\mathbb{R}^n)}.
\end{align*}
This implies that
$\Pi_1$ can be extended to a bilinear operator
bounded from $h^p(\rn)\times\Lambda_{n\alpha}(\mathbb{R}^n)$ to $h^1(\rn)$.
This finishes the proof of Proposition \ref{prop-pi1}.
\end{proof}

\begin{proposition}\label{prop-pi2}
Let $p\in(0,1)$ and $\az:=1/p-1$. Assume that the regularity parameter $d\in\nn$ appearing in (P3)
and (P4) of wavelets satisfies that
$d> n(2/p+1/2)$.
Then  the bilinear operator $\Pi_2$, defined as in \eqref{eq-pi2},
can be extended to a bilinear operator
bounded from $h^{p}(\rn)\times\Lambda_{n\alpha}(\mathbb{R}^n)$ to $h^{\Phi_p}(\rn)$
with $\Phi_{p}$ as in \eqref{333}.
\end{proposition}

\begin{remark}
Let $p\in(0,1)$. Compared with
the corresponding target space appearing in
\cite[Proposition 4.9]{bckly},
the target space $h^{\Phi_p}(\rn)$
in Proposition \ref{prop-pi2}
looks much simpler than the corresponding one $H^{\phi_p}(\rn)$ in \cite{bckly} with
$\phi_p$ as in \eqref{356}.
The key reason for this is that, for $\alpha:=1/p-1$, each function in $\Lambda_{n\az}(\rn)$ is bounded but the function in $\mathfrak{C}_{\alpha}(\mathbb{R}^n)$
is not necessarily bounded. Moreover, Proposition \ref{prop-pi2} can not contain the case $p=1$, because the function in $\bmo(\mathbb{R}^n)$, the dual space of $h^{1}(\rn)$,
is also not necessarily bounded.
\end{remark}

\begin{proof}[Proof of Proposition \ref{prop-pi2}]
Let $g\in \Lambda_{n\alpha}(\mathbb{R}^n)$  and $s:=\lfloor n\az\rfloor$.  Assume that $a$ is a
local $(p,2,2s)$-atom supported in the cube
$mR$ with $m$ as in \eqref{220} and $R\in\mathcal{D}_0$ as in Lemma \ref{lem-hp}, and $a$ has a finite  wavelet expansion.
Let $B$ be  the smallest ball in $\rn$ containing $9mR$ and
$P_B^s g$ the minimizing polynomial of $g$ on $B$ with degree $\le s$
as in Definition \ref{campa}.
Let $\eta$ be a smooth cut-off function such that $\supp\eta\subset 9m R$ and $\eta\equiv 1$ on $5mR$.

We now consider two cases based on the size of $|B|$.
If $|B|\geq1$,
by Lemma \ref{444}(i), Remark \ref{rem-w3}, the fact that $\Pi_2$ is bounded from $L^2(\rn)\times L^2(\rn)$
to $H^1(\rn)$ (see Lemma \ref{rem-w1}), $|B| \sim|9mR|$ and $\supp\eta\subset 9mR$,
we conclude that
\begin{align*}
\|\Pi_2(a,\, g)\|_{h^{\Phi_p}(\rn)}
&\ls\|\Pi_2(a,\,\eta g)\|_{h^{1}(\rn)}\ls\|\Pi_2(a,\,\eta g)\|_{H^{1}(\rn)}
\ls\|a\|_{L^2(\rn)}\|\eta g\|_{L^2(\rn)}\\
&\ls\frac{|R|^{1/2}}{|R|^{1/p}}|9mR|^{1/2}\|g\|_{L^\infty(\rn)}
\ls\|g\|_{L^\infty(\rn)}
\ls \|g\|_{\Lambda_{n\alpha}(\mathbb{R}^n)},
\end{align*}
which is the desired estimate.

If $|B|<1$, by \eqref{eqn 3.x4}, we
immediately know that $a$ can be written into the following form:
$$
a=\sum_{I\in \mathcal{D}_0, {I\subset R}}\sum_{\lz\in E} c_{(I,\,\lz,\,l)} \psi_I^\lz,
$$
which, together with the orthogonality of the wavelets
$\{\psi_I^{\lz}\}_{I\in\mathcal{D}_0,\,\lz\in E}\bigcup\{\phi_I\}_{I\in\mathcal{D}_0,|I|=1}$, implies that,
for any $I\in\mathcal{D}_0$ satisfying $|I|=1$, $\la a,\phi_I \ra=0$ and hence
$$
\sum_{I,\,I'\in \mathcal{D}_0,
|I|=|I'|=1}  \langle a,\,\phi_I \rangle
\langle g,\,\phi_{I'} \rangle \phi_I \phi_{I'}=0\quad\text{and}\quad
\dsum_{I\in \mathcal{D}_0,
|I|=1}  \langle a,\,\phi_I \rangle
\langle g,\,\phi_I \rangle \lf(\phi_I\r)^2=0.
$$
From this, Remark \ref{rem-w3} and
the property (P4) of wavelets with $d$ therein satisfying
$d> \lfloor n\az\rfloor$,
it follows
that, for any $i\in\{1,\,3,\,4\}$,
\begin{align*}
\Pi_i(a,\,\eta P_B^s g)
=\Pi_i(a,\,P_B^s g)=0.
\end{align*}
This, combined with \eqref{eq-L2L2}, implies that
\begin{align*}
aP_B^s g=
a(\eta P_B^s g)
=\dsum_{i=1}^4\Pi_i(a,\,\eta P_B^s g)
=\Pi_2(a,\,\eta P_B^s g)=\Pi_2(a,\, P_B^s g).
\end{align*}
Moreover, by Lemma \ref{444}(i), the fact that $\Pi_2$ is bounded from $L^2(\rn)\times L^2(\rn)$
to $H^1(\rn)$ (see Lemma \ref{rem-w1}), $|B| \sim|9mR|$, $\supp\eta\subset 9mR$ and
Lemma \ref{GR-DDY}, we conclude that
\begin{align*}
\|\Pi_2(a,\,\eta [g-P_B^s g])\|_{h^{\Phi_p}(\rn)}
&\ls \|\Pi_2(a,\,\eta [g-P_B^s g])\|_{h^1(\rn)}
\ls\|\Pi_2(a,\,\eta [g-P_B^s g])\|_{H^1(\rn)}\\
&\ls\|a\|_{L^2(\rn)}\|\eta(g-P_B^s g)\|_{L^2(\rn)}
\ls\frac{|R|^{1/2}}{|R|^{1/p}}|B|^{1/2}\lf[\dashint_{B}|g(x)-P_B^s g(x)|^2\,dx\r]^{\frac{1}{2}}\\
&\sim\frac{1}{|B|^{\alpha}}\lf[\dashint_{B}|g(x)-P_B^s g(x)|^2\,dx\r]^{\frac{1}{2}}
\ls\|g\|_{\mathcal{L}_{\loc}^{\az,2,s}(\rn)}
\sim \|g\|_{\Lambda_{n\alpha}(\mathbb{R}^n)}.
\end{align*}
From this and Proposition \ref{ap}, we deduce that
\begin{align}\label{eqn-3.35}
\|\Pi_2(a,\,g)\|_{h^{\Phi_p}(\rn)}&
\ls \|\Pi_2(a,\,g-P_B^s g)\|_{h^{\Phi_p}(\rn)}+\|aP_B^s g\|_{h^{\Phi_p}(\rn)}
\ls \|g\|_{\Lambda_{n\alpha}(\mathbb{R}^n)},
\end{align}
which is also the desired estimate.

We now extend the above boundedness
from an atom $a$ to any $f\in h^p(\rn)$ with a finite wavelet expansion.
Let $g\in \Lambda_{n\alpha}(\mathbb{R}^n)$ and $f\in h^p(\rn)$ have a finite atomic decomposition
$f=\sum_{l=1}^L\mu_la_l$ with the same notation as in Lemma \ref{lem-hp}, where,
for any $l\in\{1,\dots, L\}$, $a_l$ is a
local $(p,2,2s)$-atom.
By the definition of $\|\cdot\|_{h^{\Phi_p}(\rn)}$,
to obtain the desired boundedness, it suffices to show that there exists a
positive constant $C$ such that
\begin{align}\label{aim-xx1}
\int_\rn \Phi_p\lf(\frac{m(\Pi_2(f,g),\varphi)(x)}{C\|f\|_{h^p(\rn)}
\|g\|_{\Lambda_{n\alpha}(\mathbb{R}^n)}}\r)\,dx\le 1,
\end{align}
where $m(\cdot,\varphi)$ is as in Definition \ref{hp}(i).
Without loss of generality, we may assume that $\|f\|_{h^p(\rn)}=1$ and
$\|g\|_{\Lambda_{n\alpha}(\mathbb{R}^n)}=1$.
Otherwise, we may replace $f$ and $g$, respectively, by $\wz f:={f}/{\|f\|_{h^p(\rn)}}$
and $\wz g:={g}/{\|g\|_{\Lambda_{n\alpha}(\mathbb{R}^n)}}$ in the argument below.

Now, we prove \eqref{aim-xx1}. From Lemma \ref{lem-hp}, we
deduce that
\begin{align}\label{aim}
\lf(\sum_{l=1}^L |\mu_l|^p\r)^{1/p}\le \wz C
\|f\|_{h^p(\rn)}=\wz C
\end{align}
for some positive constant $\wz C$ independent of $f$.
Without loss of generality, we may as well assume
that $\wz C\ge 1$.
From the fact that, for any $l\in\{1,\dots, L\}$, $a_l$ is a
local $(p,2,2s)$-atom, it follows that
\eqref{eqn-3.35} holds true with the atom $a$ replaced by $a_l$.
By this, $\|g\|_{\Lambda_{n\alpha}(\mathbb{R}^n)}=1$ and
the definition of $\|\cdot\|_{h^{\Phi_p}(\rn)}$, we conclude that
there exists a positive constant $C_1\in(1,\infty)$, independent of $a_l$ and $g$, such that
\begin{align}\label{eqn-3.35x}
\int_\rn \Phi_p\lf(\frac{m(\Pi_2(a_l,\,g),\varphi)(x)}{C_1}\r)\,dx\le 1.
\end{align}
Let $\mathbb A:= 2^{1/p} \wz C$. Then $\mathbb A>1$.
Observe that
$$
m(\Pi_2(f,g),\varphi)
\le \sum_{l=1}^L |\mu_l|
m(\Pi_2(a_l, g),\varphi).
$$
By this, the fact that $\Phi_p$ is strictly increasing, and Lemma \ref{444}(ii), we find that
\begin{align}\label{xx2}
\dint_\rn \Phi_p\lf(\frac{m(\Pi_2(f,g),\varphi)(x)}{\mathbb A C_1}\r)\,dx&
\le \int_\rn \Phi_p\lf(\frac{\sum_{l=1}^L |\mu_l|  m(\Pi_2(a_l, g),\varphi)(x)}
{\mathbb A C_1}\r)\,dx\\
\nonumber&\le\dsum_{l=1}^L \int_\rn \Phi_p\lf(\frac{|\mu_l| m(\Pi_2(a_l, g),\varphi)(x)}{\mathbb A C_1}\r)\,dx
=:\dsum_{l=1}^L {\rm D}_l.
\end{align}
Then \eqref{eqn-3.35x} and Lemma \ref{444}(ii) imply that
\begin{align*}
\begin{cases}
\displaystyle{\rm D}_l\le \frac{|\mu_l|^p}{\mathbb A^p}\int_\rn \Phi_p\lf(
\frac{ m(\Pi_2(a_l, g),\varphi)(x)}{C_1}\r)\,dx \le\frac{|\mu_l|^p}{\mathbb A^p}&\qquad \textup{when}\; |\mu_l|\le \mathbb A,
\vspace{0.2cm}\\
\displaystyle
{\rm D}_l\le \frac{|\mu_l|}{\mathbb A}\int_\rn \Phi_p\lf(
\frac{  m(\Pi_2(a_l, g),\varphi)(x)}{C_1}\r)\,dx \le\frac{|\mu_l|}{\mathbb A}&\qquad \textup{when}\;
|\mu_l|> \mathbb A.
\end{cases}
\end{align*}
Combining this, \eqref{aim} and \eqref{xx2}, we obtain
\begin{align*}
\int_\rn \Phi_p\lf(\frac{m(\Pi_2(f,g)),\varphi)(x)}{\mathbb A C_1}\r)\,dx
&\le \frac1 {\mathbb A^p}\sum_{\{1\le l\le L:\, |\mu_l|\le \mathbb A\}} |\mu_l|^p
+\frac1{\mathbb A}\sum_{\{1\le l\le L:\, |\mu_l|> \mathbb A\}} |\mu_l|\le  \frac{\wz C^p} {\mathbb A^p}+\frac{\wz C}{\mathbb A}  < 1,
\end{align*}
which implies that \eqref{aim-xx1} holds true with the positive constant
$C$ therein taken as $\mathbb A C_1$. Thus,  we arrive at the conclusion that
\begin{align*}
\|\Pi_2(f,g)\|_{h^{\Phi_p}(\rn)}
\ls \|f\|_{h^p(\rn)}\|g\|_{\Lambda_{n\alpha}(\mathbb{R}^n)},
\end{align*}
whenever $f\in h^p(\rn)$ has a finite wavelet expansion.

Repeating the
argument similar to that used in
the proof of Proposition \ref{prop-pi1}, we find that the definition of
$\Pi_2(f,g)$ can be extended to any $f\in h^p(\rn)$ and
$g\in \Lambda_{n\alpha}(\mathbb{R}^n)$ with
the desired boundedness and the details are omitted.
This finishes the proof of Proposition \ref{prop-pi2}.
\end{proof}

Similarly to the proof of Proposition \ref{prop-pi1}, we obtain the following results
on the boundedness of the bilinear operators $\Pi_3$ and $\Pi_4$.

\begin{proposition}\label{prop-pi3}
Let $p\in(0,1)$ and $\az:=1/p-1$. Assume that the regularity parameter $d\in\nn$ appearing in (P3)
and (P4) of wavelets satisfies that
$
d> n(2/p+1/2)$.
Then   the bilinear operator $\Pi_3$, defined as in \eqref{eq-pi3},
can be extended to a bilinear operator
bounded from $h^p(\rn)\times\Lambda_{n\alpha}(\mathbb{R}^n)$ to $h^1(\rn)$.
\end{proposition}

\begin{proof}
According to Lemma \ref{1q}, for any $f,\ g\in L^2(\rn)$ with
$f$ having a finite wavelet expansion, the operator $\Pi_3(f,g)$
is well defined
and we have
\begin{align*}
\Pi_3(f,\,g)&=\dsum_{\gfz{I,\,I'\in \mathcal{D}_0}{
|I|=|I'|}} \dsum_{\gfz{\lz,\,\lz'\in E}{
(I,\,\lz)\ne (I',\,\lz')}} \langle f,\,\psi_I^\lz \rangle
\langle g,\,\psi_{I'}^{\lz'} \rangle \psi_I^\lz \psi_{I'}^{\lz'}+
\dsum_{\gfz{I,\,I'\in \mathcal{D}_0}
{|I|=|I'|=1}}  \langle f,\,\phi_I \rangle
\langle g,\,\phi_{I'} \rangle \phi_I \phi_{I'}\\
&=:\Pi_{3,1}(f,\,g)+\Pi_{3,2}(f,\,g).
\end{align*}
From Lemma \ref{rem-w1},
it follows that $\Pi_{3,1}(f,\,g)$ and $\Pi_{3,2}(f,\,g)$
can be extended to a bounded bilinear operator from $L^2(\rn)\times L^2(\rn)$ to $h^1(\rn)$.
Then the bilinear operator $\Pi_{3,1}$ can be handled in a way similar to that used in the proof of
Proposition \ref{prop-pi1} and the details are omitted.

To prove the boundedness of $\Pi_{3,2}$,
let $g\in \Lambda_{n\alpha}(\mathbb{R}^n)$  and $s:=\lfloor n\az\rfloor$.  Assume that $a$ is a
local $(p,2,s)$-atom supported
in the cube
$mR$ with $m$ as in \eqref{220} and $R\in\mathcal{D}_0$,
and $a$ has a finite  wavelet expansion of the following form
as in \eqref{eqn 3.x4}:
\begin{equation}\label{1e}
a=\sum_{|I|=1, {I\subset R}} c_{(I)} \phi_I+
\sum_{I\in \mathcal{D}_0, {I\subset R}}\sum_{\lz\in E} c_{(I,\,\lz)} \psi_I^\lz.
\end{equation}

We now consider two cases based on the size of $|R|$. If $|R|<1$, by \eqref{1e},
we immediately know that
$$
a=\dsum_{\gfz{I\in \mathcal{D}_0} {I\subset R}}\sum_{\lz\in E}c_{(I,\,\lz)} \psi_I^\lz,
$$
which, together with the orthogonality of the wavelets $\{\phi_I\}_{I\in \mathcal{D}_0,\, |I|=1}$
and $\{\psi_I^\lz\}_{I\in \mathcal{D}_0,\,\lz\in E}$, implies that, for any $I\in\mathcal{D}_0$
such that $|I|=1$, $\la a,\phi_I\ra=0$ and hence
$$
\Pi_{3,2}(a,\,g)=
\dsum_{\gfz{I,\,I'\in \mathcal{D}_0}
{|I|=|I'|=1,I\neq I'}}
 \langle a,\,\phi_I \rangle
\langle g,\,\phi_{I'} \rangle \phi_I \phi_{I'}=0.
$$

If $|R|=1$, let $B$ be the smallest ball in $\rn$ containing $9mR$ and
$\eta$ be a smooth cut-off function such that $\supp\eta\subset 9mR$ and $\eta\equiv 1$ on $5mR$.
Then, by Lemma \ref{444}(i), the fact that $\Pi_{3,2}$ is bounded from $L^2(\rn)\times L^2(\rn)$
to $H^1(\rn)$, $a$ is a local $(p,2,s)$-atom and $\supp\eta\subset 9mR$,
we conclude that
\begin{align*}
\|\Pi_{3,2}(a,\, g)\|_{h^{\Phi_p}(\rn)}
&\ls\|\Pi_{3,2}(a,\,\eta g)\|_{h^{1}(\rn)}\ls\|\Pi_{3,2}(a,\,\eta g)\|_{H^{1}(\rn)}
\ls\|a\|_{L^2(\rn)}\|\eta g\|_{L^2(\rn)}\\
&\ls\frac{|R|^{1/2}}{|R|^{1/p}}|9mR|^{1/2}\|g\|_{L^\infty(\rn)}
\ls\|g\|_{L^\infty(\rn)}
\ls \|g\|_{\Lambda_{n\alpha}(\mathbb{R}^n)}.
\end{align*}
Then, repeating the
argument similar to that used in
the proof of Proposition \ref{prop-pi1},
we obtain the desired conclusion and the details are omitted, which completes the proof of Proposition
\ref{prop-pi3}.
\end{proof}

\begin{proposition}\label{prop-pi4}
Let $p\in(0,1)$ and $\az:=1/p-1$. Assume that the regularity parameter $d\in\nn$ appearing in (P3)
and (P4) of wavelets satisfies that
$
d> n(2/p+1/2)$.
Then
 the bilinear operator $\Pi_4$, defined as in \eqref{eq-pi4},
can be extended to a bilinear operator
bounded from $h^p(\rn)\times\Lambda_{n\alpha}(\mathbb{R}^n)$ to $L^1(\rn)$.
\end{proposition}

\begin{proof}
Assume that $f\in h^p(\rn)$ has a finite wavelet expansion and $g\in\Lambda_{n\alpha}(\rn)$.
Let $s\in\zz_+$ be such that $s\ge \lfloor n(1/p-1)\rfloor$.
In this case, by Lemma \ref{lem-hp}, we know that $f=\sum_{l=1}^L \mu_l a_l$ has a
finite atomic decomposition with the same notation as therein.
Assume that every local $(p,2,s)$-atom $a_l$ is supported in
the cube
$mR_l$ with $m$ as in \eqref{220} and $R_l\in\mathcal{D}_0$ as in Lemma \ref{lem-hp}.
For any $l\in\{1,\dots, L\}$, define
\begin{align*}
b_l:=\sum_{I\in \mathcal{D}_0,I\subset 5mR_l,|I|=1}\langle g,\,\phi_I \rangle \phi_I+\sum_{I\in\mathcal{D}_0,I\subset 5mR_l}\sum_{\lz\in E} \langle g,\,\psi_I^\lz \rangle \psi_I^\lz.
\end{align*}
Applying the property (P1) of wavelets and Lemma \ref{thm}, we conclude that, for any $l\in\{1,\dots, L\}$,
\begin{align*}
\|b_l\|_{L^2(\rn)}\ls
\lf(\sum_{I\in\mathcal{D}_0,I\subset 5mR_l}
\sum_{\lz\in E}\lf|\langle g,\,\psi_I^\lz \rangle\r|^2+\dsum_{I\in \mathcal{D}_0,I\subset 5mR_l,
|I|=1}  |\langle g,\,\phi_I \rangle|^2\r)^{\frac{1}{2}}
\ls |R_l|^{\az+1/2} \|g\|_{\Lambda_{n\alpha}(\mathbb{R}^n)}.
\end{align*}
By the fact that, for any $l\in\{1,\dots, L\}$, $b_l\in L^2(\rn)$ and $a_l$ is a local $(p,2,s)$-atom,
and Lemmas \ref{1q} and \ref{rem-w1}, we know that, for any $l\in\{1,\dots, L\}$,
\begin{align}\label{x6}
\|\Pi_4(a_l,\,b_l)\|_{L^1(\rn)}\ls
\|a_l\|_{L^2(\rn)}\|b_l\|_{L^2(\rn)}
\ls \|g\|_{\Lambda_{n\alpha}(\mathbb{R}^n)}.
\end{align}
From the property (P1) of wavelets, it follows that, for any $l\in\{1,\dots, L\}$,
\begin{align}\label{x5}
 \Pi_4(a_l,\,b_l)
&=  \sum_{I\in \mathcal{D}_0} \sum_{\lz\in E} \langle a_l,\,\psi_I^\lz \rangle
\langle b_l,\,\psi_{I}^\lz \rangle \lf(\psi_{I}^\lz\r)^2
+
\dsum_{I\in \mathcal{D}_0
,|I|=1}  \langle a_l,\,\phi_I \rangle
\langle b_l,\,\phi_I \rangle \lf(\phi_I\r)^2\\\noz
&=  \sum_{I\in \mathcal{D}_0,I\subset 5mR_l} \sum_{\lz\in E} \langle a_l,\,\psi_I^\lz \rangle
\langle g,\,\psi_{I}^\lz \rangle \lf(\psi_{I}^\lz\r)^2+
\dsum_{\gfz{I\in \mathcal{D}_0,I\subset 5mR_l}
{|I|=1}}  \langle a_l,\,\phi_I \rangle
\langle g,\,\phi_I \rangle \lf(\phi_I\r)^2.
\end{align}
Moreover, according to Lemma \ref{lem-hp}(iii), the wavelet expansion of $a_l$ has only finite terms.
By this, properties (P2) and (P5) of wavelets, we know that, for any $\lz\in E$, $l\in\{1,\dots, L\}$
and $I\in \mathcal{D}_0$, the factors
$\langle a_l,\,\psi_I^\lz \rangle(\psi_{I}^\lz)^2$
and $\langle a_l,\,\phi_I \rangle(\phi_I)^2$
in the above sums are non-zero only when
 $|I|\leq|R_l|$
and $mI\cap mR_l\neq\emptyset$, which implies that
$ I\subset 5mR_l$.
Therefore, the restriction in the last term of \eqref{x5} can be removed and hence we have,
for any $l\in\{1,\dots, L\}$,
\begin{align*}
 \Pi_4(a_l,\,b_l)
= \sum_{I\in \mathcal{D}_0} \sum_{\lz\in E} \langle a_l,\,\psi_I^\lz \rangle
\langle g,\,\psi_{I}^\lz \rangle \lf(\psi_{I}^\lz\r)^2
+
\dsum_{I\in \mathcal{D}_0,
|I|=1}  \langle a_l,\,\phi_I \rangle
\langle g,\,\phi_I \rangle \lf(\phi_I\r)^2=\Pi_4(a_l,g).
\end{align*}
From this, the fact that $\Pi_4$ is bilinear, \eqref{x6} and Lemma \ref{lem-hp}(i), we deduce that
\begin{align*}
\|\Pi_4(f,\,g)\|_{L^1(\rn)}
&=\lf\|\sum_{l=1}^L \mu_l \Pi_4(a_l,\,g) \r\|_{L^1(\rn)}
=\lf\|\sum_{l=1}^L \mu_l \Pi_4(a_l,\,b_l) \r\|_{L^1(\rn)}\\\noz
&\le  \sum_{l=1}^L |\mu_l|\|\Pi_4(a_l,\,b_l)\|_{L^1(\rn)}
\ls \lf(\sum_{l=1}^L|\mu_l|^p\r)^{1/p} \|g\|_{\Lambda_{n\alpha}(\mathbb{R}^n)}\\\noz
&\ls \|f\|_{h^p(\rn)} \|g\|_{\Lambda_{n\alpha}(\mathbb{R}^n)}. \notag
\end{align*}

Finally, for any $f\in h^p(\rn)$, repeating an
argument similar to that used in
the proof of Proposition \ref{prop-pi1},
we obtain the desired estimate and we omit the details, which completes the proof of Proposition \ref{prop-pi4}.
\end{proof}

In what follows, for any $N\in\nn$, let
\begin{equation*}
\cf_N(\rn):=\lf\{\varphi\in\cs(\rn):\sum_{\beta\in\zz_+^n,|\beta|\le N}
\sup_{x\in\rn}\lf[\lf(1+|x|\r)^{N+n}\lf|\partial_x^\beta\varphi(x)\r|\r]\le1\r\}.
\end{equation*}
For any given $f\in\cs'(\rn)$, the \emph{radial grand maximal function} $m_N(f)$
of $f$ is defined by setting, for any $x\in\rn$,
\begin{equation}\label{EqMN0}
m_N(f)(x):=\sup\lf\{|f\ast\varphi_t(x)|:\ t\in(0,1),\ |y-x|<t\ \text{and}\ \varphi\in\cf_N(\rn)\r\},
\end{equation}
where, for any $t\in(0,\infty)$ and $\xi\in\mathbb R^n,\varphi_t(\xi):=t^{-n}\varphi(\xi/t)$.

To prove Theorem \ref{mainthm1}, we also need the following lemma.

\begin{lemma}\label{lyyy}
Let $\Phi$ be an Orlicz function with positive lower type $p_{\Phi}^-$
and positive upper type $p_{\Phi}^+$.
Assume that
$\{f_j\}_{j\in\nn}\subset h^\Phi(\rn)$ and
$f\in h^\Phi(\rn)$ satisfy $\lim_{j\to\infty}\|f_j-f\|_{h^\Phi(\rn)}=0$.
Then $\{f_j\}_{j\in\nn}$ converges to $f$ in $\cs'(\rn)$.
\end{lemma}

\begin{proof}
Let $\varphi\in\cs(\rn)$ and $\widetilde{\varphi}(\cdot):=2\varphi(2\cdot)$.
Observe that,
for any $f\in h^\Phi(\rn)$,
$x\in\rn$ and $y\in B(x,1/2)$,
$$|f\ast\varphi(x)|=|f\ast\widetilde{\varphi}_{\frac12}(x)|\lesssim m_N(f)(y),$$
where $m_N(f)$ is as in \eqref{EqMN0} and $N\in\nn$.
From this, \cite[Section 7.6 and Theorem 5.3]{shyy}
(see also \cite{wyyz,ZWYY} for some corrections),
it follows that, if $N$ is sufficiently large,
then, for any given $\varphi\in\cs(\rn)$,
\begin{align*}
\lf|\la f_j-f,\varphi\ra\r|&=\lf|(f_j-f)\ast[\varphi(-\cdot)](\vec 0_n)\r|
\lesssim\inf_{y\in B(\vec{0}_n,1)}m_N(f_j-f)(y)\\
&\lesssim\frac{\|\mathbf{1}_{B(\vec 0_n,1/2)}m_N(f_j-f)\|_{L^\Phi(\rn)}}
{\|\mathbf{1}_{B(\vec 0_n,1/2)}\|_{L^\Phi(\rn)}}
\lesssim \frac{\|f_j-f\|_{h^\Phi(\rn)}}{\|\mathbf{1}_{B(\vec 0_n,1/2)}\|_{L^\Phi(\rn)}}
\to0
\end{align*}
as $j\to\infty$, which further implies the desired conclusion of this lemma and
hence completes the proof of Lemma \ref{lyyy}.
\end{proof}

For any  $f\in h^p(\rn)$ with $p\in(0,\,1)$, and
$g\in \Lambda_{n\alpha}(\mathbb{R}^n)$ with $\az:=1/p-1$,
the {\it product} $f\times g$ is defined to be  a Schwartz
distribution in $\cs'(\rn)$ such that, for any $\phi\in \mathcal{S}(\rn)$,
\begin{align}\label{def-product}
\langle f\times g,\,\phi \rangle:=\langle \phi g,\,f \rangle,
\end{align}
where the last bracket denotes the dual pair between
$\Lambda_{n\alpha}(\mathbb{R}^n)$ and $h^p(\rn)$. From
Corollary \ref{cor-pwm}, we deduce that every
$\phi \in \mathcal{S}(\rn)$ is a pointwise multiplier of $\Lambda_{n\alpha}(\mathbb{R}^n)$,
which implies that the
equality \eqref{def-product} is well defined. By Theorem \ref{thm-main2}, we know that
$L^\infty(\rn)\cap \mathcal{L}_{\loc}^{\Phi_p}(\rn)$
characterizes the class of pointwise multipliers of
$\Lambda_{n\alpha}(\rn)$.
From this, it follows that
the largest range of $\varphi$ that makes
$$\langle  f\times g,\,\varphi\rangle=\langle  g\varphi, f\rangle$$ meaningful is $\varphi\in L^\infty(\rn)\cap \mathcal{L}_{\loc}^{\Phi_p}(\rn)$.

\begin{theorem}\label{mainthm1}
Let $p\in(0,1)$, $\az=1/p-1$
and $\Phi_p$ be as in \eqref{333}.
Then the following statements
hold true.
\begin{itemize}
\item[\rm (i)] There exist two bounded bilinear operators
$$S:\, h^p(\rn)\times \Lambda_{n\alpha}(\mathbb{R}^n)\to L^1(\rn)$$ and
$$T:\, h^p(\rn)\times \Lambda_{n\alpha}(\mathbb{R}^n)\to h^{\Phi_p}(\rn)$$
such that, for any $(f,g)\in h^p(\rn)\times \Lambda_{n\alpha}(\rn)$,
\begin{align*}
f\times g=S(f,\,g)+T(f,\,g)\qquad in\;\cs'(\rn).
\end{align*}
Moreover, there exists a positive constant $C$ such that, for any $(f,g)\in h^p(\rn)\times
\Lambda_{n\alpha}(\rn)$,
$$\lf\|S(f,g)\r\|_{L^1(\rn)}\le C \|f\|_{h^p(\rn)}\|g\|_{\Lambda_{n\alpha}(\mathbb{R}^n)}$$
and
$$
\lf\|T(f,g)\r\|_{h^{\Phi_p}(\rn)}\le C \|f\|_{h^p(\rn)}\|g\|_{\Lambda_{n\alpha}(\mathbb{R}^n)}.$$
\item[\rm (ii)] For any given $(f,g)\in h^p(\rn)\times \Lambda_{n\alpha}(\rn)$ and any
$\varphi\in L^\infty(\rn)\cap \mathcal{L}_{\loc}^{\Phi_p}(\rn)$,
$$
\la f\times g,\varphi\ra=\la \varphi, S(f,g)\ra+\la \varphi,T(f,g)\ra.
$$
\end{itemize}
\end{theorem}

\begin{proof}
Let $f\in h^p(\rn)$ and $g\in \Lambda_{n\alpha}(\mathbb{R}^n)$.
Assume that the regularity parameter $d\in\nn$ appearing in the properties (P3)
and (P4) of wavelets satisfies that
$
d> n(2/p+1)$. From this and \cite[Theorem 1.64]{gg}, we deduce that the family
of wavelets, $\{\psi_I^{\lz}\}_{I\in\mathcal{D}_0,\,\lz\in E}\bigcup\{\phi_I\}_{I\in\mathcal{D}_0,|I|=1}$,
is an \emph{unconditional basis} of
$h^p(\rn)$ (see \cite[Definition 1.56]{gg} for the precise definition),
and hence there exists a sequence
 $\{f_k\}_{k\in\nn}\subset
h^p(\rn)$ with finite wavelet expansions and satisfying
$\lim_{k\to \fz} f_k=f$ in $h^p(\rn)$.
By the definition of $f\times g$ in \eqref{def-product}, Corollary \ref{cor-pwm}
and Theorem \ref{thm-main2}, we know that
\begin{align}\label{88}
f\times g=\dlim_{k\to \fz} f_k \,g \qquad\textup{in}\ \mathcal{S}'(\rn),
\end{align}
and, for any $\varphi\in L^\infty(\rn)\cap \mathcal{L}_{\loc}^{\Phi_p}(\rn)$,
\begin{align}\label{99}
\la f\times g,\varphi\ra=\la\varphi g,f\ra=\dlim_{k\to \fz} \la\varphi g,f_k\ra=
\dlim_{k\to \fz} \int_{\rn}\varphi(x) g(x)f_k(x)\,dx,
\end{align}
where $f_k\, g$ denotes the usual pointwise product of $f_k$
and $g$. Since $f_k$ has a finite wavelet expansion, it follows that $f_k\in L^2(\rn)$ and
$f_k$ is supported in a ball $B(\vec 0_n, R_k)$ for some $R_k\in(0,\infty)$. Let $\eta_k$
be a cut-off function satisfying $\supp \eta_k\subset
B(\vec 0_n, 9mR_k)$ and $\eta_k\equiv 1$ on $B(\vec 0_n, 5mR_k)$, where $m$ is as in the
property (P2) of wavelets.
By Remark \ref{rem-w3}, we find that, for any $i\in\{1,\,2,\,3,\,4\}$ and $k\in\nn$,
$$\Pi_i(f_k,\,\eta_kg)=\Pi_i(f_k,\,g).$$
Using this and \eqref{eq-L2L2}, we conclude that
$$
f_k g =f_k(\eta_k g)=\sum_{i=1}^4 \Pi_i(f_k, \eta_k g)=\sum_{i=1}^4 \Pi_i(f_k, g)\qquad \textup{in} \;\; L^1(\rn).
$$
From Lemma \ref{444}(i),
$\lim_{k\to \fz} f_k=f$ in $h^p(\rn)$ and Propositions \ref{prop-pi1} through \ref{prop-pi4}, we
deduce that, for any $i\in\{1,2,3\}$,
\begin{align}\label{66}
\dlim_{k\to \fz} \Pi_i(f_k,\,g)
=\Pi_i(f,\,g)\qquad  \textup{in}\;h^{\Phi_p}(\rn),
\end{align}
and
\begin{align}\label{77}
\dlim_{k\to \fz} \Pi_4(f_k,\,g)
=\Pi_4(f,\,g)\qquad  \textup{in}\;L^1(\rn).
\end{align}
By Lemma \ref{lyyy}, we know that
the convergence of a sequence in $h^1(\rn)$ or $h^{\Phi_p}(\rn)$
implies its convergence  in $\cs'(\rn)$.
By this, \eqref{66} and \eqref{77}, we conclude that,
for any $i\in\{1,2,3,4\}$,
$$\dlim_{k\to \fz} \Pi_i(f_k,\,g)
=\Pi_i(f,\,g)\qquad  \textup{in}\;\cs'(\rn).$$
Therefore, from \eqref{88}, it follows that
\begin{align*}
f\times g=\dlim_{k\to \fz} f_k \,g =\dsum_{i=1}^4\dlim_{k\to \fz} \Pi_i(f_k,\,g)
=\dsum_{i=1}^4 \Pi_i(f,\,g) \qquad  \textup{in}\;\cs'(\rn).
\end{align*}
Thus, if we define
\begin{align}\label{f-So}
S(f,\,g):=\Pi_4(f,\,g)
\end{align}
and
\begin{align}\label{f-To}
T(f,\,g):=\sum_{i=1}^3\Pi_i(f,\,g),
\end{align}
then $S(f,\,g)\in L^1(\rn)$ and $T(f,g)\in h^{\Phi_p}(\rn)$.
Applying Propositions \ref{prop-pi1} through \ref{prop-pi4}, we obtain the desired conclusion of (i).

As for (ii), by \eqref{99}, \eqref{66}, \eqref{77} and the fact that $\varphi\in L^\infty(\rn)\cap \mathcal{L}_{\loc}^{\Phi_p}(\rn)$, we conclude that
\begin{align}\label{991}
\la f\times g,\varphi\ra&=
\dlim_{k\to \fz} \int_{\rn}\varphi(x) g(x)f_k(x)\,dx=
\dlim_{k\to \fz} \int_{\rn}\varphi(x) \sum_{i=1}^4\Pi_i(f_k,\,g)(x)\,dx\\\noz
&=\dlim_{k\to \fz} \int_{\rn}\varphi(x) S(f_k,\,g)(x)\,dx+
\dlim_{k\to \fz} \int_{\rn}\varphi(x) T(f_k,\,g)(x)\,dx\\\noz
&=\la \varphi, S(f,g)\ra+\la \varphi,T(f,g)\ra.
\end{align}
This finishes the proof of (ii)
and hence of
Theorem \ref{mainthm1}.
\end{proof}

\begin{theorem}\label{449}
Let $p\in(0,1)$, $\az=1/p-1$
and $\Phi_p$ be as in \eqref{333}.
Assume that $\cy$ is a quasi-Banach space satisfying
$\cy\subset h^{\Phi_p}(\rn)$, $\|\cdot\|_{h^{\Phi_p}(\rn)}\lesssim\|\cdot\|_{\cy}$ and Theorem \ref{mainthm1} with $h^{\Phi_p}(\rn)$ therein replaced
by $\cy$. Then
$L^\infty(\rn)\cap (\cy)^\ast=L^\infty(\rn)\cap (h^{\Phi_p}(\rn))^\ast$.
\end{theorem}
\begin{proof}
From the fact that $\cy$ satisfies Theorem \ref{mainthm1}(ii) with $h^{\Phi_p}(\rn)$ therein replaced
by $\cy$, it follows that, for any given $(f,g)\in h^p(\rn)\times \Lambda_{n\alpha}(\rn)$,
\begin{align*}
\langle\varphi,S(f,\,g)\rangle+\langle\varphi,T(f,\,g)\rangle=
\langle f\times g,\,\varphi\rangle=\la\varphi g,f\ra
,\qquad \forall\, \varphi\in L^\infty(\rn)\cap (\cy)^\ast,
\end{align*}
where $S$ and $T$ are as in \eqref{f-So} and \eqref{f-To}.
From this, we deduce that $\varphi$ is a pointwise multiplier of $\Lambda_{n\alpha}(\mathbb{R}^n)$, which,
by Theorem \ref{thm-main2}, implies that $\varphi\in L^\infty(\rn)\cap \mathcal{L}_{\loc}^{\Phi_p}(\rn)$.
We therefore obtain $L^\infty(\rn)\cap (\cy)^\ast\subset L^\infty(\rn)\cap \mathcal{L}_{\loc}^{\Phi_p}(\rn)$.
From $\cy\subset h^{\Phi_p}(\rn)$ and the fact $\|\cdot\|_{h^{\Phi_p}(\rn)}\lesssim\|\cdot\|_{\cy}$,
it follows that, for any $L\in (h^{\Phi_p}(\rn))^\ast$ and $f\in\cy$,
$$
|L(f)|\lesssim\|f\|_{h^{\Phi_p}(\rn)}\lesssim\|f\|_{\cy},
$$
which implies that $L\in(\cy)^\ast$ and hence $(h^{\Phi_p}(\rn))^\ast\subset (\cy)^\ast$.
From this, we
further deduce that
$$ L^\infty(\rn)\cap (\cy)^\ast\subset L^\infty(\rn)\cap
\mathcal{L}_{\loc}^{\Phi_p}(\rn)= L^\infty(\rn)\cap (h^{\Phi_p}(\rn))^\ast\subset L^\infty(\rn)\cap (\cy)^\ast,$$
which implies that $L^\infty(\rn)\cap (\cy)^\ast=L^\infty(\rn)\cap (h^{\Phi_p}(\rn))^\ast$.
This finishes the proof Theorem \ref{449}.
\end{proof}

\begin{remark}\label{rem-add}
The sharpness of  Theorem \ref{mainthm1} is implied by Theorem \ref{449}. Indeed,
suppose that Theorem \ref{mainthm1} holds true with $h^{\Phi_p}(\rn)$ therein
replaced by a smaller quasi-Banach space $\cy$. From Theorem \ref{449}, we deduce that
$L^\infty(\rn)\cap (\cy)^\ast=L^\infty(\rn)\cap (h^{\Phi_p}(\rn))^\ast$.
In this sense,
we say that Theorem \ref{mainthm1} is sharp. It
is still unclear whether or not the Orlicz Hardy
space $h^{\Phi_p}(\rn)$ is indeed the smallest space,
in the sense of the inclusion of sets, having the property
in Theorem \ref{mainthm1}.
\end{remark}

\subsection{Bilinear decomposition of $H^p(\rn)\times \Lambda_{n\alpha}(\rn)$
with $p\in (0,1)$ and $\az:=1/p-1$}\label{s4.2}

In this subsection, we establish a bilinear decomposition of the
product space $H^p(\rn)\times \Lambda_{n\alpha}(\rn)$, where $p\in(0,1)$ and $\az:=1/p-1$.

Now we recall the notions of the classical Hardy and Campanato spaces.

\begin{definition}\label{Hp}
Let $p\in(0,1]$.
\begin{itemize}
\item[\rm (i)] Let $\varphi\in\mathcal{S}(\rn)$ and $f\in\mathcal{S}'(\rn)$.
The \emph{radial maximal function} $M(f,\varphi)$ is defined by setting
$$M(f,\varphi)(x):=\sup_{s\in(0,\infty)}\lf|(\varphi_s \ast f)(x)\r|,\quad \forall\,x\in\rn,$$
where
$\varphi_s(\cdot):=\frac{1}{s^n}\varphi(\frac{\cdot}{s})$
for any $s\in(0,\infty)$.
\item[\rm (ii)]
The \emph{Hardy  space $H^{p}(\rn)$} is defined to be the set of all $f\in\cs'(\rn)$
such that
$$
\|f\|_{H^{p}(\rn)}
:=\|M(f,\varphi)\|_{L^p(\rn)}<\infty,
$$
where $\varphi\in\mathcal{S}(\rn)$ satisfies
$\int_{\rn}\varphi(x)\,dx\neq0$ and $M(f,\varphi)$
is as in (i).
\item[\rm (iii)]
Let
$\Phi_p$ be as in \eqref{333}.
The \emph{Orlicz Hardy
space $H^{\Phi_p}(\rn)$} is defined to be the set of all $f\in\cs'(\rn)$
such that $\|f\|_{H^{\Phi_p}(\rn)}:=
\|M(f,\varphi)\|_{L^{\Phi_p}(\rn)}<\infty$,
where $\varphi\in\mathcal{S}(\rn)$ satisfies $\int_{\rn}\varphi(x)\,dx\neq0$ and $M(f,\varphi)$ is as
in (i).
\end{itemize}
\end{definition}

The dual of the Hardy space turns out to be the
Campanato space
which was first introduced by Campanato in \cite{Ca63,Ca64}.

\begin{definition}\label{cs}
Let $\az\in[0,\,\fz)$ and $s:=\lfloor n\az\rfloor$. The {\it Campanato space}
$\mathfrak C_{\az}(\rn)$ is defined to be the collection of all locally integrable functions $g$ such that
\begin{align*}
\|g\|_{\mathfrak C_{\az}(\rn)}:=
\dsup_{\mathrm{ball\,}B\subset \rn}
\frac{1}{|B|^{\az+1}}\int_B|g(x)-P_B^sg(x)|\,dx<\infty,
\end{align*}
where $P_B^sg$ for any ball $B\subset\rn$ denotes the minimizing polynomial of $g$ on $B$ with degree
not greater than $s$.
\end{definition}

\begin{remark}\label{rem-Cam}
Let $\az$ and $s$ be as in Definition \ref{cs}.
When $p\in(0,\,1]$ such that $\az=1/p-1$, we deduce from \cite[Theorem 5.30]{GR85} or \cite[p. 55,\,Theorem 4.1]{Lu95} that
$$\lf(H^p(\rn)\r)^*=\mathfrak C_{\az}(\rn)/\cp_{s}(\rn)$$ with equivalent quasi-norms.
\end{remark}

Let $\{\psi_I^{\lz}\}_{I\in\mathcal{D},\,\lz\in E}$ be the wavelets
as in \eqref{www}
with the regularity parameter $d\in\nn$.
Observe that the family $\{\psi_I^{\lz}\}_{I\in\mathcal{D},\,\lz\in E}$
forms an orthonormal basis of $L^2(\rn)$.
Thus, for any $f\in L^2(\rn)$, we have
\begin{equation}\label{ecrf}
f=\sum_{I\in \mathcal{D}}\sum_{\lz\in E} \langle f,\,\psi_I^\lz \rangle \psi_I^\lz
\end{equation}
in $L^2(\rn)$.
From \cite[(4.6)]{bckly}, we deduce that, for any $f$, $g\in L^2(\rn)$ satisfying
that at least one of them has a finite wavelet expansion,
\begin{equation*}
fg=\sum^4_{i=1}\Pi_i'(f,\,g)
\qquad \textup{in} \;\; L^1(\rn),
\end{equation*}
where
\begin{align}\label{pi1}
\Pi_1'(f,\,g):=\dsum_{\gfz{I,\,I'\in \mathcal{D}}{|I|=|I'|}} \dsum_{\lz\in E} \langle f,\,\phi_I \rangle
\langle g,\,\psi_{I'}^\lz \rangle \phi_I \psi_{I'}^\lz \qquad \textup{in} \;\; L^1(\rn),
\end{align}
\begin{align}\label{pi2}
\Pi_2'(f,\,g):=\dsum_{\gfz{I,\,I'\in \mathcal{D}}{
|I|=|I'|}} \dsum_{\lz\in E} \langle f,\,\psi_I^\lz \rangle
\langle g,\,\phi_{I'} \rangle \psi_I^\lz \phi_{I'}\qquad \textup{in} \;\; L^1(\rn),
\end{align}
\begin{align}\label{pi3}
\Pi_3'(f,\,g):=\dsum_{\gfz{I,\,I'\in \mathcal{D}}{
|I|=|I'|}} \dsum_{\gfz{\lz,\,\lz'\in E}{
(I,\,\lz)\ne (I',\,\lz')}} \langle f,\,\psi_I^\lz \rangle
\langle g,\,\psi_{I'}^{\lz'} \rangle \psi_I^\lz \psi_{I'}^{\lz'}\qquad \textup{in} \;\; L^1(\rn)
\end{align}
and
\begin{align}\label{pi4}
\Pi_4'(f,\,g):=\dsum_{I\in \mathcal{D}} \dsum_{\lz\in E} \langle f,\,\psi_I^\lz \rangle
\langle g,\,\psi_{I}^\lz \rangle \lf(\psi_{I}^\lz\r)^2\qquad \textup{in} \;\; L^1(\rn).
\end{align}

\begin{remark}\label{rw3}
Assume that $f$ has a finite wavelet expansion, $g\in L_\loc^2(\rn)$ (the set of all locally $L^2(\rn)$
integrable functions), and
the coefficients $\langle f,\,\psi_I^\lz \rangle$ in
\eqref{ecrf} have only
finite non-zero terms.
Then we may as well assume that there exists an $R\in\mathcal{D}$ satisfying that,
for any $I\in\mathcal{D}$ and $\lz\in E$ such that $\langle f,\,\psi_I^\lz \rangle\neq0$,
$I\subset R$ holds true.
Then we know that $f$ is supported in the cube $mR$, where $m$ is as in \eqref{220}.
Take $\eta$ to be a smooth cut-off function such that $\supp\eta\subset 9m R$ and $\eta\equiv 1$ on $5mR$.
It follows from \cite[Remark 4.4]{bckly} that, for any $i\in\{1,2,3,4\}$,
every $\Pi_i'(f,\eta g)$ is well defined in $L^1(\rn)$ and
\begin{equation*}
\Pi_i'(f,g)=\Pi_i'(f, \eta g).
\end{equation*}
\end{remark}

From \cite[Propositions 4.8, 4.10 and 4.11]{bckly} and the fact $\Lambda_{n\alpha}(\rn)\subset\mathfrak C_{\az}(\rn)$,
we deduce the following conclusions of $\Pi_1'$, $\Pi_3'$ and $\Pi_4'$.

\begin{proposition}\label{ppi1}
Let $p\in(0,1)$ and $\az:=1/p-1$.
Assume that the regularity parameter $d\in\nn$ of wavelets satisfies that
$d>\lfloor n(1/p-1)\rfloor$.
Then the bilinear operator $\Pi_1'$, defined as in \eqref{pi1},
can be extended to a bilinear operator
bounded from $H^p(\rn)\times\Lambda_{n\alpha}(\rn)$ to $H^1(\rn)$.
\end{proposition}

\begin{proposition}\label{ppi3}
Let $p\in(0,1)$ and $\az:=1/p-1$.
Assume that the regularity parameter $d\in\nn$ of wavelets satisfies that
$d>\lfloor n(1/p-1)\rfloor$. Then the bilinear operator $\Pi_3'$, defined as in \eqref{pi3},
can be extended to a bilinear operator
bounded from $H^p(\rn)\times\Lambda_{n\alpha}(\rn)$ to $H^1(\rn)$.
\end{proposition}

\begin{proposition}\label{ppi4}
Let $p\in(0,1)$ and $\az:=1/p-1$.
Assume that the regularity parameter $d\in\nn$ of wavelets satisfies that
$d>\lfloor n(1/p-1)\rfloor$. Then
the bilinear operator $\Pi_4'$, defined as in \eqref{pi4},
can be extended to a bilinear operator
bounded from $H^p(\rn)\times\mathfrak \Lambda_{n\alpha}(\rn)$ to $L^1(\rn)$.
\end{proposition}

Now we establish the following estimate of $\Pi_2'$. Let us first recall the notion of $(p,2,l)$-atoms (see, for instance, \cite{GR85,Lu95} for more details).
\begin{definition}\label{defa}
Let $p\in(0,\,1]$ and $l\in\zz_+$.
A function $a\in L^2(\rn)$ is called a $(p,2,l)${\it-atom} if
\begin{enumerate}
\item[\rm (i)] there exists
a ball $B$ such that $\supp a:=\{x\in\rn:\ a(x)\neq0\} \subset B$;
\item[\rm (ii)] $\|a\|_{L^2(\rn)}\le |B|^{1/2-1/p}$;
\item[\rm (iii)] $
\int_\rn x^\az a(x)\,dx=0
$ for any multi-index $\az\in\zz_+^n$ satisfying $|\az|\le l$.
\end{enumerate}
\end{definition}

\begin{lemma}\label{apW}
Let $p\in (0,\,1)$, $\Phi_p$ be as in \eqref{333}, $s:=\lfloor n(1/p-1)\rfloor$ and $d\in\zz_+\cap [2s,\,\fz)$.
Assume that $g\in L^\infty(\rn)$ and $a$ is a $(p,2,d)$-atom as in Definition
\ref{defa}
supported in a ball $B\subset\rn$. Then
$$\|aP_B^s g\|_{H^{\Phi_p}(\rn)}\le C \|g\|_{L^\infty(\rn)},$$
where $P_B^sg$ for any ball $B\subset\rn$ denotes the minimizing polynomial of $g$ on $B$ with degree
not greater than $s$, and the positive constant $C$ is independent of $a$ and $g$.
\end{lemma}

\begin{proof}
Let $p\in (0,\,1)$, $s:=\lfloor n(1/p-1)\rfloor$
and $g\in L^\infty(\rn)$. Let $d\in\zz_+\cap [2s,\,\fz)$, $a$ be a $(p,2,d)$-atom as in Definition
\ref{defa}
supported in a ball $B\subset\rn$.
Without loss of generality, we may assume that $\|g\|_{L^\infty(\rn)}=1$.
By the definition of $H^{\Phi_p}(\rn)$
and \eqref{333}, to show $\|aP_B^s g\|_{H^{\Phi_p}(\rn)}\lesssim1$,
it suffices to prove that there exists a positive constant $c$, independent of $a$
and $g$, such that
\begin{equation}\label{1wW}
\int_{\rn}\Phi_p\lf(\frac{M(a P_B^s g,\varphi)(x)}{c}\r)\,dx=
\int_{\rn}\frac{M(a P_B^s g,\varphi)(x)/c}{1+[M(a P_B^s g,\varphi)(x)/c]^{1-p}}\,dx
\leq1,
\end{equation}
where
$M(\cdot,\varphi)$ is as in Definition \ref{Hp}(i).
Moreover, from some elementary calculations, we deduce that \eqref{1wW} holds true provided that
\begin{equation}\label{3wW}
\int_{\rn}[M(a P_B^s g,\varphi)(x)]^p\,dx
\leq c^p.
\end{equation}
Next, we show \eqref{3wW}.
By \eqref{2w} and the assumption that $a$ is a $(p,2,d)$-atom supported in the ball
$B$, we
find that $\widetilde{a}:= \frac{1}{\widetilde{C}}a P_B^s g$ is a $(p,2,d-s)$-atom for some positive constant $\widetilde{C}$.
From this, $d-s\ge\lfloor n(1/p-1)\rfloor$ and
the atomic characterization of $H^{p}(\rn)$ (see \cite[p.\,21, Theorem 1.1]{Lu95}),
it follows that
$$
\int_{\rn}[M(a P_B^s g,\varphi)(x)]^p\,dx=
\|a P_B^s g\|_{H^{p}(\rn)}\lesssim\|\widetilde{a}\|_{H^{p}(\rn)}\lesssim1,
$$
which implies that \eqref{3wW} holds true and hence \eqref{1wW} holds true.
This finishes the proof of Lemma \ref{apW}.
\end{proof}

The next lemma gives a finite atomic decomposition of elements in the
Hardy space $H^p(\rn)$ that have finite wavelet expansions, which is just \cite[Lemma 4,5]{bckly}
(see also \cite[ Theorem~5.12 of Section~6.5]{HW96} and \cite[Section~6.8]{HW96}).

\begin{lemma}\label{-hp}
Let $p\in(0,\,1]$, $s\in\zz_+$ with $s\ge \lfloor n(1/p-1)\rfloor$
and $d\in\nn$ with $
d>\lfloor n(1/p-1)\rfloor$.
Let $\{\psi_I^{\lz}\}_{I\in\mathcal{D},\,\lz\in E}$ be the wavelets
as in \eqref{www}
with the regularity parameter $d$.
Assume that $f\in H^p(\rn)$ has a finite wavelet expansion, namely,
\begin{equation*}
f=
\sum_{I\in \mathcal{D}}\sum_{\lz\in E} \langle f,\,\psi_I^\lz \rangle \psi_I^\lz,
\end{equation*}
where
the coefficients $\langle f,\,\psi_I^\lz \rangle$ in
\eqref{eqn 3.7} have only
finite non-zero terms.
Then $f$ has a finite atomic decomposition satisfying
$f=\sum_{l=1}^L \mu_l a_l$, where $L\in\nn$, $\mu_l\in\mathbb{C}$ for any $l\in\{1,\,\ldots,\,L\}$,
and the following properties hold true:
\begin{itemize}
\item[\rm (i)] for any $l\in\{1,\,\ldots,\,L\}$,
$a_l$ is a $(p,2,s)$-atom
supported in some cube $mR_l$ with $m$ as in \eqref{220} and $R_l\in\mathcal{D}$,
which can be written into the following form:
\begin{equation*}
a_l=
\sum_{I\in \mathcal{D}, {I\subset R_l}}\sum_{\lz\in E} c_{(I,\,\lz,\,l)} \psi_I^\lz,
\end{equation*}
where $\{c_{(I,\,\lz,\,l)}\}_{I\subset R_l,\,\lz\in E}\subset\rr$;

\item[\rm (ii)] there exists a positive constant $ C$,
independent of $\{\mu_l\}_{l=1}^L$, $\{a_l\}_{l=1}^L$
and $f$, such that
$$\lf\{\sum_{l=1}^L |\mu_l|^p\r\}^{\frac{1}{p}}\le  C \|f\|_{H^p(\rn)};$$

\item[\rm (iii)] for any $l\in\{1,\,\ldots,\,L\}$,
the wavelet expansion of $a_l$ in \eqref{eqn 3.x4} is also finite which is
extracted from that of $f$ in \eqref{eqn 3.7}.
\end{itemize}
\end{lemma}

\begin{proposition}\label{ppi2}
Let $p\in(0,1)$ and $\az:=1/p-1$.
Assume that the regularity parameter $d\in\nn$ of wavelets satisfies that
$d>\lfloor n(1/p-1)\rfloor$. Then the bilinear operator $\Pi_2'$, defined as in \eqref{pi2},
can be extended to a bilinear operator
bounded from $H^p(\rn)\times\Lambda_{n\alpha}(\rn)$ to $H^{\Phi_p}(\rn)$ with $\Phi_p$
as in \eqref{333}.
\end{proposition}

\begin{remark}\label{612}
Let $p\in(0,1)$. Compared with
the corresponding target space appeared in
\cite[Proposition 4.9]{bckly},
the target space $H^{\Phi_p}(\rn)$ with $\Phi_p$
as in \eqref{333}
in Proposition \ref{ppi2} is
smaller than the corresponding one $H^{\phi_p}(\rn)$ with
$\phi_p$ as in \eqref{356} in \cite{bckly}, namely, $H^{\Phi_p}(\rn)\subset H^{\phi_p}(\rn)$
which can be deduced directly from their definitions together with the obvious fact that
$\Phi_p\le\phi_p$. The key reason for this being true is that, for any $\alpha:=1/p-1$ with $p\in(0,1)$, $\Lambda_{n\az}(\rn)\subsetneqq\mathfrak{C}_{\alpha}(\mathbb{R}^n)$.
\end{remark}

\begin{proof}[Proof of Proposition \ref{ppi2}]
Let $g\in \Lambda_{n\alpha}(\mathbb{R}^n)$  and $s:=\lfloor n\az\rfloor$.  Assume that $a$ is a
$(p,2,2s)$-atom supported in the cube
$mR$ with $m$ as in \eqref{220} and $R\in\mathcal{D}$ as in Lemma \ref{-hp}, and $a$ has a finite  wavelet expansion.
Let $B$ be  the smallest ball in $\rn$ containing $9mR$ and
$P_B^s g$ the minimizing polynomial of $g$ on $B$ with degree $\le s$
as in Definition \ref{campa}.
Let $\eta$ be a smooth cut-off function such that $\supp\eta\subset 9m R$ and $\eta\equiv 1$ on $5mR$.

Applying Remark \ref{rw3} and property (P4) in Section \ref{s4} with $d>\lfloor n(1/p-1)\rfloor$,
together with the expressions \eqref{pi1} through \eqref{pi4},
we know that
\begin{align*}
aP_{B}^sg=
a(\eta P_{B}^sg)
=\dsum_{i=1}^4\Pi_i'(a,\,\eta P_{B}^s g)
=\Pi_2'(a,\,\eta P_{B}^s g)=\Pi_2'(a,\, P_{B}^s g).
\end{align*}
Again, by Remark \ref{rw3},
we write
\begin{align*}
\Pi_2'(a,\,g)= \Pi_2'(a,\,\eta g)
=\Pi_2'(a,\,\eta [g-P_{B}^sg])+ \Pi_2'(a, \eta  P_{B}^sg)=\Pi_2'(a,\,\eta [g-P_{B}^sg])+aP_{B}^sg.
\end{align*}
Notice that the expression of the function $\Phi_p$ obviously implies that
$H^1(\rn)\subset H^{\Phi_p}(\rn)$.
Moreover, by  the fact that $\Pi_2'$ is bounded from $L^2(\rn)\times L^2(\rn)$
to $H^1(\rn)$ (see \cite[Lemma 4.2]{bckly}), $a$ is a $(p,2,2s)$-atom, $\supp\eta\subset 9mR$ and  Lemma \ref{GR-DDY}, we conclude that
\begin{align*}
\|\Pi_2'(a,\,\eta [g-P_{B}^sg])\|_{H^{\Phi_p}(\rn)}
&\ls \|\Pi_2'(a,\,\eta [g-P_{B}^sg])\|_{H^1(\rn)}\\
&\ls\|a\|_{L^2(\rn)}\|\eta(g-P_{B}^sg)\|_{L^2(\rn)}
\ls \|g\|_{\Lambda_{n\az}(\rn)}.
\end{align*}
From this and Lemma \ref{apW}, we deduce that
\begin{align*}
\|\Pi_2'(a,\,g)\|_{H^{\Phi_p}(\rn)}&
\ls \|\Pi_2'(a,\,g-P_{B}^sg)\|_{H^{\Phi_p}(\rn)}+\|aP_{B}^sg\|_{H^{\Phi_p}(\rn)}\ls \|g\|_{\Lambda_{n\az}(\rn)}.
\end{align*}
Repeating the
argument similar to that used in
the proof of Proposition \ref{prop-pi2}, we find that the definition of
$\Pi_2'(f,g)$ can be extended to any $f\in H^p(\rn)$ and
$g\in \Lambda_{n\alpha}(\mathbb{R}^n)$ with
the desired boundedness and the details are omitted.
This finishes the proof of Proposition \ref{ppi2}.
\end{proof}

\begin{definition}\label{defn}
Let $p\in(0,1)$, $d:=\lfloor n(1/p-1)\rfloor$
and $\Phi_p$ be as in \eqref{333}.
The {\em Orlicz Campanato space}
$\mathcal{L}^{\Phi_p}(\rn)$ is defined to be the set
of all locally integrable functions $g$ on $\rn$ such that
\begin{align*}
\|g\|_{\mathcal{L}^{\Phi_p}(\rn)}:&=
\dsup_{\mathrm{ball\,}B\subset\rn} \frac1{\|\mathbf1_B\|_{L^{\Phi_p}(\rn)}} \int_B |g(x)-P_B^d g(x)|\,dx
<\infty,
\end{align*}
where $P_B^dg$ for any ball $B\subset\rn$ denotes the minimizing polynomial of $g$ on $B$ with degree
not greater than $d$.
\end{definition}

By Remark \ref{rem-add1} and \cite[Theorem 5.2.1]{ylk},
we have the following duality result and
we omit the details.

\begin{lemma}\label{lemdu}
Let $p\in(0,1)$ and $\Phi_p$ be as in \eqref{333}.
Then
$(H^{\Phi_p}(\rn))^*=\mathcal{L}^{\Phi_p}(\rn).$
\end{lemma}

We have the following relations on
$(H^{\Phi_p}(\rn))^\ast$ and $(h^{\Phi_p}(\rn))^\ast$ with
$p\in(0,1)$ and $\Phi_p$ as in \eqref{333}.

\begin{lemma}\label{asz}
Let $p\in(0,1)$ and
$\Phi_p$ be as in \eqref{333}.
Then $[L^\infty(\rn)\cap (H^{\Phi_p}(\rn))^\ast]=[L^\infty(\rn)\cap (h^{\Phi_p}(\rn))^\ast]$.
\end{lemma}

\begin{proof}
Let $d:=\lfloor n(1/p-1)\rfloor$ and $\alpha:=1/p-1$.
From Remark \ref{rem-add1} and Lemma \ref{lemdu}, we deduce that
$(h^{\Phi_{p}}(\rn))^*=\mathcal{L}_{\loc}^{\Phi_{p}}(\rn)$
and $(H^{\Phi_{p}}(\rn))^*=\mathcal{L}^{\Phi_{p}}(\rn)$.
Thus, to show $[L^\infty(\rn)\cap (H^{\Phi_p}(\rn))^\ast]=[L^\infty(\rn)\cap (h^{\Phi_p}(\rn))^\ast]$,
it suffices to prove that
$[L^\infty(\rn)\cap \mathcal{L}^{\Phi_{p}}(\rn)]=[L^\infty(\rn)\cap \mathcal{L}_{\loc}^{\Phi_{p}}(\rn)]$.
From Lemma \ref{lem-ep} and
the definitions of $\mathcal{L}^{\Phi_{p}}(\rn)$ and $\mathcal{L}_{\loc}^{\Phi_{p}}(\rn)$,
it easily follows that
$[L^\infty(\rn)\cap \mathcal{L}_{\loc}^{\Phi_{p}}(\rn)]\subset [L^\infty(\rn)\cap \mathcal{L}^{\Phi_{p}}(\rn)]$.
Now we show
\begin{align}\label{ppg}
[L^\infty(\rn)\cap \mathcal{L}^{\Phi_{p}}(\rn)]\subset [L^\infty(\rn)\cap \mathcal{L}_{\loc}^{\Phi_{p}}(\rn)].
\end{align}
From Lemma \ref{444}(i) and
the definitions of $\mathcal{L}^{\Phi_{p}}(\rn)$ and $\mathcal{L}_{\loc}^{\Phi_{p}}(\rn)$, it follows that,
for any $f\in L^\infty(\rn)\cap \mathcal{L}^{\Phi_{p}}(\rn)$,
\begin{align*}
\|f\|_{\mathcal{L}_{\loc}^{\Phi_{p}}(\rn)}&\sim
\dsup_{\mathrm{ball\,}B\subset\rn,|B|<1} \frac1{|B|^{1/p}} \int_B |f(x)-P_B^d f(x)|\,dx\\\noz
&\quad+\dsup_{\mathrm{ball\,}B\subset\rn,|B|\geq1} \frac1{|B|} \int_B |f(x)|\,dx
\lesssim \|f\|_{\mathcal{L}^{\Phi_{p}}(\rn)}+\|f\|_{L^\infty(\rn)}<\infty,
\end{align*}
where $P_B^df$ for any ball $B\subset\rn$ denotes the minimizing polynomial of $g$ on $B$ with degree
not greater than $d$, which implies that $f\in L^\infty(\rn)\cap \mathcal{L}_{\loc}^{\Phi_{p}}(\rn)$
and hence \eqref{ppg} holds true.
This finishes the proof of
Lemma \ref{asz}.
\end{proof}

Using Theorem \ref{thm-main2}, Lemmas \ref{lemdu} and \ref{asz},
Propositions \ref{ppi1}, \ref{ppi3}, \ref{ppi4} and \ref{ppi2},
we obtain the following
bilinear decomposition theorem, which
plays a key role in the proof of Theorem \ref{div} below.

\begin{theorem}\label{mainthm4}
Let $p\in(0,1)$, $\az=1/p-1$
and $\Phi_p$ be as in \eqref{333}. Then the following statements
hold true.
\begin{itemize}
\item[\rm (i)] There exist two bounded bilinear operators
$$S:\, H^p(\rn)\times \Lambda_{n\alpha}(\mathbb{R}^n)\to L^1(\rn)$$ and
$$T:\, H^p(\rn)\times \Lambda_{n\alpha}(\mathbb{R}^n)\to H^{\Phi_p}(\rn)$$
such that, for any $(f,g)\in H^p(\rn)\times \Lambda_{n\alpha}(\mathbb{R}^n)$,
\begin{align*}
f\times g=S(f,\,g)+T(f,\,g)\qquad in\;\cs'(\rn).
\end{align*}
Moreover, there exists a positive constant $C$ such that, for any $(f,g)\in H^p(\rn)\times \Lambda_{n\alpha}(\rn)$,
$$\lf\|S(f,g)\r\|_{L^1(\rn)}\le C \|f\|_{H^p(\rn)}\|g\|_{\Lambda_{n\alpha}(\mathbb{R}^n)}$$
and
$$
\lf\|T(f,g)\r\|_{H^{\Phi_p}(\rn)}\le C \|f\|_{H^p(\rn)}\|g\|_{\Lambda_{n\alpha}(\mathbb{R}^n)}.$$
\item[\rm (ii)] The pointwise
multiplier class of $\Lambda_{n\alpha}(\rn)$ equals to $L^\infty(\rn)\cap (H^{\Phi_{p}}(\rn))^*$.
\end{itemize}
\end{theorem}

\begin{proof}
From Theorem \ref{thm-main2}, Lemmas \ref{lemdu} and \ref{asz}, it follows that (ii) holds true.

As for (i), for any $(f,g)\in H^p(\rn)\times \Lambda_{n\alpha}(\mathbb{R}^n)$, let
\begin{align}\label{So}
S(f,\,g):=\Pi_4'(f,\,g)
\end{align}
and
\begin{align*}
T(f,\,g):=\sum_{i=1}^3\Pi_i'(f,\,g).
\end{align*}
Using
Propositions \ref{ppi1}, \ref{ppi3}, \ref{ppi4} and \ref{ppi2}, and repeating the
argument similar to that used in
the proof of Theorem \ref{mainthm1}, we obtain the desired conclusion of this theorem. This finishes
the proof of Theorem \ref{mainthm4}.
\end{proof}

\begin{remark}\label{432}
Let $p\in(0,1)$, $\az:=1/p-1$ and $\Phi_p$ be as in \eqref{333}
\begin{itemize}
\item[\rm (i)]
Let $\phi_p$ be as in \eqref{356}.
In Theorem \ref{mainthm4}(i), if we replace $\Lambda_{n\alpha}(\mathbb{R}^n)$ and $H^{\Phi_p}(\rn)$,
respectively, by $\mathfrak{C}_{\alpha}(\mathbb{R}^n)$ and $H^{\phi_p}(\rn)$, then
the corresponding conclusions of Theorem \ref{mainthm4}(i)
still hold true, which is just \cite[Theorem 1.2]{bckly}.
Observe that, by Remark \ref{612},
we have
$\Lambda_{n\alpha}(\mathbb{R}^n)\subsetneqq
\mathfrak{C}_{\alpha}(\mathbb{R}^n)$
and
$H^{\Phi_p}(\rn)\subset
H^{\phi_p}(\rn)$.
Thus, Theorem \ref{mainthm4} is of independent interest.
However, it is still unclear whether or not
$H^{\Phi_p}(\rn)$ is a proper subspace of
$H^{\phi_p}(\rn)$.

\item[\rm (ii)] By Theorem \ref{mainthm1},
\cite[Theorem 1.2]{bckly} and Theorem \ref{mainthm4}, we have bilinear decompositions of $h^p(\rn)\times \Lambda_{n\alpha}(\mathbb{R}^n)$,
$H^p(\rn)\times \mathfrak{C}_{\alpha}(\mathbb{R}^n)$ and $H^p(\rn)\times \Lambda_{n\alpha}(\rn)$.
Motivated by these results, to complete the whole story,
it is quite natural to conjecture that $h^p(\rn)\times \mathfrak{C}_{\alpha}(\mathbb{R}^n)$ should also have
a bilinear decomposition.
However, the present definition \eqref{def-product}
of the product does not make sense for the product of
elements in $h^p(\rn)$ and $\mathfrak{C}_{\alpha}(\mathbb{R}^n)$
because the dual space of
$h^p(\rn)$ is smaller than $\mathfrak{C}_{\alpha}(\mathbb{R}^n)$, namely,
$(h^p(\rn))^*=\Lambda_{n\alpha}(\rn)
\subsetneqq\mathfrak{C}_{\alpha}(\mathbb{R}^n)$.
\item[\rm (iii)] The sharpness of  Theorem \ref{mainthm4} is implied by Theorem \ref{mainthm4}(ii). Indeed,
suppose that Theorem \ref{mainthm4} holds true with $H^{\Phi_p}(\rn)$ therein
replaced by a smaller quasi-Banach space $\cy$.
In this case, we know that the pointwise
multiplier class of $\Lambda_{n\alpha}(\rn)$ equals to $L^\infty(\rn)\cap (\cy)^\ast$
From this and Theorem \ref{mainthm4}(ii), we deduce that
$L^\infty(\rn)\cap (\cy)^\ast=L^\infty(\rn)\cap (H^{\Phi_p}(\rn))^\ast$.
In this sense,
we say that Theorem \ref{mainthm4} is sharp. It
is still unclear whether or not the Orlicz Hardy
space $H^{\Phi_p}(\rn)$ is indeed the smallest space,
in the sense of the inclusion of sets, having the property as
in Theorem \ref{mainthm4}.
\end{itemize}
\end{remark}

\subsection{Bilinear decomposition of $H^1(\rn)\times \bmo(\rn)$}\label{s4.3}

In this subsection, we establish the bilinear decomposition of the
product space $H^1(\rn)\times \bmo(\rn)$.
Now, we recall the notions of both the local Hardy-type space $h_*^\Phi(\rn)$ and the Hardy-type space $H_*^\Phi(\rn)$
introduced in \cite{bf}.

\begin{definition}\label{cjs}
For any $\tau\in[0,\infty)$, let $\Phi$ be as in \eqref{111}.
\begin{itemize}
\item[(i)] The \emph{variant Orlicz space} $L_*^\Phi(\rn)$ is defined to be the set of all measurable functions $f$
such that $$\|f\|_{L_*^\Phi(\rn)}
:=\sum_{k\in\zz^n}\|f\mathbf1_{Q_k}\|_{L^{\Phi}(\rn)}<\infty,$$
where $Q_{k}:=k+[0,1)^n$ for any $k\in\mathbb{Z}^n$.
\item[(ii)] The \emph{variant Orlicz Hardy space} $H_*^\Phi(\rn)$ is defined
by setting
 $$H_*^\Phi(\rn):=\lf\{f\in\mathcal{S}'(\rn):\ \|f\|_{H_*^\Phi(\rn)}:=
\|M(f,\varphi)\|_{L_*^\Phi(\rn)}<\infty \r\},$$
where $\varphi\in\mathcal{S}(\rn)$ satisfies
$\int_{\rn}\varphi(x)\,dx\neq0.$
\item[(iii)] The \emph{variant local Orlicz Hardy space} $h_*^\Phi(\rn)$ is defined
by setting
 $$h_*^\Phi(\rn):=\lf\{f\in\mathcal{S}'(\rn):\ \|f\|_{h_*^\Phi(\rn)}:=
\|m(f,\varphi)\|_{L_*^\Phi(\rn)}<\infty \r\},$$
where $\varphi\in\mathcal{S}(\rn)$ satisfies $\int_{\rn}\varphi(x)\,dx\neq0$.
\item[(iv)]
The \emph{local $\BMO$-type space ${\bmo}^{\Phi}(\rn)$} is defined
to be the set of all measurable functions $f\in L_{\loc}^1(\rn)$ such that
\begin{align*}
\|f\|_{{\bmo}^{\Phi}(\rn)}:&=
\sup_{\mathrm{ball\,}B\subset\rn,|B|<1}
\frac{\log(e+\frac{1}{|B|})}{|B|}\int_{B}|f(x)-f_B|\,dx\\\noz
&\quad+\sup_{\mathrm{ball\,}B\subset\rn,|B|\geq1}
\frac{\log(e+\frac{1}{|B|})}{|B|}
\int_{B}|f(x)|\,dx
<\infty,
\end{align*}
where $f_B:=\frac{1}{|B|}\int_{B}f(x)\,dx$ for any ball $B\subset\rn$.
\item[(v)]
The \emph{$\BMO$-type space ${\BMO}^{\Phi}(\rn)$} is defined
to be the set of all measurable functions $f\in L_{\loc}^1(\rn)$ such that
\begin{align*}
\|f\|_{{\BMO}^{\Phi}(\rn)}:&=
\sup_{\mathrm{ball\,}B\subset\rn}
\frac{\log(e+\frac{1}{|B|})}{|B|}\int_{B}|f(x)-f_B|\,dx,
\end{align*}
where $f_B:=\frac{1}{|B|}\int_{B}f(x)\,dx$ for any ball $B\subset\rn$.
\end{itemize}
\end{definition}

From \cite[Theorems 5.2 and 5.3]{BGK12} and the proof of \cite[Proposition 6.1]{BGK12},
together with the fact that $\bmo(\rn)\subset\BMO(\rn)$,
we deduce the following conclusions of $\Pi_1'$, $\Pi_3'$ and $\Pi_4'$.

\begin{proposition}\label{pp11}
Assume that the regularity parameter $d\in\nn$ of wavelets satisfies that
$d\geq1$.
Then the bilinear operator $\Pi_1'$, defined as in \eqref{pi1},
can be extended to a bilinear operator
bounded from $H^1(\rn)\times\bmo(\rn)$ to $H^1(\rn)$.
\end{proposition}

\begin{proposition}\label{pp33}
Assume that the regularity parameter $d\in\nn$ of wavelets satisfies that
$d\geq1$. Then the bilinear operator $\Pi_3'$, defined as in \eqref{pi3},
can be extended to a bilinear operator
bounded from $H^1(\rn)\times\bmo(\rn)$ to $H^1(\rn)$.
\end{proposition}

\begin{proposition}\label{pp44}
Assume that the regularity parameter $d\in\nn$ of wavelets satisfies that
$d\geq1$. Then
the bilinear operator $\Pi_4'$, defined as in \eqref{pi4},
can be extended to a bilinear operator
bounded from $H^1(\rn)\times\bmo(\rn)$ to $L^1(\rn)$.
\end{proposition}

To obtain the estimate of $\Pi_2'$, we need the following technical lemma.

\begin{lemma}\label{apW11}
Let $d\in\zz_+$ satisfy that
$d\geq0$ and $\Phi$ be as in \eqref{111}.
Assume that $g\in \bmo(\rn)$ and $a$ is a $(1,2,d)$-atom as in Definition
\ref{defa}
supported in a ball $B\subset\rn$. Then
$$\|ag_B\|_{H_*^\Phi(\rn)}\le C \|g\|_{\bmo(\rn)},$$
where $g_B:=\frac{1}{|B|}
\int_{B}g(x)\,dx$ and the positive constant $C$ is independent of $a$ and $g$.
\end{lemma}

\begin{proof}
Let $g\in \bmo(\rn)$, $d\in\zz_+\cap [0,\,\fz)$, $a$ be a $(1,2,d)$-atom as in Definition
\ref{defa}
supported in a ball $B\subset\rn$.
By the definition of $H_*^\Phi(\rn)$,
we conclude that
\begin{align}\label{1zwW}
\lf\|a g_B\r\|_{H_*^\Phi(\rn)}&\le\lf\| |g_B| M(a,\varphi)\r\|_{L_*^\Phi(\rn)}\lesssim
\lf\| |g-g_B| M(a,\varphi)\r\|_{L_*^\Phi(\rn)}+\lf\| |g| M(a,\varphi)\r\|_{L_*^\Phi(\rn)}\\\noz
&=: \mathrm{I}_1+\mathrm{I}_2,
\end{align}
where $\varphi\in\mathcal{S}(\rn)$ satisfies
$\int_{\rn}\varphi(x)\,dx\neq0$ and
$M(\cdot,\varphi)$ is as in Definition \ref{Hp}(i).

For $\mathrm{I}_1$, using the fact that $a$ is a $(1,2,d)$-atom, $L^1(\rn)\subset L_*^\Phi(\rn)$ and \cite[Lemma 5.2]{BGK12},
we know that
\begin{align}\label{2zwW}
\mathrm{I}_1=
\lf\||g-g_B| M(a,\varphi)\r\|_{L_*^\Phi(\rn)}\le\lf\||g-g_B| M(a,\varphi)\r\|_{L^1(\rn)}\lesssim\|g\|_{\bmo(\rn)},
\end{align}
which is the desired estimate.

As for $\mathrm{I}_2$, by the estimate in \cite[p.\,25]{cky1}, we know that, for any
$f\in L^1(\rn)$ and $h\in\bmo(\rn)$,
\begin{align*}
\lf\|fh\r\|_{L_*^\Phi(\rn)}\lesssim\lf\|f\r\|_{L^1(\rn)}\|h\|_{\bmo(\rn)}.
\end{align*}
From this, the atomic characterization of $H^{1}(\rn)$ (see \cite[p.\,21, Theorem 1.1]{Lu95}) and the fact that $a$ is a $(1,2,d)$-atom, it further follows that
\begin{align*}
\mathrm{I}_2=
\lf\||g| M(a,\varphi)\r\|_{L_*^\Phi(\rn)}\le\lf\|M(a,\varphi)\r\|_{L^1(\rn)}
\|g\|_{\bmo(\rn)}\lesssim\|g\|_{\bmo(\rn)},
\end{align*}
which, combined with \eqref{1zwW} and \eqref{2zwW}, implies that
$\|ag_B\|_{H_*^\Phi(\rn)}\lesssim \|g\|_{\bmo(\rn)}$ holds true.
This finishes the proof of Lemma \ref{apW11}.
\end{proof}

Now we establish the following estimate of $\Pi_2'$.
\begin{proposition}\label{pp22}
Let $\Phi$ be as in \eqref{111}.
Assume that the regularity parameter $d\in\nn$ of wavelets satisfies
$d\geq1$. Then the bilinear operator $\Pi_2'$, defined as in \eqref{pi2},
can be extended to a bilinear operator
bounded from $H^1(\rn)\times\bmo(\rn)$ to $H_*^\Phi(\rn)$.
\end{proposition}

\begin{proof}
Let $g\in \bmo(\mathbb{R}^n)$.  Assume that $a$ is a
$(1,2,0)$-atom supported in the cube
$mR$ with $m$ as in \eqref{220} and $R\in\mathcal{D}$ as in Lemma \ref{-hp}, and $a$ has a finite  wavelet expansion.
Let $B$ be the smallest ball in $\rn$ containing $9mR$ and
$g_B:=\frac{1}{|B|}
\int_{B}g(x)\,dx$.
Let $\eta$ be a smooth cut-off function such that $\supp\eta\subset 9m R$ and $\eta\equiv 1$ on $5mR$.

Applying Remark \ref{rw3} and property (P4) in Section \ref{s4} with $d\geq1$,
together with the expressions \eqref{pi1} through \eqref{pi4},
we know that
\begin{align*}
ag_B=
a(\eta g_B)
=\dsum_{i=1}^4\Pi_i'(a,\,\eta g_B)
=\Pi_2'(a,\,\eta g_B)=\Pi_2'(a,\, g_B).
\end{align*}
Again, by Remark \ref{rw3},
we write
\begin{align*}
\Pi_2'(a,\,g)= \Pi_2'(a,\,\eta g)
=\Pi_2'(a,\,\eta [g-g_B])+ \Pi_2'(a, \eta  g_B)=\Pi_2'(a,\,\eta [g-g_B])+ag_B.
\end{align*}
Notice that the definition of the space $L_*^\Phi(\rn)$ easily implies that
$H^1(\rn)\subset H_*^\Phi(\rn)$.
Moreover, by the fact that $\Pi_2$ is bounded from $L^2(\rn)\times L^2(\rn)$
to $H^1(\rn)$ (see Lemma \ref{rem-w1}), $a$ is a $(1,2,0)$-atom and $\supp\eta\subset 9mR$, we conclude that
\begin{align*}
\|\Pi_2'(a,\,\eta [g-g_B])\|_{H_*^\Phi(\rn)}
&\ls \|\Pi_2'(a,\,\eta [g-g_B])\|_{H^1(\rn)}\\
&\ls\|a\|_{L^2(\rn)}\|\eta(g-g_B)\|_{L^2(\rn)}
\ls \|g\|_{\bmo(\rn)}.
\end{align*}
From this and Lemma \ref{apW11}, we deduce that
\begin{align*}
\|\Pi_2'(a,\,g)\|_{H_*^\Phi(\rn)}&
\ls \|\Pi_2'(a,\,g-g_B)\|_{H_*^\Phi(\rn)}+\|ag_B\|_{H_*^\Phi(\rn)}\ls \|g\|_{\bmo(\rn)}.
\end{align*}
Repeating the
argument similar to that used in
the proof of Proposition \ref{prop-pi1}, we find that the definition of
$\Pi_2'(f,g)$ can be extended to any $f\in H^1(\rn)$ and
$g\in \bmo(\mathbb{R}^n)$ with
the desired boundedness and the details are omitted.
This finishes the proof of Proposition \ref{pp22}.
\end{proof}

By \cite[Lemma 3.4]{zyy} and \cite[Theorem 5.7]{zyyw},
we have the following duality result and
we omit the details.

\begin{lemma}\label{lemdu2}
Let $\Phi$ be as in \eqref{111}.
Then
$(H_\ast^{\Phi}(\rn))^*={\BMO}^{\Phi}(\rn).$
\end{lemma}

\begin{lemma}\label{asz2}
Let $\Phi$ be as in \eqref{111}.
Then $[L^\infty(\rn)\cap (H_\ast^{\Phi}(\rn))^\ast]=[L^\infty(\rn)\cap {\bmo}^{\Phi}(\rn)]$.
\end{lemma}

\begin{proof}
From Lemma \ref{lemdu2}, we deduce that
$(H_\ast^{\Phi}(\rn))^*={\BMO}^{\Phi}(\rn)$.
Thus, to show $[L^\infty(\rn)\cap (H_\ast^{\Phi}(\rn))^\ast]=[L^\infty(\rn)\cap {\bmo}^{\Phi}(\rn)]$,
it suffices to prove that
$[L^\infty(\rn)\cap {\BMO}^{\Phi}(\rn)]=[L^\infty(\rn)\cap {\bmo}^{\Phi}(\rn)]$.
From
the definitions of ${\BMO}^{\Phi}(\rn)$ and ${\bmo}^{\Phi}(\rn)$,
it easily follows that
$[L^\infty(\rn)\cap {\bmo}^{\Phi}(\rn)]\subset [L^\infty(\rn)\cap {\BMO}^{\Phi}(\rn)]$.
Now we show
\begin{align}\label{ppg2}
[L^\infty(\rn)\cap {\BMO}^{\Phi}(\rn)]\subset [L^\infty(\rn)\cap {\bmo}^{\Phi}(\rn)].
\end{align}
From
the definitions of ${\BMO}^{\Phi}(\rn)$ and ${\bmo}^{\Phi}(\rn)$, it follows that,
for any $f\in L^\infty(\rn)\cap {\BMO}^{\Phi}(\rn)$,
\begin{align*}
\|f\|_{{\bmo}^{\Phi}(\rn)}&=
\sup_{\mathrm{ball\,}B\subset\rn,|B|<1}
\frac{\log(e+\frac{1}{|B|})}{|B|}\int_{B}|f(x)-f_B|\,dx\\\noz
&\quad+\sup_{\mathrm{ball\,}B\subset\rn,|B|\geq1}
\frac{\log(e+\frac{1}{|B|})}{|B|}
\int_{B}|f(x)|\,dx
\lesssim \|f\|_{{\BMO}^{\Phi}(\rn)}+\|f\|_{L^\infty(\rn)}<\infty,
\end{align*}
which implies that $f\in L^\infty(\rn)\cap {\bmo}^{\Phi}(\rn)$
and hence \eqref{ppg2} holds true.
This finishes the proof of
Lemma \ref{asz2}.
\end{proof}

\begin{theorem}\label{mainthm7}
Let $\Phi$ be as in \eqref{111}.
Then the following statements
hold true.
\begin{itemize}
\item[\rm (i)]
There exist two bounded bilinear operators
$$S:\, H^1(\rn)\times \bmo(\mathbb{R}^n)\to L^1(\rn)$$ and
$$T:\, H^1(\rn)\times \bmo(\mathbb{R}^n)\to H_\ast^{\Phi}(\rn)$$
such that, for any $(f,g)\in H^1(\rn)\times \bmo(\mathbb{R}^n)$,
\begin{align*}
f\times g=S(f,\,g)+T(f,\,g)\qquad in\;\cs'(\rn).
\end{align*}
Moreover, there exists a positive constant $C$ such that, for any $(f,g)\in H^1(\rn)\times \bmo(\rn)$,
$$\lf\|S(f,g)\r\|_{L^1(\rn)}\le C \|f\|_{H^1(\rn)}\|g\|_{\bmo(\mathbb{R}^n)}$$
and
$$
\lf\|T(f,g)\r\|_{H_\ast^{\Phi}(\rn)}\le C \|f\|_{H^1(\rn)}\|g\|_{\bmo(\mathbb{R}^n)}.$$
\item[\rm (ii)] The pointwise
multiplier class of $\bmo(\mathbb{R}^n)$ equals to $L^\infty(\rn)\cap (H_\ast^{\Phi}(\rn))^*$.
\end{itemize}
\end{theorem}

\begin{proof}
From \cite[Theorem 3.6(iii)]{zyy} and Lemma \ref{asz2}, it follows that (ii) holds true.

As for (i),
for any $(f,g)\in H^1(\rn)\times \bmo(\mathbb{R}^n)$, let
\begin{align}\label{So2}
S(f,\,g):=\Pi_4'(f,\,g)
\end{align}
and
\begin{align*}
T(f,\,g):=\sum_{i=1}^3\Pi_i'(f,\,g).
\end{align*}
Using Propositions \ref{pp11}, \ref{pp33}, \ref{pp44} and \ref{pp22}, and repeating some
arguments similar to those used in
the proof of Theorem \ref{mainthm1}, we obtain the desired conclusion of Theorem \ref{mainthm7}.
This finishes the proof of Theorem \ref{mainthm7}.
\end{proof}

\begin{remark}\label{442}
Let $p\in(0,1)$ and $\az:=1/p-1$.
\begin{itemize}
\item[\rm (i)]
In Theorem \ref{mainthm7}(i), if we replace $\bmo(\mathbb{R}^n)$ and $H_\ast^{\Phi}(\rn)$,
respectively, by $\BMO(\mathbb{R}^n)$ and $H^{\log}(\rn)$ as in \eqref{11z}, then
the corresponding conclusions of Theorem \ref{mainthm7}(i)
still hold true, which is just \cite[Theorem 1.1]{BGK12}.
Observe that, by \cite[Lemma 3.4]{zyy},
we have
$\bmo(\mathbb{R}^n)\subsetneqq
\BMO(\mathbb{R}^n)$
and
$H_\ast^{\Phi}(\rn)\subset
H^{\log}(\rn)$.
Thus, Theorem \ref{mainthm7} is of independent interest.
However, it is still unclear whether or not
$H_\ast^{\Phi}(\rn)$ is a proper subset of
$H^{\log}(\rn)$.

\item[\rm (ii)]
Also, in Theorem \ref{mainthm7}(i), if we replace $H^1(\mathbb{R}^n)$ and $H_\ast^{\Phi}(\rn)$,
respectively, by $h^1(\mathbb{R}^n)$ and $h_\ast^{\Phi}(\rn)$, then
the corresponding conclusions of Theorem \ref{mainthm7}(i)
still hold true, which is just \cite[Theorem 1.1(ii)]{cky1}.
This is quite natural by observing that $H_\ast^{\Phi}(\rn)\subsetneqq h_\ast^{\Phi}(\rn)$.
To see $H_\ast^{\Phi}(\rn)\subsetneqq h_\ast^{\Phi}(\rn)$, from the definitions of $H_\ast^{\Phi}(\rn)$ and $h_\ast^{\Phi}(\rn)$, it is easy to deduce that
$H_\ast^{\Phi}(\rn)\subset h_\ast^{\Phi}(\rn)$.
Using the atomic characterization of $h_\ast^{\Phi}(\rn)$
(see \cite[Theorem 2.18]{zyy}), we know
that $\mathbf{1}_{B(\vec{0}_n,2)}\in h_\ast^{\Phi}(\rn)$.
On the other hand, by an argument similar to that used in the estimation of \cite[(4.20)]{zyyw}, we conclude that, for any $f\in H_\ast^{\Phi}(\rn)\cap L^2(\rn)$
with compact supports, $\int_{\rn}f(x)\,dx=0$,
which implies that $\mathbf{1}_{B(\vec{0}_n,2)}\notin H_\ast^{\Phi}(\rn)$ and hence
$H_\ast^{\Phi}(\rn)\subsetneqq h_\ast^{\Phi}(\rn)$.

\item[\rm (iii)] By Theorem \ref{mainthm7},
\cite[Theorem 1.1]{BGK12} and \cite[Theorem 1.1]{cky1}, we have bilinear decompositions of $H^1(\rn)\times \bmo(\mathbb{R}^n)$,
$H^1(\rn)\times \BMO(\mathbb{R}^n)$ and $h^1(\rn)\times \bmo(\rn)$.
Motivated by these results, to complete the
whole picture, it is quite natural to conjecture that $h^1(\rn)\times \BMO(\mathbb{R}^n)$
should also have a bilinear decomposition. However,
the present definition \eqref{def-product} of the product
does not make sense for the product of functions in $h^1(\rn)$ and $\BMO(\mathbb{R}^n)$
because the dual space of
$h^1(\rn)$ is smaller than $\BMO(\mathbb{R}^n)$, namely,
$(h^1(\rn))^*=\bmo(\rn)
\subsetneqq\BMO(\mathbb{R}^n)$.
\item[\rm (iv)] The sharpness of  Theorem \ref{mainthm7} is implied by Theorem \ref{mainthm7}(ii). Indeed,
suppose that Theorem \ref{mainthm7} holds true with $H_\ast^{\Phi}(\rn)$ therein
replaced by a smaller quasi-Banach space $\cy$.
In this case, we know that the pointwise
multiplier class of $\bmo(\rn)$ equals to $L^\infty(\rn)\cap (\cy)^\ast$
From this and Theorem \ref{mainthm7}(ii), we deduce that
$L^\infty(\rn)\cap (\cy)^\ast=L^\infty(\rn)\cap (H_\ast^{\Phi}(\rn))^\ast$.
In this sense,
we say that Theorem \ref{mainthm7} is sharp. It
is still unclear whether or not the Orlicz Hardy
space $H_\ast^{\Phi}(\rn)$ is indeed the smallest space,
in the sense of the inclusion of sets, having the property
in Theorem \ref{mainthm7}.
\end{itemize}
\end{remark}

\section{Intrinsic structures of (local) Orlicz Hardy spaces $h^{\Phi_p}(\rn)$\\ and $H^{\Phi_p}(\rn)$
with $p\in (0,1)$}\label{s5}

In this section, we establish new structures of the spaces $h^{\Phi_p}(\rn)$ and $H^{\Phi_p}(\rn)$ by
showing that, for any $p\in(0,1)$, $h^{\Phi_p}(\rn)=h^{p}(\rn)+h^{1}(\rn)$ and
$H^{\Phi_p}(\rn)=H^{p}(\rn)+H^{1}(\rn)$ with
equivalent quasi-norms, and then further clarify the relations between $h^{\Phi_p}(\rn)$ and $h^{\phi_p}(\rn)$.

Recall that, in \cite{bl}, for any two quasi-Banach spaces $A_0$ and $A_1$,
the pair $(A_0,\,A_1)$ is said to be \emph{compatible} if there exists a Hausdorff
topological vector space $\mathbb{X}$ such that $A_0\subset\mathbb{X}$ and $A_1\subset\mathbb{X}$.
For any compatible pair $(A_0,\,A_1)$ of quasi-Banach spaces, the \emph{sum space}
$A_0+A_1$ is defined by setting
\begin{align}\label{1.8}
A_0+A_1:=\{a\in\mathbb{X}:\ \exists\, a_0\in A_0\ \textup{and}\ a_1\in A_1\
\textup{such}\ \textup{that}\ a=a_0+a_1\}
\end{align}
equipped with the quasi-norm
$$
\|a\|_{A_0+A_1}:=\inf\{\|a_0\|_{A_0}+\|a_1\|_{A_1}:\ a=a_0+a_1,\ a_0\in A_0\ \textup{and}\ a_1\in A_1\}.
$$
In what follows, we use $h^1(\rn)+h^p(\rn)$ to denote the sum space,
defined as in \eqref{1.8}, with $\mathbb{X}:=\mathcal{S}'(\rn)$, $A_0:=h^1(\rn)$
and $A_1:=h^p(\rn)$.

We first recall the following two lemmas, established in \cite{yy2},
on the Calder\'on--Zygmund decomposition of the elements of local
Musielak--Orlicz Hardy spaces.

The following lemma is a part of \cite[Corollary 4.8]{yy2}.

\begin{lemma}\label{den}
Let $p\in(0,1)$, $q\in(1,\infty)$ and $\Phi_p$ be as in \eqref{333}.
Then $L^q(\rn)\cap h^{\Phi_p}(\rn)$ is dense in $h^{\Phi_p}(\rn)$.
\end{lemma}

Define,
for any $N\in\mathbb{N}$
and $\varphi\in\mathcal{S}(\rn)$,
$$
p_N(\varphi):=\sum_{\alpha\in\mathbb{Z}^n_+,|\alpha|\le N} \sup_{x\in\rn}(1+|x|)^{N+n}|\partial^\alpha\varphi(x)|,
$$
and let $\mathcal{F}_N(\rn):=\{\varphi\in\mathcal{S}(\rn):\ \ p_N(\varphi)\le1\}$. Also recall that
$\mathbb{R}_+^{n+1}:=\rn\times(0,\infty)$.

\begin{definition}\label{jdhs}
Let $\varphi\in\mathcal{S}(\rn)$, $N\in\mathbb{N}$ and $f\in\mathcal{S}'(\rn)$.
\begin{itemize}
\item[{\rm(i)}] The \emph{grand maximal function $M_N(f)$} is defined by setting, for any $x\in\rn$,
$$M_N(f)(x):=\sup\lf\{|\varphi_s \ast f(y)|:\ s\in(0,\infty),\  |x-y|<s,\  \varphi\in\mathcal{F}_N(\rn)\r\};$$
\item[{\rm(ii)}] The \emph{local grand maximal function $m_N(f)$} is defined by setting, for any $x\in\rn$,
$$m_N(f)(x):=\sup\lf\{|\varphi_s \ast f(y)|:\ s\in(0,1),\  |x-y|<s,\  \varphi\in\mathcal{F}_N(\rn)\r\};$$
\end{itemize}
\end{definition}

By the proof of \cite[Lemma 5.5]{yy2}, we immediately obtain
the following Calder\'on--Zygmund decomposition of the elements of the local
Musielak--Orlicz Hardy space $h^{\Phi_p}(\rn)$ and we omit the details.

\begin{lemma}\label{dec}
Let $p\in(0,1)$, $q\in(1,\infty)$ and $\Phi_p$ be as in \eqref{333}. Assume that $d\in\mathbb{Z}_+$ satisfies $d\ge\lfloor n(\frac{1}{p}-1)\rfloor$.
Then, for any
$f\in h^{\Phi_p}(\rn)\cap L^q(\rn)$,
there exists a sequence $\{h_i^k\}_{k\in\zz,i\in\nn}$
of elements in $L^\infty(\rn)$ supported, respectively, in balls
$\{B_i^k\}_{k\in\zz,i\in\nn}$
such that
\begin{equation}\label{c2}
f=\sum_{k\in\zz,i\in\nn}h_i^k \qquad  \textup{in}\;\cs'(\rn)
\end{equation}
and, for any $k\in\zz$,
\begin{equation}\label{c5}
\mathcal{O}_k=\cup_{i\in\nn}(B_i^k)^\ast\ and\
\{(B_i^k)^\ast\}_{i\in\nn}\ has\ finite\ intersection\ property,
\end{equation}
where $\mathcal{O}_k:=\{x\in\rn:\ m_N(f)(x)>2^k\}$
with $N:=\lfloor \frac{2n}{p}+1\rfloor$, and $(B_i^k)^\ast:=cB_i^k$ with $c\in(0,1)$ being a constant.
Moreover, there exists a positive constant $C$ such that, for any $k\in\zz$ and $i\in\nn$,
\begin{equation}\label{c6}
|h_i^k|\le C2^k
\end{equation}
and, when $|B_i^k|<1$, for any multi-index $\alpha$ satisfying $|\alpha|\leq d$, it holds true
that
\begin{equation}\label{c8}
\int_{\rn}x^\alpha h_i^k(x)\,dx=0.
\end{equation}
\end{lemma}

\begin{definition}\label{17fa}
Let $p\in(0,\,1)$, $r\in(1,\,\infty]$ and $d\in\zz_+$.
Then a measurable function $a$ on $\rn$ is called a \emph{local} $(p,r,d)${\it-atom}
if there exists a ball $B\in\BB$ such that
\begin{itemize}
\item[(i)] $\supp a:=\{x\in\rn:\ a(x)\neq0\}\subset B$;
\item[(ii)] $\|a\|_{L^r(\rn)}\le |B|^{1/r-1/p}$;
\item[(iii)] if $|B|<1$, then $\int_{\rn}a(x)x^\alpha\,dx=0$ for any
$\alpha\in\zz_+^n$ with $|\alpha|\le d$.
\end{itemize}
\end{definition}

\begin{theorem}\label{FINN}
Let $p\in(0,1)$ and $\Phi_p$ be as in \eqref{333}. Then the space
$h^{\Phi_p}(\rn)$ and $h^1(\rn)+h^p(\rn)$ coincide with equivalent quasi-norms.
\end{theorem}
\begin{proof}
Let $p\in(0,1)$. We first establish the inclusion
$[h^1(\rn)+h^p(\rn)]\subset h^{\Phi_p}(\rn)$.
For any $f\in h^1(\rn)+h^p(\rn)$, let $f_0\in h^1(\rn)$ and
$f_1\in h^p(\rn)$ be such that $f=f_0+f_1$ in $\mathcal{S}'(\rn)$
and
$$
\|f_0+f_1\|_{h^1(\rn)+h^p(\rn)}\sim\|f_0\|_{h^1(\rn)}+\|f_1\|_{h^p(\rn)}.
$$
By Lemma \ref{444}(i), we know that
$$
\|f\|_{h^{\Phi_p}(\rn)}\lesssim\|f_0\|_{h^{\Phi_p}(\rn)}+\|f_1\|_{h^{\Phi_p}(\rn)}
\lesssim\|f_0\|_{h^1(\rn)}+\|f_1\|_{h^p(\rn)}
\sim\|f_0+f_1\|_{h^1(\rn)+h^p(\rn)}.
$$
This immediately implies the inclusion $[h^1(\rn)+h^p(\rn)]\subset h^{\Phi_p}(\rn)$.

It remains to prove $h^{\Phi_p}(\rn)\subset [h^1(\rn)+h^p(\rn)]$.
First, we let $q\in (1,\infty)$ and $f\in h^{\Phi_p}(\rn)\cap L^q(\rn)$.
Applying Lemma \ref{dec}, we know that there exists a sequence $\{h_i^k\}_{k\in\zz,i\in\nn}$
of elements in $L^\infty(\rn)$ supported, respectively, in balls
$\{B_i^k\}_{k\in\zz,i\in\nn}$
satisfying that,
for any $k\in\zz$,
$\sum_{i\in\nn}\mathbf{1}_{cB_{i,k}}\lesssim 1$ with $c\in(0,1)$ being a constant
such that
\begin{equation}\label{sll}
f=\sum_{k\in\zz}\sum_{i\in\nn}h_i^k \qquad  \textup{in}\;\cs'(\rn).
\end{equation}
Let
\begin{equation}\label{sdd}
E:=\lf\{x\in\rn:\ m_N(f)(x)<1 \r\},
\end{equation}
where $m(\cdot,\varphi)$ is as in Definition \ref{hp}(i).
For any $k\in\zz$ and $i\in\nn$, let
$$
B_{i,E}^k:=(cB_{i}^k)\cap E \ \ \ \ \text{and} \ \ \ \
B_{i,E^\complement}^k:=(cB_{i}^k)\cap E^\complement.
$$
Moreover, let
$$
\mathrm{I}_0:=\lf\{(k,i)\in\zz\times\nn:\ \lf|B_{i,E}^k\r|\geq\frac12\lf|cB_{i}^k\r|\r\}
 \ \ \ \ \text{and} \ \ \ \
\mathrm{I}_1:=\lf\{(k,i)\in\zz\times\nn:\ \lf|B_{i,E^\complement}^k\r|>\frac12\lf|cB_{i}^k\r|\r\}.
$$
It is easy to see that $\mathrm{I}_1\cap \mathrm{I}_0=\emptyset$ and
$$
\sum_{k\in\zz}\sum_{i\in\nn}h_i^k=\sum_{(k,i)\in \mathrm{I}_0}h_i^k+\sum_{(k,i)\in \mathrm{I}_1}h_i^k
\qquad  \textup{in}\;\cs'(\rn).
$$
Fix $k_0\in\zz$. By Lemma \ref{dec}, it holds true that
\begin{equation}\label{3.6}
\sum_{(k_0,i)\in \mathrm{I}_0}\lf|B_{i}^{k_0}\r|\leq\frac2c\sum_{(k_0,i)\in \mathrm{I}_0}\lf|B_{i,E}^{k_0}\r|
\lesssim \lf|\lf\{x\in E:\ m_N(f)(x)>2^{k_0}\r\}\r|.
\end{equation}
Similarly, we have
\begin{equation}\label{3.7}
\sum_{(k_0,i)\in \mathrm{I}_1}\lf|B_{i}^{k_0}\r|<\frac2c\sum_{(k_0,i)\in \mathrm{I}_1}\lf|B_{i,E^\complement}^{k_0}\r|
\lesssim \lf|\lf\{x\in E^\complement:\ m_N(f)(x)>2^{k_0}\r\}\r|.
\end{equation}
For any $(k,i)\in \mathrm{I}_0$, let $\lambda_{k,i}^{(0)}:=2^k|B_i^k|$ and
$a_{k,i}^{(0)}:=h_i^k/\lambda_{k,i}^{(0)}$ and, for any $(k,i)\in \mathrm{I}_1$,
let $\lambda_{k,i}^{(1)}:=2^k|B_i^k|^{1/p}$ and
$a_{k,i}^{(1)}:=h_i^k/\lambda_{k,i}^{(1)}$.
Then we write the decomposition in \eqref{sll} into
$$
f=\sum_{(k,i)\in \mathrm{I}_0}h_i^k+\sum_{(k,i)\in \mathrm{I}_1}h_i^k=:
\sum_{(k,i)\in \mathrm{I}_0}\lambda_{k,i}^{(0)}a_{k,i}^{(0)}+
\sum_{(k,i)\in \mathrm{I}_1}\lambda_{k,i}^{(1)}a_{k,i}^{(1)}=:f_0+f_1.
$$
Assume that $d\in\mathbb{Z}_+$ satisfies $d\ge\lfloor n(\frac{1}{p}-1)\rfloor$.
By Definition \ref{17fa} and Lemma \ref{dec}, it is easy to see that
$a_{k,i}^{(0)}$ is a local $(1,\infty,d)$-atom supported in the ball $B_i^k$
and $a_{k,i}^{(1)}$ a local $(p,\infty,d)$-atom supported in the ball $B_i^k$.
From \eqref{sdd}, \cite[Theorem 3.14]{yy} and the fact that $f\in h^{\Phi_P}(\rn)$, it follows that
\begin{align*}
\sum_{(k,i)\in \mathrm{I}_0}|\lambda_{k,i}^{(0)}|
&\lesssim\sum_{k\in\zz}2^k\lf|\lf\{x\in E:\ m_N(f)(x)>2^k\r\}\r|
\lesssim\int_{E}m_N(f)(x)\,dx\\
&\sim\int_{E}\frac{m_N(f)(x)}{1+[m_N(f)(x)]^{1-p}}\,dx
\sim\int_{\rn}\Phi_p(m_N(f)(x))\,dx\lesssim1.
\end{align*}
Moreover, using the atomic characterization of $h^1(\rn)$
(see \cite[Theorem 5]{g}), we know
that $f_0\in h^1(\rn)$ and $\|f_0\|_{h^1(\rn)}\lesssim\|f\|_{h^{\Phi_P}(\rn)}$.

As for $f_1$, by \eqref{3.7}, \cite[Theorem 3.14]{yy} and the fact that $f\in h^{\Phi_P}(\rn)$, we find that
\begin{align*}
\sum_{(k,i)\in \mathrm{I}_0}|\lambda_{k,i}^{(1)}|^p
&\lesssim\sum_{k\in\zz}2^{kp}\lf|\lf\{x\in E^\complement:\ m_N(f)(x)(x)>2^k\r\}\r|
\lesssim\int_{E^\complement}[m_N(f)(x)(x)]^{p}\,dx\\
&\sim\int_{E^\complement}\frac{m_N(f)(x)}{1+[m_N(f)(x)]^{1-p}}\,dx
\sim\int_{\rn}\Phi_p(m_N(f)(x))\,dx\lesssim1.
\end{align*}
Then, using the atomic characterization of $h^p(\rn)$ (see \cite[Theorem 5]{g}),
we know that $f_1\in h^p(\rn)$ and $\|f_1\|_{h^p(\rn)}\lesssim\|f\|_{h^{\Phi_P}(\rn)}$.

From the above estimates, we deduce that
$[h^{\Phi_p}(\rn)\cap L^q(\rn)]\subset [h^1(\rn)+h^p(\rn)]$ and, for any
$f\in h^{\Phi_P}(\rn)\cap L^q(\rn)$, there exist $f_0\in h^1(\rn)$ and $f_1\in h^p(\rn)$
such that $f=f_0+f_1$ in $\mathcal{S}'(\rn)$ and
\begin{align}\label{x12}
\|f_0\|_{h^1(\rn)}+\|f_1\|_{h^p(\rn)}\lesssim\|f\|_{h^{\Phi_P}(\rn)}.
\end{align}

Now, we consider the general case. For any $f\in h^{\Phi_P}(\rn)$, by Lemma \ref{den}, we know that
there exist $\{f_l\}_{l\in\nn}\subset h^{\Phi_p}(\rn)\cap L^q(\rn)$ such that
$f=\sum_{l\in\nn}f_l$ in $h^{\Phi_p}(\rn)$ and
$$
\|f_l\|_{h^{\Phi_p}(\rn)}\ls 2^{-l}\|f\|_{h^{\Phi_p}(\rn)}.
$$
From \eqref{x12}, we deduce that, for any $l\in\nn$,
there exist $f_{l,0}\in h^1(\rn)$ and $f_{l,1}\in h^p(\rn)$
such that $f_l=f_{l,0}+f_{l,1}$ in $\mathcal{S}'(\rn)$ and
\begin{align}\label{xx12}
\|f_{l,0}\|_{h^1(\rn)}+\|f_{l,1}\|_{h^p(\rn)}\lesssim\|f_l\|_{h^{\Phi_P}(\rn)}
\lesssim2^{-l}\|f\|_{h^{\Phi_p}(\rn)}.
\end{align}
By this, we know that,
for any $M,\ N\in\nn$ with $M>N$,
$$
\lf\|\sum^{M}_{l=N}f_{l,0}\r\|_{h^1(\rn)}\le\sum^{M}_{l=N}\lf\|f_{l,0}\r\|_{h^1(\rn)}
\lesssim\sum^{M}_{l=N}2^{-l}\|f\|_{h^{\Phi_p}(\rn)}\lesssim2^{-N}\|f\|_{h^{\Phi_p}(\rn)}
$$
and
$$
\lf\|\sum^{M}_{l=N}f_{l,1}\r\|_{h^p(\rn)}^p\le\sum^{M}_{l=N}\lf\|f_{l,1}\r\|_{h^p(\rn)}^p
\lesssim\sum^{M}_{l=N}2^{-lp}\|f\|_{h^{\Phi_p}(\rn)}^p\lesssim2^{-Np}\|f\|_{h^{\Phi_p}(\rn)}^p,
$$
which implies that
$\sum_{l\in\nn}f_{l,0}$ and
$\sum_{l\in\nn}f_{l,1}$ converge in $\cs'(\rn)$.
Let $f_0:=\sum_{l\in\nn}f_{l,0}$ and $f_1:=\sum_{l\in\nn}f_{l,1}$.
It follows from \eqref{xx12} that $f_0\in h^1(\rn)$, $f_1\in h^p(\rn)$
and $f_l=f_{l,0}+f_{l,1}$ in $\mathcal{S}'(\rn)$ satisfying
that
$$
\|f_0\|_{h^1(\rn)}\leq\sum_{l\in\nn}\|f_{l,0}\|_{h^1(\rn)}\lesssim\|f\|_{h^{\Phi_p}(\rn)}
$$
and
$$
\|f_1\|_{h^p(\rn)}^p\leq\sum_{l\in\nn}\|f_{l,1}\|_{h^p(\rn)}^p\lesssim\|f\|_{h^{\Phi_p}(\rn)}^p.
$$
Then we have $f=f_0+f_1$ in $\mathcal{S}'(\rn)$ and
$$
\|f_0\|_{h^1(\rn)}+\|f_1\|_{h^p(\rn)}\lesssim\|f\|_{h^{\Phi_P}(\rn)},
$$
which implies that
$f\in h^1(\rn)+h^p(\rn)$ and hence $h^{\Phi_p}(\rn)\subset
[h^1(\rn)+h^p(\rn)]$. This finishes the proof of Theorem \ref{FINN}.
\end{proof}

We have the following further relations on
$h^p(\rn)$ and $h^{\Phi_P}(\rn)$ with
$p\in(0,1)$ and $\Phi_p$ as in \eqref{333}.

\begin{corollary}\label{11p}
Let $p\in(0,1)$ and $\Phi_p$ be as in \eqref{333}.
Then
\begin{itemize}
\item[\rm(i)]
$h^p(\rn)\subsetneqq h^{\Phi_P}(\rn)$;
\item[\rm(ii)]
$[h^p(\rn)]^*$ and $[h^{\Phi_P}(\rn)]^*$ coincide with equivalent quasi-norms.
\end{itemize}
\end{corollary}

\begin{proof}
From Theorem \ref{FINN}, it easily follows that $h^p(\rn)\subsetneqq h^{\Phi_P}(\rn)$,
which completes the proof of (i).

Now we show (ii).
Let $d:=\lfloor n(1/p-1)\rfloor$ and $\alpha:=1/p-1$.
From this, Lemmas \ref{444}(i) and \ref{lem-dualMHC}, it follows that
$(h^{\Phi_{p}}(\rn))^*=\mathcal{L}_{\loc}^{\Phi_{p},1,d}(\rn)$
and, for any $f\in\mathcal{L}_{\loc}^{\Phi_{p},1,d}(\rn)$,
\begin{align}\label{pp}
\|f\|_{\mathcal{L}_{\loc}^{\Phi_{p},1,d}(\rn)}&\sim
\dsup_{\mathrm{ball\,}B\subset\rn,|B|<1} \frac1{|B|^{1/p}} \int_B |f(x)-P_B^d f(x)|\,dx\\\noz
&\quad+\dsup_{\mathrm{ball\,}B\subset\rn,|B|\geq1} \frac1{|B|} \int_B |f(x)|\,dx
<\infty,
\end{align}
where $P_B^df$ for any ball $B\subset\rn$ denotes the minimizing polynomial of $g$ on $B$ with degree
not greater than $d$.
Moreover, by Remark \ref{r11}(i), we know that
$(h^{p}(\rn))^*=\mathcal{L}_{\loc}^{\alpha,1,d}(\rn)$
and, for any $f\in\mathcal{L}_{\loc}^{\alpha,1,d}(\rn)$
\begin{align*}
\|f\|_{\mathcal{L}_{\loc}^{\alpha,1,d}(\rn)}&:=\sup_{\mathrm{ball\,}B\subset\rn,|B|<1}
\frac{1}{|B|^{1/p}}\lf[\int_{B}|f(x)-P_B^d f(x)|\,dx\r]\\
&\quad+\sup_{\mathrm{ball\,}B\subset\rr,|B|\geq1}
\frac{1}{|B|^{1/p}}\int_{B}|f(x)|\,dx
<\infty.
\end{align*}
From this and \eqref{pp}, we deduce that, for any $f\in\mathcal{L}_{\loc}^{\Phi_{p},1,d}(\rn)$,
\begin{align}\label{990}
\|f\|_{\mathcal{L}_{\loc}^{\alpha,1,d}(\rn)}\lesssim\|f\|_{\mathcal{L}_{\loc}^{\Phi_{p},1,d}(\rn)}.
\end{align}
Thus, $\mathcal{L}_{\loc}^{\Phi_{p},1,d}(\rn)\subset
\mathcal{L}_{\loc}^{\alpha,1,d}(\rn)$.
On the other hand, by Lemma \ref{GR-DDY}, we conclude that,
for any $f\in\mathcal{L}_{\loc}^{\az,1,d}(\rn)$,
\begin{align*}
\dsup_{\mathrm{ball\,}B\subset\rn,|B|\geq1} \frac1{|B|} \int_B |f(x)|\,dx
\le\|f\|_{L^{\infty}(\rn)}\le\|f\|_{\Lambda_{n\alpha}(\rn)}\sim
\|f\|_{\mathcal{L}_{\loc}^{\alpha,1,d}(\rn)}
<\infty,
\end{align*}
which implies that
$$
\|f\|_{\mathcal{L}_{\loc}^{\Phi_{p},1,d}(\rn)}\lesssim\|f\|_{\mathcal{L}_{\loc}^{\az,1,d}(\rn)}.
$$
Thus, $\mathcal{L}_{\loc}^{\alpha,1,d}(\rn)
\subset\mathcal{L}_{\loc}^{\Phi_{p},1,d}(\rn)$.
This, together with \eqref{990}, then finishes the proof of (ii) and hence
of Corollary \ref{11p}.
\end{proof}

To obtain a new structure of $H^{\Phi_p}(\rn)$,
we need the following Calder\'on--Zygmund decomposition of on its the elements, 
which was established in \cite[Section 5.2]{Ky14}.

\begin{lemma}\label{dec2}
Let $p\in(0,1)$, $q\in(1,\infty)$ and $\Phi_p$ be as in \eqref{333}. Assume that $d\in\mathbb{Z}_+$ satisfies $d\ge\lfloor n(\frac{1}{p}-1)\rfloor$.
Then, for any
$f\in H^{\Phi_p}(\rn)\cap L^q(\rn)$,
there exists a sequence $\{h_i^k\}_{k\in\zz,i\in\nn}$
of elements in $L^\infty(\rn)$ supported, respectively, in balls
$\{B_i^k\}_{k\in\zz,i\in\nn}$
such that
\begin{equation*}
f=\sum_{k\in\zz,i\in\nn}h_i^k \qquad  \textup{in}\;\cs'(\rn)
\end{equation*}
and, for any $k\in\zz$,
\begin{equation*}
\mathcal{O}_k=\cup_{i\in\nn}(B_i^k)^\ast\ and\
\{(B_i^k)^\ast\}_{i\in\nn}\ has\ finite\ intersection\ property,
\end{equation*}
where $\mathcal{O}_k:=\{x\in\rn:\ M_N(f)(x)>2^k\}$
with $N:=\lfloor \frac{n}{p}+1\rfloor$, and $(B_i^k)^\ast:=cB_i^k$ with $c\in(0,1)$ being a constant.
Moreover, there exists a positive constant $C$ such that, for any $k\in\zz$ and $i\in\nn$,
\begin{equation*}
|h_i^k|\le C2^k
\end{equation*}
and, for any multi-index $\alpha$ satisfying $|\alpha|\leq d$, it holds true
that
\begin{equation*}
\int_{\rn}x^\alpha h_i^k(x)\,dx=0.
\end{equation*}
\end{lemma}

\begin{theorem}\label{FIN2}
Let $p\in(0,1)$ and $\Phi_p$ be as in \eqref{333}. Then the space
$H^{\Phi_p}(\rn)$ and $H^1(\rn)+H^p(\rn)$ coincide with equivalent quasi-norms.
\end{theorem}
\begin{proof}
Completely repeating the proof of Theorem \ref{FINN} via replacing
Lemma \ref{dec} used therein by the above Lemma \ref{dec2},
we can easily obtain the conclusion of Theorem \ref{FIN2} and we omit
the details here. This finishes the proof of Theorem \ref{FIN2}.
\end{proof}

\begin{remark}
\begin{itemize}
\item[\rm (i)] Let $p\in(0,1)$. From Theorem \ref{FIN2},
it follows that $H^p(\rn)\subsetneqq H^{\Phi_p}(\rn)$ and $(H^{\Phi_p}(\rn))^\ast\subset(H^p(\rn))^\ast$.
In Theorem \ref{FIN2}, if we replace $H^{\Phi_p}(\rn)$ and $H^p(\rn)$,
respectively, by $H^{\phi_p}(\rn)$ and $H^{p}_{W_p}(\rn)$,
where $H^{p}_{W_p}(\rn)$ denotes the weighted Hardy space associated with $W_p$ as in \cite[(1.7)]{c}, then
the corresponding conclusions of Theorem \ref{FIN2}
still hold true, which is just \cite[Theorem 1.2]{c}.
Moreover, using the equivalence that $\|\mathbf1_B\|_{L^{p}_{W_p}(\rn)}\sim\|\mathbf1_B\|_{L^{\phi_p}(\rn)}$
for any $B\subset\rn$ with the positive equivalence constants independent of $B$,
Cao et al. in \cite[Remark 3.1]{c} further showed that $[H^{p}_{W_p}(\rn)]^*$
and $[H^{\phi_p}(\rn)]^*$ coincide with equivalent quasi-norms.
However, from Lemma \ref{444}(i), we deduce that, for any $B\subset\rn$ with $|B|>1$,
$\|\mathbf1_B\|_{L^{\Phi_p}(\rn)}\sim|B|$ and $\|\mathbf1_B\|_{L^{p}(\rn)}=|B|^{1/p}$.
Thus, we cannot use a method similar to that used in the proof of \cite[Remark 3.1]{c} 
to obtain the coincidence of $(H^{\Phi_P}(\rn))^\ast$ and
$(H^p(\rn))^\ast$, which is still unknown.

\item[\rm (ii)] Let $\Phi$ be as in \eqref{111}. We point out that it looks
difficult to obtain intrinsic structures of $H_\ast^{\Phi}(\rn)$ and $h_\ast^{\Phi}(\rn)$ 
via classic Hardy spaces $H^p(\rn)$ or $h^p(\rn)$, with $p\in(0,1]$, because
the structure of the space $L_\ast^{\Phi}(\rn)$ is complicated
and, moreover, the lower type of $\Phi$ as an Orlicz function is uncertain, indeed,
$\Phi$ is an Orlicz function with positive upper type $1$ and positive lower type $q$ 
for any $q\in(0,1)$.
\end{itemize}
\end{remark}

Let $\az\in(0,\infty)$.
For any ball $B(c_B,r_B)\subset\rn$ with $c_B\in\rn$
and $r_B\in(0,\infty)$, let
\begin{align}\label{psiB}
\Psi_\alpha (B):=
\begin{cases} \dfrac{|B|^\az}{(1+|c_B|+r_B)^{n\az}} &\qquad \textup{when}\; n\alpha\notin \nn,
\vspace{0.2cm}\\
\dfrac{|B|^\az}{(1+|c_B|+r_B)^{n\az}\log(e+|c_B|+r_B)}&\qquad \textup{when}\; n\alpha\in \nn.
\end{cases}
\end{align}

We now recall some notions about
local BMO spaces in \cite{yy2}.

\begin{definition}\label{bmo1}
Let $p\in(0,1)$, $\az:=1/p-1$, $d:=\lfloor n(1/p-1)\rfloor$ and
$\phi_p$ be as in \eqref{356}.
\begin{itemize}
\item[(i)]
The \emph{local $\BMO$-type space ${\bmo}^{\az}(\rn)$} is defined
to be the set of all measurable functions $f\in L_{\loc}^1(\rn)$ such that
\begin{align*}
\|f\|_{{\bmo}^{\az}(\rn)}:&=
\sup_{\mathrm{ball\,}B\subset\rn,\ell(B)<1}
\frac{1}{\Psi_\alpha (B)|B|}\int_{B}|f(x)-P_B^df(x)|\,dx\\\noz
&\quad+\sup_{\mathrm{ball\,}B\subset\rn,\ell(B)\geq1}
\frac{1}{\Psi_\alpha (B)|B|}
\int_{B}|f(x)|\,dx
<\infty,
\end{align*}
where $f_B:=\frac{1}{|B|}\int_{B}f(x)\,dx$ for any ball $B\subset\rn$.
\item[(ii)] The \emph{Musielak--Orlicz space} $L^{\phi_p}(\rn)$
is defined to be the set of all measurable functions $f$
such that $$\|f\|_{L^{\phi_p}(\rn)}
:=\inf\lf\{\lambda\in(0,\infty):\ \int_{\rn}\phi_p
\lf(x,\,\frac{|f(x)|}{\lambda}\r)\,dx\leq1 \r\}<\infty.$$
\item[(iii)]
Then the \emph{local Hardy space $h^{\phi_p}(\rn)$}
is defined by setting
$$
h^{\phi_p}(\rn):=\lf\{f\in\cs'(\rn):\ \|f\|_{h^{\phi_p}(\rn)}:
=\lf\|m(f,\varphi)\r\|_{L^{\phi_p}(\rn)}<\infty\r\},
$$
where $\varphi\in\mathcal{S}(\rn)$ satisfies
$\int_{\rn}\varphi(x)\,dx\neq0.$
\end{itemize}
\end{definition}

We have the following further relations on
$h^{\phi_P}(\rn)$ and $h^{\Phi_P}(\rn)$ with
$p\in(0,1)$, $\Phi_p$ as in \eqref{333} and $\phi_p$ be as in \eqref{356}.

\begin{theorem}\label{cheng}
Let $p\in(0,1)$, $\az:=1/p-1$,
$\phi_p$ be as in \eqref{356} and
$\Phi_p$ be as in \eqref{333}.
\begin{itemize}
\item[\rm(i)]
The dual space of $h^{\phi_p}(\rn)$, denoted by $(h^{\phi_p}(\rn))^*$,
is ${\bmo}^{\az}(\rn)$.
\item[\rm(ii)] ${\bmo}^{\az}(\rn)\subsetneqq\mathcal{L}_{\loc}^{\alpha,1,d}(\rn)$.
\item[\rm(iii)]$h^{\Phi_p}(\rn)\subsetneqq h^{\phi_p}(\rn)$.
\end{itemize}
\end{theorem}

\begin{proof}
By \cite[Proposition 2.18]{bckly} and \cite[Corollary 7.6]{yy}, we easily obtain (i).

Now we prove (ii). Let $p\in(0,1)$, $\az:=1/p-1$ and
$d:=\lfloor n(1/p-1)\rfloor$.
Applying
the definitions of $\mathcal{L}_{\loc}^{\alpha,1,d}(\rn)$ and ${\bmo}^{\az}(\rn)$,
we easily find that
${\bmo}^{\az}(\rn)\subset\mathcal{L}_{\loc}^{\alpha,1,d}(\rn)$,
$1\in\mathcal{L}_{\loc}^{\alpha,1,d}(\rn)$, $1\notin{\bmo}^{\az}(\rn)$
and hence ${\bmo}^{\az}(\rn)\subsetneqq\mathcal{L}_{\loc}^{\alpha,1,d}(\rn)$.

As for (iii), from Definition \ref{campa} and
Corollary \ref{11p}, it follows that
$(h^{\Phi_p}(\rn))^\ast=\mathcal{L}_{\loc}^{\alpha,1,d}(\rn)$,
which, together with (i) and (ii), further
implies that
\begin{equation}\label{za}
(h^{\phi_p}(\rn))^\ast\subsetneqq(h^{\Phi_p}(\rn))^\ast.
\end{equation}
From the definitions of $h^{\phi_p}(\rn)$ and $h^{\Phi_p}(\rn)$, it is easy to deduce that
\begin{equation}\label{jkl3}
h^{\Phi_p}(\rn)\subset h^{\phi_p}(\rn)
\end{equation} and
\begin{equation}\label{jkl2}
\|\cdot\|_{h^{\phi_p}(\rn)}\le\|\cdot\|_{h^{\Phi_p}(\rn)}.
\end{equation}
Now we show that $h^{\phi_p}(\rn)\subsetneqq h^{\Phi_p}(\rn)$.
Assume that, as sets,
\begin{equation}\label{jkl}
h^{\phi_p}(\rn)= h^{\Phi_p}(\rn).
\end{equation}
From Remark \ref{rem-add1}, \cite[Proposition 2.12]{bckly}, Definitions \ref{bmo1}(iii) and \ref{defn-hmo-1},
we deduce that $(h^{\Phi_p}(\rn),\|\cdot\|_{h^{\Phi_p}(\rn)}^{p/2})$
and $(h^{\phi_p}(\rn),\|\cdot\|_{h^{\phi_p}(\rn)}^{p/2})$ are Fr\'echet
spaces (see, for instance, \cite[p.\,52, Definition 1]{y}),
which, together with \eqref{jkl2} and the norm-equivalence theorem
(see, for instance, \cite[Corollary 2.12(d)]{R}), further implies that
$\|\cdot\|_{h^{\Phi_p}(\rn)}^{p/2}\lesssim\|\cdot\|_{h^{\phi_p}(\rn)}^{p/2}$
and hence $\|\cdot\|_{h^{\Phi_p}(\rn)}\lesssim\|\cdot\|_{h^{\phi_p}(\rn)}$.
By this, we know that,
for any $L\in (h^{\Phi_p}(\rn))^\ast$ and $f\in h^{\phi_p}(\rn)$,
$$
|L(f)|\lesssim\|f\|_{h^{\Phi_p}(\rn)}\lesssim\|f\|_{h^{\phi_p}(\rn)}
$$
and hence $(h^{\Phi_p}(\rn))^\ast\subset(h^{\phi_p}(\rn))^\ast$,
which leads to a
contradiction with \eqref{za}. Thus, the assumption \eqref{jkl} is not true. From this and
\eqref{jkl3}, we further deduce that $h^{\Phi_p}(\rn)\subsetneqq h^{\phi_p}(\rn)$.
This finishes the proof of (iii) and hence of
Theorem \ref{cheng}.
\end{proof}

\section{Div-curl estimates}\label{s6}

In this section, we establish some estimates on the product of elements in the local Hardy space and its
dual space having, respectively, zero $\lfloor n\alpha\rfloor$-inhomogeneous curl and
zero divergence
(see Theorems \ref{div} and \ref{sssa} below).

\subsection{Div-curl estimates on the product of elements in
$h^p(\rn)$ and $\Lambda_{n\alpha}(\rn)$\\ with $p\in(0,1)$ and
$\az:=\frac1p-1$}

In this subsection, applying the target space $h^{\Phi_p}(\mathbb R^n)$ in the
bilinear decomposition theorem in Section \ref{s4}, we establish
an estimate on the product of elements in $h^p(\rn)$ with $p\in(0,1)$
and $\Lambda_{n\alpha}(\rn)$ with $\az:=\frac1p-1$
having, respectively, zero $\lfloor n\alpha\rfloor$-inhomogeneous curl and
zero divergence
(see Theorem \ref{div} below).

In what follows,
for any $\phi\in\cs(\rn)$, $\widehat\phi$ denotes its \emph{Fourier transform}
which is defined by setting, for any $\xi\in\rn$,
$$
\widehat\phi(\xi):=\int_\rn e^{-2\pi \imath x\xi}\phi(x)\,dx,
$$
where $\imath:=\sqrt{-1}$.
For any $f\in\cs'(\rn)$, $\widehat f$ is defined by setting, for any $\varphi\in\mathcal{S}(\rn)$,
$\la\widehat f,\varphi\ra:=\la f,\widehat\varphi\ra$; also, for any $f\in\mathcal{S}(\rn)$
[resp., $\mathcal{S}'(\rn)$],
$f^{\vee}$ denotes its \emph{inverse Fourier transform} which is defined by setting,
for any $\xi\in\rn$, $f^{\vee}(\xi):=\widehat{f}(-\xi)$ [resp., for any $\varphi\in\mathcal{S}(\rn)$,
$\la f^{\vee},\varphi\ra:=\la f,\varphi^{\vee}\ra$].

The following lemma clarifies the relation between $H^p(\rn)$ and
$h^p(\rn)$ (see \cite[Lemma 4]{g} for details).

\begin{lemma}\label{rhh}
Let $p\in(0,1]$ and $\psi \in \mathcal{S}(\rn)$ satisfy $\widehat\psi(\vec{0}_n)=1$ and
$\int_{\rn}\psi(x)x^\alpha\,dx=0$ for any
$\alpha\in\zz_+^n$ with $0<|\alpha|\le\lfloor n(\frac1p-1)\rfloor$.
Then there exists a positive constant
$C$ such that, for any $f \in h^p(\rn)$,
$$\lf\|f-\psi\ast f\r\|_{H^p(\rn)}\le C\|f\|_{h^p(\rn)}.$$
\end{lemma}

The next lemma concerns the wavelet characterization
of the Hardy space (see, for instance, \cite[Theorem 4.2]{GM01} or \cite[Theorem 6.3(i)]{LY12}).

\begin{lemma}\label{thm17-3.2}
Let $p\in(0,\,1]$ and $d\in\nn$ with
$d\geq\lfloor n(1/p-1)\rfloor$.
Let $\{\psi_I^{\lz}\}_{I\in\mathcal{D},\,\lz\in E}$ be the wavelets
as in \eqref{www}
with the regularity parameter $d$.
Then $f\in H^p(\rn)$
if and only if
\begin{align*}
\lf\|\mathcal W_\psi f\r\|_{L^p(\rn)}&:=\lf\|\lf\{ \dsum_{I\in\cd}\sum_{\lz\in E}
\lf|\langle f,\,\psi_I^\lz \rangle\r|^2 |I|^{-1}
{\mathbf 1}_I\r\}^{\frac{1}{2}} \r\|_{L^p(\rn)}<\infty.
\end{align*}
Moreover,
$\|f\|_{H^p(\rn)}\sim \|\mathcal W_\psi f\|_{L^p(\rn)}$
with positive equivalence constants independent of $f$.
\end{lemma}

The following lemma is on the wavelet characterization of
the Campanato space (see, for instance, \cite[Corollary 2]{LLL09} or \cite[Theorem 6.3(i)]{LY12}).

\begin{lemma}\label{thm 3.2}
Let $\az\in[0,\,\fz)$, and $d\in\nn$ with
$d>n/p$.
Let $\{\psi_I^{\lz}\}_{I\in\mathcal{D},\,\lz\in E}$ be the wavelets
as in \eqref{www}
with the regularity parameter $d$.
Then $g\in \mathfrak C_{\az}(\rn)$
if and only if its wavelet coefficients
$\{s_{I,\,\lz}\}_{I\in\mathcal{D},\,\lz\in {\rm E}}
:=\{\langle g,\,\psi^\lz_I\rangle\}_{I\in\mathcal{D},\,\lz\in {\rm E}}$
satisfy that
\begin{align*}
\lf\|\{s_{I,\,\lz}\}_{I\in\mathcal{D},\,\lz\in {E}}\r\|_{\mathcal{C}_\az(\rn)}:=\dsup_{I\in \mathcal{D}}
\lf\{\frac{1}{|I|^{2\az+1}}\dsum_{\gfz{J\in\mathcal{D}}{J\subset I}}\dsum_{\lz\in E}\lf|s_{J,\,\lz}\r|^2\r\}^{\frac{1}{2}}<\infty.
\end{align*}
Moreover, $\|g\|_{\mathfrak C_{\az}(\rn)}\sim \lf\|\{s_{I,\,\lz}\}_{I\in\mathcal{D},\,\lz\in {E}}\r\|_{\mathcal{C}_\az(\rn)}$
with positive equivalence constants independent of $g$.
\end{lemma}

We use the \emph{symbol} $C_{\mathrm{c}}^\fz(\rn)$ to denote the set of all infinitely differentiable
functions on $\rn$ with compact supports. Let $\psi\in\cs(\rn)$.
For a vector field $\mathbf{F}:= (F_1, \ldots , F_n)$
of
distributions in $\cs'(\rn)$, let $\psi\ast\mathbf{F}:= (\psi\ast F_1,\,\ldots,\,\psi\ast F_n)$.
The {\it divergence} $\mathrm{div}\,\mathbf{F}$ of $\mathbf{F}$ is defined by setting, for
any
$\phi\in C_{\mathrm{c}}^\fz(\rn)$,
\begin{align}\label{dis}
\lf\langle \mathrm{div}\,\mathbf{F},\,\phi  \r\rangle:=
-\sum_{j=1}^n \lf\langle F_j,\,\frac{\partial \phi}{\partial x_j} \r\rangle.
\end{align}
Moreover, the {\it curl} $\mathrm{curl}\, \mathbf{F}$ of $\mathbf{F}$ is defined by setting
$\mathrm{curl}\, \mathbf{F}:=\{(\mathrm{curl}\, \mathbf{F})_{i,\,j}\}_{i,\,j\in\{1,\,\ldots,\,n\}}$,
where, for any $i,\ j\in\{1,\ldots,n\}$ and
$\phi\in C_{\mathrm{c}}^\fz(\rn)$,
\begin{align}\label{dis2}
\lf\langle (\mathrm{curl}\, \mathbf{F})_{i,\,j},\,\phi  \r\rangle
:=\lf\langle F_j,\,\frac{\partial \phi}{\partial x_i} \r\rangle-
\lf\langle F_i,\,\frac{\partial \phi}{\partial x_j} \r\rangle;
\end{align}
see, for instance, \cite[p.\,507]{Da95}. Observe that div-curl estimates have been investigated in several
articles; see, for instance,
\cite{BFG,BGK12,cds05,cdy16,cdy10,cdy09,CLMS}.

For any $p\in(0,\,1]$, let
\begin{align}\label{eqn 5.1}
h^p(\rn;\,\rn):=\lf\{\mathbf{F}:=(F_1,\,\ldots,\,F_n):\ \ \text{for any}\ i\in\{1,\,\ldots,\,n\}, \ F_i\in h^p(\rn)\r\}
\end{align}
equipped with the quasi-norm
\begin{align*}
\|\mathbf{F}\|_{h^p(\rn;\,\rn)}:=\lf[\dsum_{i=1}^n\|F_i\|_{h^p(\rn)}^2\r]^{\frac{1}{2}}.
\end{align*}
In a similar way, we define $H^p(\rn;\,\rn)$, $L^2(\rn;\,\rn)$, $\bmo(\rn;\,\rn)$ and the \emph{vector-valued
inhomogeneous Lipschitz space}
$\Lambda_{n\alpha}(\rn;\,\rn)$ as well as the norms $\|\cdot\|_{H^p(\rn;\,\rn)}$,
$\|\cdot\|_{L^2(\rn;\,\rn)}$, $\|\cdot\|_{\bmo(\rn;\,\rn)}$ and $\|\cdot\|_{\Lambda_{n\alpha}(\rn;\,\rn)}$, where  $\az\in(0,\fz)$.

\begin{definition}\label{a2.15}
Let $\az\in\zz_+$. A vector field $\mathbf{F}:= (F_1, \ldots , F_n)$
of
distributions in $\cs'(\rn)$
is said to have the \emph{zero
$\alpha$-inhomogeneous curl}, denoted by \emph{$\alpha$-inhomogeneous $\mathrm{curl}\,\mathbf{F}=0$},
if
there exists a $\psi \in \mathcal{S}(\rn)$ satisfying $\widehat\psi(\vec{0}_n)=1$ and
$$\int_{\rn}\psi(x)x^\beta\,dx=0$$ for any
$\beta\in\zz_+^n$ with $0<|\beta|\le\az$
such that
$\mathrm{curl}\,[\mathbf{F}-\psi\ast\mathbf{F}]\equiv0$ in the sense
of \eqref{dis2}.
\end{definition}

\begin{remark}
Choose $\psi \in \mathcal{S}(\rn)$ such that
$\mathbf{1}_{Q({\vec 0_n},2)}\leq\psi \le \mathbf{1}_{Q({\vec 0_n},4)}.$
Then $\widehat\psi(\vec{0}_n)=1$ and
$$\int_{\rn}\psi(x)x^\beta\,dx=0$$ for any
$\beta\in\zz_+^n$ with $|\beta|>0$, namely, $\psi$
satisfies the assumptions of $\psi$ in Definition \ref{a2.15}.
\end{remark}

\begin{proposition}\label{weak}
Let $p\in(0,\,1]$, $\az:=1/p-1$ and $\mathbf{F}:= (F_1, \ldots , F_n)\in h^p(\rn;\,\rn)$
satisfy ${\rm curl}\, \mathbf{F}\equiv0$.
Then the vector field $\mathbf{F}$ also
satisfies $\lfloor n\alpha\rfloor$-inhomogeneous $\mathrm{curl}\,\mathbf{F}=0$ as in Definition \ref{a2.15}.
\end{proposition}
\begin{proof}
Let $p\in(0,\,1]$. From \cite[Theorem 7.10(a)]{R}, we deduce that
$C_{\mathrm{c}}^\fz(\rn)$ is dense in $\cs(\rn)$,
which, combined with the fact that ${\rm curl}\, \mathbf{F}\equiv0$, implies
that, for any $i,\,j\in\{1,\,\ldots,\,n\}$ and $\phi\in\cs(\rn)$,
\begin{align}\label{pp2}
\lf\langle (\mathrm{curl}\,\mathbf{F})_{i,\,j},\,\phi  \r\rangle
:=\lf\langle F_j,\,\frac{\partial \phi}{\partial x_i} \r\rangle-
\lf\langle F_i,\,\frac{\partial \phi}{\partial x_j} \r\rangle=0.
\end{align}
Let $\psi \in \mathcal{S}(\rn)$ satisfy $\widehat\psi(\vec{0}_n)=1$ and
$\int_{\rn}\psi(x)x^\beta\,dx=0$ for any
$\beta\in\zz_+^n$ with $0<|\beta|\le\lfloor n\alpha\rfloor$.
By this and \eqref{pp2}, we know that, for any $i,\,j\in\{1,\,\ldots,\,n\}$ and $\phi\in\cs(\rn)$ ,
\begin{align*}
\lf\langle (\mathrm{curl}\, \psi\ast\mathbf{F})_{i,\,j},\,\phi  \r\rangle
&=\lf\langle \psi\ast F_j,\,\frac{\partial \phi}{\partial x_i} \r\rangle-
\lf\langle \psi\ast F_i,\,\frac{\partial \phi}{\partial x_j} \r\rangle\\
&=
\lf\langle F_j,\,\frac{\partial [\psi(-\cdot)\ast \phi]}{\partial x_i} \r\rangle-
\lf\langle F_i,\,\frac{\partial [\psi(-\cdot)\ast \phi]}{\partial x_j} \r\rangle=0
\end{align*}
and hence $\mathrm{curl}\,[\psi\ast\mathbf{F}]\equiv0$,
which further implies that $\mathrm{curl}\,[\mathbf{F}-\psi\ast\mathbf{F}]\equiv0$.
Thus, the vector field $\mathbf{F}$ also
satisfies $\lfloor n\alpha\rfloor$-inhomogeneous $\mathrm{curl}\,\mathbf{F}=0$ as in Definition \ref{a2.15}.
This finishes the proof of Proposition \ref{weak}.
\end{proof}

Using Theorem \ref{mainthm4},
we obtain the following priori estimate on the div-curl product of elements in $h^{\Phi_p}(\rn)$
and its dual space. Recall that, for any $j\in\{1,\ldots,n\}$ and $g\in L^2(\rn)$,
$R_j(g):=[-\imath\frac{\xi_j}{|\xi|}\widehat g(\xi)]^{\vee}$
denotes the \emph{$j$-th Riesz transform} of $g$, where $\imath:=\sqrt{-1}$.

\begin{theorem}\label{div}
Let $p\in(0,\,1)$, $\az:=1/p-1$ and
$\Phi_p$ be as in \eqref{333}.
Assume further that $\mathbf{F}\in h^p(\rn;\,\rn)$
satisfies $\lfloor n\alpha\rfloor$-inhomogeneous $\mathrm{curl}\,\mathbf{F}=0$ as in Definition \ref{a2.15}
and
$\mathbf{G}\in \Lambda_{n\alpha}(\rn;\,\rn)$
satisfies ${\rm div}\, \mathbf{G}\equiv0$
{\rm[}the equality holds true in the sense of \eqref{dis}{\rm]}.
Then the inner product $\mathbf{F}\cdot
\mathbf{G}\in h^{\Phi_p}(\rn)$ and
\begin{align*}
\| \mathbf{F}\cdot
\mathbf{G}\|_{h^{\Phi_p}(\rn)} \le C \|\mathbf{F}\|_{h^p(\rn;\,\rn)}
\|\mathbf{G}\|_{\Lambda_{n\alpha}(\rn;\,\rn)},
\end{align*}
where $C$ is a positive constant independent of $\mathbf{F}$ and $\mathbf{G}$.
\end{theorem}

\begin{remark}
In Theorem \ref{div}, if we replace $h^p(\rn;\,\rn)$,
$\lfloor n\alpha\rfloor$-inhomogeneous $\mathrm{curl}\,\mathbf{F}=0$,
$h^{\Phi_p}(\rn)$
and $\Lambda_{n\alpha}(\rn;\,\rn)$, respectively, by
$H^p(\rn;\,\rn)$, ${\rm curl}\, \mathbf{F}\equiv0$, $H^{\phi_p}(\rn)$ and $\mathfrak C_{\az}(\rn;\,\rn)$with
$\phi_p$ as in \eqref{356},
then the corresponding conclusion of Theorem \ref{div}
still holds true, which is just \cite[Theorem 1.5]{bckly}.
Observe that, when $p\in(0,1)$ and $\az:=\frac1p-1$,
$H^p(\rn;\,\rn)\subset h^p(\rn;\,\rn)$,
$\Lambda_{n\alpha}(\rn;\,\rn)\subset
\mathfrak C_{\az}(\rn;\,\rn)$
and $H^{\phi_p}(\rn)\subset h^{\phi_p}(\rn)$.
Thus, it is natural to use the space $h^{\phi_p}(\rn)$
as the target space in Theorem \ref{div}; however,
by Theorem \ref{cheng}(iii), we know that the target space $h^{\Phi_p}(\rn)$
in Theorem \ref{div} satisfies
$h^{\Phi_p}(\rn)\subsetneqq h^{\phi_p}(\rn)$. In this sense, Theorem \ref{div} is a little surprising.
Moreover, by Proposition \ref{weak}, we conclude that the assumption,
$\lfloor n\alpha\rfloor$-inhomogeneous $\mathrm{curl}\,\mathbf{F}=0$, in Theorem \ref{div}
is weaker than the assumption ${\rm curl}\, \mathbf{F}\equiv0$ in
\cite[Theorem 1.5]{bckly}, which perfectly matches
the inhomogeneity of $h^p(\rn;\,\rn)$ and $\Lambda_{n\alpha}(\rn;\,\rn)$.
\end{remark}

Now, we use Theorem \ref{mainthm4} to prove
Theorem \ref{div}.

Let $p\in(0,1]$ and $\az\in[0,\infty)$. For any sequence $\{a_I \}_{I\in\mathcal D}$ of complex numbers, we let
$$\lf\|\{a_I\}_{I\in\mathcal{D}}\r\|_{\mathcal{C}_\az(\rn)}:=\dsup_{I\in \mathcal{D}}
\lf\{\frac{1}{|I|^{2\az+1}}\dsum_{{J\in\mathcal{D}},\,{J\subset I}}\lf|a_{J}\r|^2\r\}^{\frac{1}{2}}
$$
and
\begin{align*}
\lf\|\{a_{I}\}_{I\in\mathcal{D}}\r\|_{ \dot{f}_{p,\,2}^0(\rn)}
:=\lf\|\lf\{\dsum_{{I\in\mathcal{D}}} \lf[\lf|a_{I}\r| |I|^{-\frac{1}2}
\mathbf1_I\r]^2\r\}^{\frac{1}{2}} \r\|_{L^p(\rn)}.
\end{align*}

We also need the following duality argument for sequence spaces, which is just
\cite[Lemma 5.2]{bckly}.

\begin{lemma}\label{lem4-13}
Let $p\in(0,\,1]$. Then there exists a positive constant $C$ such that,
for any sequences $\{a_I \}_{I\in\mathcal D}$  and $\{b_I \}_{I\in\mathcal D}$ of complex numbers,
\begin{align*}
\lf|\dsum_{I\in\mathcal{D}} a_I b_I\r|\le C
\lf\|\{a_I\}_{I\in \mathcal{D}}\r\|_{\dot{f}^0_{p,\,2}(\rn)}
\lf\|\{b_{I}\}_{I\in \mathcal{D}}\r\|_{\mathcal{C}_{1/p-1}(\rn)},
\end{align*}
whenever the right hand side of the above inequality is finite.
\end{lemma}

Using Theorem \ref{mainthm1} and Lemma \ref{lem4-13},
we now turn to the proof of Theorem \ref{div}.

\begin{proof}[Proof of Theorem \ref{div}]
Let $p\in(0,1)$, $\az:=1/p-1$,
$\mathbf{F}:=(F_1, \ldots , F_n)\in h^p(\rn;\,\rn)$
with $\lfloor n\alpha\rfloor$-inhomogeneous $\mathrm{curl}\,\mathbf{F}=0$ as in Definition \ref{a2.15}, and
$\mathbf{G}:=(G_1, \ldots , G_n)\in \Lambda_{n\alpha}(\rn;\,\rn)$
with ${\rm div}\, \mathbf{G}\equiv0$ in the sense of \eqref{dis}.
By the fact that $\mathbf{F}$ is $\lfloor n\alpha\rfloor$-inhomogeneous $\mathrm{curl}\,\mathbf{F}=0$, we know that
there exists a $\psi \in \mathcal{S}(\rn)$ satisfying $\widehat\psi(\vec{0}_n)=1$ and
$\int_{\rn}\psi(x)x^\beta\,dx=0$ for any
$\beta\in\zz_+^n$ with $0<|\beta|\le\lfloor n\alpha\rfloor$
such that
\begin{equation}\label{dofa}
\mathrm{curl}\,[\mathbf{F}-\psi\ast\mathbf{F}]\equiv0
\end{equation}
in the sense
of \eqref{dis2}.
Then we write
\begin{equation}\label{dof}
\mathbf{F}=[\mathbf{F}-(\psi\ast F_1,\,\ldots,\,\psi\ast F_n)]+(\psi\ast F_1,\,\ldots,\,\psi\ast F_n)=:\mathbf{F}_1
+\mathbf{F}_2.
\end{equation}
By Lemma \ref{rhh}, we know that, for any $i\in\{1,\ldots,n\}$,
\begin{equation}\label{dof1}
\lf\|F_i-\psi\ast F_i\r\|_{H^p(\rn)}
\lesssim \|F_i\|_{h^p(\rn)} < \infty,
\end{equation}
which implies that $\mathbf{F}_1 \in H^p(\rn;\,\rn)$. Moreover, from \eqref{dofa}, it easily
follows that ${\rm curl}\,\mathbf{F}_1\equiv0$ in the sense of \eqref{dis2}.

Let $\Phi_p$ be as in \eqref{333}.
We first prove that
\begin{equation}\label{lne1}
\|\mathbf{F}_2\cdot\mathbf{G}\|_{h^{\Phi_p}(\rn)}\lesssim
\|\mathbf{F}\|_{h^p(\rn;\,\rn)}
\|\mathbf{G}\|_{\Lambda_{n\alpha}(\rn;\,\rn)}.
\end{equation}
By Lemma \ref{444}(i),
it suffices to show that, for any $i\in\{1,\ldots,n\}$,
\begin{equation}\label{lne}
\|(\psi\ast F_i)G_i\|_{h^{p}(\rn)}\lesssim
\|F_i\|_{h^p(\rn)}
\|G_i\|_{L^\infty(\rn)}.
\end{equation}
Fix $i\in\{1,\ldots,n\}$.
Let $Q_{k}:=2k+[0,2)^n$ for any   $k\in\mathbb{Z}^n$.
Then, for almost every $x\in\rn$, we have
$$(\psi\ast F_i)(x)G_i(x) = \sum_{k\in\mathbb{Z}^n}(\psi\ast F_i)(x)G_i(x)\mathbf{1}_{Q_{k}}(x).$$
For any $k\in\mathbb{Z}^n$, if $\|(\psi\ast F_i)G_i\|_{L^\infty(Q_k)}=0$, define
$$\lambda_k:=0\quad \text{and}\quad a_k:=0;$$
if $\|(\psi\ast F_i)G_i\|_{L^\infty(Q_k)}\neq0$, define
$$\lambda_k:=\|\mathbf{1}_{Q_k}\|_{L^p(\rn)}\|
(\psi\ast F_i)G_i\|_{L^\infty(Q_k)} \quad \text{and} \quad
a_k:=\frac{(\psi\ast F_i)G_i\mathbf{1}_{Q_k}}
{\|\mathbf{1}_{Q_k}\|_{L^p(\rn)}\|(\psi\ast F_i)G_i\|_{L^\infty(Q_k)}}.$$
Then, for almost every $x\in\rn$,
\begin{equation}\label{term21}
(\psi\ast F_i)(x)G_i(x)=\sum_{k\in\mathbb{Z}^n}\lambda_ka_k(x).
\end{equation}
From \cite[Theorem 2.3.20]{G1}, it follows that there exists an $N\in\nn$ such that,
for any $x\in\rn$,
$|(\psi\ast F_i)(x)|\lesssim(1+|x|)^N$,
which, combined with \eqref{term21} and the fact that $G_i\in L^\infty(\rn)$, further implies
that $(\psi\ast F_i)G_i\in\cs'(\rn)$ and
\begin{equation}\label{erm21}
(\psi\ast F_i)G_i=\sum_{k\in\mathbb{Z}^n}\lambda_ka_k\qquad \text{in}\;\cs'(\rn).
\end{equation}
Since $|Q_k|>1$ for any $k\in\mathbb{Z}^n$, it follows that,
for any $k\in\mathbb{Z}^n$, $a_k$ is a local $(p,\infty,d)$-atom.
Therefore, to show \eqref{lne}, by \eqref{erm21} and \cite[Theorem 5.1]{t}, it suffices to prove that
\begin{equation}\label{term22}
\lf[\sum_{k\in\mathbb{Z}^n}(\lambda_k)^p\r]^{1/p}
\lesssim \|F_i\|_{h^p(\rn)}
\|G_i\|_{L^\infty(\rn)}.
\end{equation}
Indeed, for any fixed $x_0\in\rn$, there exists only a $k_0\in\mathbb{Z}^n$
such that $x_0 \in Q_{k_0}$. Moreover, there exists a positive
constant $c_0$ such that, for any $k\in\mathbb{Z}^n$ and $x\in Q_k$,
$Q_k\subset B(x,c_0)$.
Then, from the definition of $\{\lambda_k\}_{k\in\mathbb{Z}^n}$
and the fact $G_i\in L^\infty(\rn)$, we deduce that
\begin{align*}
&\left[\sum_{k\in\mathbb{Z}^n}\left(\frac{\lambda_k}{\|\mathbf{1}_{Q_k}\|_{L^p(\rn)}}\right)^p
\mathbf{1}_{Q_k}(x_0)\right]^{1/p}\\
&\quad=\frac{\lambda_{k_0}}
{\|\mathbf{1}_{Q_{k_0}}\|_{L^p(\rn)}}
\le\lf\|\psi\ast F_i\r\|_{L^\infty(Q_{k_0})}\|G_i\|_{L^\infty(\rn)}
\leq\sup_{y\in B(x_0,c_0)}\lf|\psi\ast F_i(y)\r|\|G_i\|_{L^\infty(\rn)}\\\noz
&\quad\leq\sup_{t\in(0,2)}\left\{\sup_{y\in B(x_0,c_0t)}
\lf|\psi_t\ast F_i(y)\r|\right\}\|G_i\|_{L^\infty(\rn)}
\lesssim\sup_{t\in(0,1)}\left\{\sup_{y\in B(x_0,2c_0t)}
\lf|\widetilde{\psi}_t\ast F_i(y)\r|\right\}\|G_i\|_{L^\infty(\rn)},
\end{align*}
where  $\widetilde{\psi}(\cdot):=\psi(\frac{\cdot}{2})$.
By this and \cite[Theorem 5.3]{shyy}, we conclude that
\begin{align*}
\lf[\sum_{k\in\mathbb{Z}^n}(\lambda_k)^p\r]^{1/p}&=\lf[\sum_{k\in\mathbb{Z}^n}
\left(\frac{\lambda_k}
{\|\mathbf{1}_{Q_k}\|_{L^p(\rn)}}\right)^p|Q_k|\r]^{1/p}=
\left\|\left[\sum_{k\in\mathbb{Z}^n}\left(\frac{\lambda_k}
{\|\mathbf{1}_{Q_k}\|_{L^p(\rn)}}\right)^p
\mathbf{1}_{Q_k}\right]^{1/p}\right\|_{L^p(\rn)} \\\noz
&\lesssim \lf\|\sup_{t\in(0,1)}\left\{\sup_{y\in B(x_0,2c_0t)}
\lf|\widetilde{\psi}_t\ast F_i(y)\r|\right\}\r\|_{L^p(\rn)}\|G_i\|_{L^\infty(\rn)}\lesssim \|F_i\|_{h^p(\rn)}\|G_i\|_{L^\infty(\rn)},
\end{align*}
which further implies that \eqref{term22} holds true and hence \eqref{lne} holds true.
This finishes the proof of
\eqref{lne1}.

Now we show that
\begin{align}\label{10}
\mathbf{F}_1\cdot\mathbf{G}\in h^{\Phi_p}(\rn).
\end{align}
Let $(F_1^{(1)},\,\ldots,\,F_n^{(1)}):=\mathbf{F}_1$.
By Theorem \ref{mainthm4}(i), we know that there exist two bounded bilinear operators
$$S:\, H^p(\rn)\times \Lambda_{n\alpha}(\mathbb{R}^n)\to L^1(\rn)$$ and
$$T:\, H^p(\rn)\times \Lambda_{n\alpha}(\mathbb{R}^n)\to H^{\Phi_p}(\rn)$$
such that
\begin{align*}
\mathbf{F}_1\cdot \mathbf{G}&=\dsum_{i=1}^n F_i^{(1)}\times G_i= \dsum_{i=1}^n S(F_i^{(1)},\,G_i)+ \dsum_{i=1}^nT(F_i^{(1)},\,G_i)\\
&=:\mathrm{A}(\mathbf{F}_1,\,\mathbf{G})+\mathrm{B}(\mathbf{F}_1,\,\mathbf{G})
\qquad \text{in}\;\cs'(\rn).
\end{align*}
Moreover, by this, Theorem \ref{mainthm4}(i) and \eqref{dof1}, we know that
$\mathrm{B}(\mathbf{F}_1,\,\mathbf{G})\in h^{\Phi_p}(\rn)$ and
\begin{align}\label{00}
\lf\|\mathrm{B}(\mathbf{F}_1,\,\mathbf{G})\r\|_{h^{\Phi_p}(\rn)}&\le
\lf\|\mathrm{B}(\mathbf{F}_1,\,\mathbf{G})\r\|_{H^{\Phi_p}(\rn)}
\ls \|\mathbf{F}_1\|_{H^p(\rn;\,\rn)}
\|\mathbf{G}\|_{\Lambda_{n\alpha}(\rn;\,\rn)}\\\noz
&\ls \|\mathbf{F}\|_{h^p(\rn;\,\rn)}
\|\mathbf{G}\|_{\Lambda_{n\alpha}(\rn;\,\rn)}.
\end{align}
Now we prove that
\begin{align}\label{007}
\lf\|\mathrm{A}(\mathbf{F}_1,\,\mathbf{G})\r\|_{h^{\Phi_p}(\rn)}\ls
\|\mathbf{F}\|_{h^p(\rn;\,\rn)}\|\mathbf{G}\|_{\Lambda_{n\alpha}(\rn;\,\rn)}.
\end{align}
By the fact $H^1(\rn)\subset h^1(\rn)\subset h^{\Phi_p}(\rn)$
and \eqref{dof1}, it suffices to show that
$\mathrm{A}(\mathbf{F}_1,\,\mathbf{G})\in H^1(\rn)$ and
\begin{align}\label{008}
\lf\|\mathrm{A}(\mathbf{F}_1,\,\mathbf{G})\r\|_{H^1(\rn)}\ls
\|\mathbf{F}_1\|_{H^p(\rn;\,\rn)}\|\mathbf{G}\|_{\Lambda_{n\alpha}(\rn;\,\rn)}.
\end{align}
Let
$f:=-\sum_{j=1}^n R_j(F_j^{(1)}).$ From the boundedness of the Riesz transform on  $H^p(\rn)$
(see, for instance, \cite[Theorem 4.7]{Lu95}), we deduce that $f\in H^p(\rn)$.
From \eqref{dofa}, it follows that
${\rm curl} \,\mathbf{F}_1\equiv0$ in the sense of \eqref{dis2}.
By this, we know that,
for any $k,\,j\in\{1,\,\ldots,\,n\}$,
$$\frac{\partial F_k^{(1)}}{\partial x_j}-\frac{\partial {F_j^{(1)}}}{\partial x_k}=0$$
in $\cs'(\rn)$,
which further implies that $-\imath\frac{\xi_j}{|\xi|}\widehat {F_k^{(1)}}(\xi)=-\imath\frac{\xi_k}{|\xi|}\widehat {F_j^{(1)}}(\xi)$
in $\cs'(\rn)$
and hence $R_j(F_k^{(1)})=R_k(F_j^{(1)})$ in $\cs'(\rn)$.
From this and the definition of $f$, we deduce that, for any $k\in\{1,\,\ldots,\,n\}$,
$$R_k(f)=-\sum_{j=1}^nR_k(R_j(F_j^{(1)}))=-\sum_{j=1}^nR_j(R_k(F_j^{(1)}))=-\sum_{j=1}^nR_j^2(F_k^{(1)})=F_k^{(1)}
\qquad \text{in}\;\cs'(\rn).$$
By ${\rm div}\, \mathbf{G}\equiv 0$ in the sense of \eqref{dis},
we know that $\sum_{j=1}^n \frac{\partial G_j}{\partial x_j}=0$ in $\cs'(\rn)$,
which implies that  $\sum_{j=1}^nR_j(G_j)=0$
in $\cs'(\rn)$.
Thus, we have
\begin{align*}
\mathrm{A}(\mathbf{F}_1,\,\mathbf{G})=\dsum_{i=1}^n S(F_i^{(1)},\,G_i)=\dsum_{i=1}^n
\lf[S(R_i(f),\,G_i)+S(f,\,R_i(G_i))\r].
\end{align*}
From this, \cite[Theorem 6.3]{LY12} and the facts that $f\in H^p(\rn)$ and that, for any $i\in\{1,\,\ldots,\,n\}$,
$G_i\in\Lambda_{n\alpha}(\rn)\subset \mathfrak C_{\az}(\rn)$, it follows that, for any $i\in\{1,\,\ldots,\,n\}$,
\begin{align}\label{x11}
G_i=
\sum_{I\in \mathcal{D}}\sum_{\lz\in E} \langle G_i,\,\psi_I^\lz \rangle \psi_I^\lz
\end{align}
in the weak $\ast$-topology induced by $H^p(\rn)$
and
\begin{align}\label{x22}
f=\sum_{I\in \mathcal{D}}\sum_{\lz\in E} \langle f,\,\psi_I^\lz \rangle \psi_I^\lz
\end{align}
in $H^p(\rn)$.
Using \eqref{pi4}, \eqref{So}, \eqref{x11}, \eqref{x22} and Theorem \ref{mainthm4}(i),
we further conclude that, for any $i\in\{1,\,\ldots,\,n\}$,
\begin{align}\label{o33}
&S(R_i(f),\,G_i)+S(f,\,R_i(G_i))\\\noz
&\hs=\dsum_{I'\in\mathcal{D}}\dsum_{\lz'\in E}
\langle R_i f,\,\psi_{I'}^{\lz'} \rangle \langle G_i,\,\psi_{I'}^{\lz'} \rangle
\lf(\psi_{I'}^{\lz'}\r)^2+
\dsum_{I\in\mathcal{D}}\dsum_{\lz\in E}
\langle f,\,\psi_I^\lz \rangle\langle R_i G_i,\,\psi_I^\lz \rangle
\lf(\psi_{I}^{\lz}\r)^2\\\noz
&\hs=\dsum_{I,\,I'\in\mathcal{D}}\dsum_{\lz,\,\lz'\in E}
\langle f,\,\psi_I^\lz \rangle \langle G_i,\,\psi_{I'}^{\lz'} \rangle
\langle R_i \psi_I^\lz,\,\psi_{I'}^{\lz'} \rangle
\lf(\psi_{I'}^{\lz'}\r)^2\\\noz
&\hs\quad+\dsum_{I,\,I'\in\mathcal{D}}\dsum_{\lz,\,\lz'\in E}
\langle f,\,\psi_I^\lz \rangle \langle G_i,\,\psi_{I'}^{\lz'} \rangle
\langle R_i \psi_{I'}^{\lz'},\,\psi_{I}^{\lz} \rangle
\lf(\psi_{I}^{\lz}\r)^2\\\noz
&\hs=\dsum_{I,\,I'\in\mathcal{D}}\dsum_{\lz,\,\lz'\in E}
\langle f,\,\psi_I^\lz \rangle \langle G_i,\,\psi_{I'}^{\lz'} \rangle
\langle R_i \psi_I^\lz,\,\psi_{I'}^{\lz'} \rangle
\lf[\lf(\psi_{I'}^{\lz'}\r)^2-\lf(\psi_{I}^{\lz}\r)^2\r].
\end{align}
For any $I,\,I'\in\mathcal{D}$ and $\lz,\,\lz'\in E$,
by the atomic decomposition of $H^1(\rn)$, we obtain
\begin{align}\label{o34}
\lf\|\lf(\psi_{I'}^{\lz'}\r)^2-\lf(\psi_{I}^{\lz}\r)^2\r\|_{H^1(\rn)}
\lesssim\log\lf(\frac{2^{-j}+2^{-j'}+|x_I-x_{I'}|}{2^{-j}+2^{-j'}}\r),
\end{align}
where $x_I$ and $x_{I'}$ denote the centers of the cubes $I$ and $I'$, respectively.
By \cite[p.\,52, Proposition 1]{m}, we find that
there exists
some $\dz\in(0,1]$ such that, for any $i\in\{1,\,\ldots,\,n\}$,
$I,\,I'\in\mathcal{D}$ with $|I|=2^{-jn}$ and $|I'|=2^{-j'n}$, and $\lz,\,\lz'\in E$,
$$
\lf|\langle R_i \psi_I^\lz,\,\psi_{I'}^{\lz'} \rangle\r|\lesssim p_\dz(I,\,I'),
$$
where
$$
 p_\dz(I,\,I'):=2^{-|j-j'|(\dz+n/2)}\lf(\frac{2^{-j}+2^{-j'}}{2^{-j}+2^{-j'}+|x_I-x_{I'}|}\r)^{n+\dz}.
$$
From this, \eqref{o33}, \eqref{o34} and the fact that,
for any $I,\,I'\in\mathcal{D}$ with $|I|=2^{-jn}$ and $|I'|=2^{-j'n}$,
$$
\log\lf(\frac{2^{-j}+2^{-j'}+|x_I-x_{I'}|}{2^{-j}+2^{-j'}}\r)\leq\frac{2}{\dz}
\lf(\frac{2^{-j}+2^{-j'}+|x_I-x_{I'}|}{2^{-j}+2^{-j'}}\r)^{\frac{\dz}{2}},
$$
it follows that
\begin{align*}
\lf\|S(R_i(f),\,G_i)+S(f,\,R_i(G_i))\r\|_{H^1(\rn)}\ls
\dsum_{I,\,I'\in\mathcal{D}}\dsum_{\lz,\,\lz'\in E}
\lf| \langle f,\,\psi_I^\lz \rangle\r|
\lf| \langle G_i,\,\psi_{I'}^{\lz'} \rangle\r|p_{\dz'}(I,\,I'),
\end{align*}
where $\dz':=\frac{\dz}{2}$.
By this, Lemma \ref{lem4-13}, \cite[Theorem 3.3]{FJ90} together with
the fact that
$\{p_{\dz'}(I,\,I')\}_{I,\,I'\in\mathcal{D}}$ is \emph{almost diagonal }
(see \cite[p.\,53]{FJ90} for the precise definition),
Lemmas
\ref{thm17-3.2} and \ref{thm 3.2}, and
the boundedness of the Riesz transform on $H^p(\rn)$,
we conclude that
\begin{align*}
&\lf\|S(R_i(f),\,G_i)+S(f,\,R_i(G_i))\r\|_{H^1(\rn)}\\
&\hs\ls \dsum_{I'\in\mathcal{D}}\dsum_{\lz'\in E}\lf[ \dsum_{I\in\mathcal{D}}
\dsum_{\lz\in E}\lf| \langle f,\,\psi_{I}^{\lz} \rangle\r|p_{\dz'}(I,\,I')\r]
\lf|\langle G_i,\,\psi_{I'}^{\lz'} \rangle\r|\\
&\hs\ls \lf\|\lf\{\dsum_{I\in\mathcal{D}}\dsum_{\lz\in E}\lf| \langle f,\,\psi_{I}^{\lz}
\rangle\r|p_{\dz'}(I,\,I')\r\}_{I'\in \mathcal{D},\,\lz'\in E}\r\|_{ \dot{f}^0_{p,\,2}(\rn)}
\lf\|\lf\{\langle G_i,\,\psi_{I'}^{\lz'} \rangle\r\}_{I'\in \mathcal{D},\,\lz'\in E}\r\|_{\mathcal{C}_{1/p-1}(\rn)}\\
&\hs\ls \lf\|\lf\{ \langle f,\,\psi_{I}^{\lz} \rangle\r\}_{I\in \mathcal{D},\,\lz\in E}\r\|_{ \dot{f}^0_{p,\,2}(\rn)}
\lf\|\lf\{\langle G_i,\,\psi_{I'}^{\lz'} \rangle\r\}_{I'\in \mathcal{D},\,\lz'\in E}\r\|_{\mathcal{C}_{1/p-1}(\rn)}\\
&\hs\ls\lf\|\sum_{j=1}^n R_j(F_j^{(1)})\r\|_{H^p(\rn)}\|\mathbf{G}\|_{\Lambda_{n\alpha}(\rn;\,\rn)}
\ls\|\mathbf{F}_1\|_{H^p(\rn;\,\rn)}\|\mathbf{G}\|_{\Lambda_{n\alpha}(\rn;\,\rn)},
\end{align*}
which implies
\eqref{008}. Thus, \eqref{007} holds true.
This, combined with \eqref{dof}, \eqref{lne1} and \eqref{00}, shows that
\begin{align*}
\| \mathbf{F}\cdot
\mathbf{G}\|_{h^{\Phi_p}(\rn)} \le C \|\mathbf{F}\|_{h^p(\rn;\,\rn)}
\|\mathbf{G}\|_{\Lambda_{n\alpha}(\rn;\,\rn)},
\end{align*} and hence
finishes the proof of Theorem \ref{div}.
\end{proof}

\subsection{Div-curl estimate on the product of functions in
$h^1(\rn)$ and $\bmo(\rn)$}

In this subsection, we establish
an estimate on the product of functions in $h^1(\rn)$
and $\bmo(\rn)$ having, respectively, zero $0$-inhomogeneous curl and zero divergence
(see Theorem \ref{sssa} below).

\begin{theorem}\label{sssa}
Let $\Phi$ be as in \eqref{111}.
Assume that $\mathbf{F}\in h^1(\rn;\,\rn)$
satisfies $0$-inhomogeneous $\mathrm{curl}\,\mathbf{F}=0$ as in Definition \ref{a2.15}
and
$\mathbf{G}\in \bmo(\rn;\,\rn)$
satisfies ${\rm div}\, \mathbf{G}\equiv0$
{\rm[}the equality holds true in the sense of \eqref{dis}{\rm]}.
Then the inner product $\mathbf{F}\cdot
\mathbf{G}\in h_\ast^{\Phi}(\rn)$ and
\begin{align*}
\| \mathbf{F}\cdot
\mathbf{G}\|_{h_\ast^{\Phi}(\rn)} \le C \|\mathbf{F}\|_{h^1(\rn;\,\rn)}
\|\mathbf{G}\|_{\bmo(\rn;\,\rn)},
\end{align*}
where $C$ is a positive constant independent of $\mathbf{F}$ and $\mathbf{G}$.
\end{theorem}

\begin{proof}
Let $\mathbf{F}\in h^1(\rn;\,\rn)$
with
$0$-inhomogeneous $\mathrm{curl}\,\mathbf{F}=0$ as in Definition \ref{a2.15} and
$\mathbf{G}\in \bmo(\rn;\,\rn)$
with ${\rm div}\, \mathbf{G}\equiv0$ in the sense of \eqref{dis}.
By the fact that $\mathbf{F}$ is $0$-inhomogeneous $\mathrm{curl}\,\mathbf{F}=0$, we know that
there exists a $\psi \in \mathcal{S}(\rn)$ satisfying $\widehat\psi(\vec{0}_n)=1$
such that
\begin{equation}\label{dofa2}
\mathrm{curl}\,[\mathbf{F}-\psi\ast\mathbf{F}]\equiv0
\end{equation}
in the sense
of \eqref{dis2}.
Then we write
\begin{equation}\label{dof5}
\mathbf{F}=[\mathbf{F}-(\psi\ast F_1,\,\ldots,\,\psi\ast F_n)]+(\psi\ast F_1,\,\ldots,\,\psi\ast F_n)=:\mathbf{F}_1
+\mathbf{F}_2.
\end{equation}
By Lemma \ref{rhh}, we know that, for any $i\in\{1,\ldots,n\}$,
\begin{equation*}
\lf\|F_i-\psi\ast F_i\r\|_{H^1(\rn)}
\lesssim \|F_i\|_{h^1(\rn)} < \infty,
\end{equation*}
which implies that $\mathbf{F}_1 \in H^1(\rn;\,\rn)$. Moreover, from \eqref{dofa2}, it easily
follows that ${\rm curl}\,\mathbf{F}_1\equiv0$ in the sense of \eqref{dis2}.

Let $\Phi$ be as in \eqref{111}.
We first prove that
\begin{equation}\label{lne3}
\|\mathbf{F}_2\cdot\mathbf{G}\|_{h_\ast^{\Phi}(\rn)}\lesssim
\|\mathbf{F}\|_{h^1(\rn;\,\rn)}
\|\mathbf{G}\|_{\bmo(\rn;\,\rn)}.
\end{equation}
By the definition of $h_\ast^{\Phi}(\rn)$,
it suffices to show that, for any $i\in\{1,\ldots,n\}$,
\begin{equation}\label{lne4}
\|(\psi\ast F_i)G_i\|_{h^{1}(\rn)}\lesssim
\|F_i\|_{h^1(\rn)}
\|G_i\|_{\bmo(\rn)}.
\end{equation}
Fix $i\in\{1,\ldots,n\}$.
Let $Q_{k}:=2k+[0,2)^n$ for any $t\in(0,\infty)$ and $k\in\mathbb{Z}^n$.
Then, for almost every $x\in\rn$, we have
$$(\psi\ast F_i)(x)G_i(x) = \sum_{k\in\mathbb{Z}^n}(\psi\ast F_i)(x)G_i(x)\mathbf{1}_{Q_{k}}(x).$$
For any $k\in\mathbb{Z}^n$, if $\|(\psi\ast F_i)\|_{L^\infty(Q_k)}\|G_i\|_{\bmo(\rn)}=0$, define
$$\lambda_k:=0\quad \text{and}\quad a_k:=0;$$
if $\|(\psi\ast F_i)\|_{L^\infty(Q_k)}\|G_i\|_{\bmo(\rn)}\neq0$, define
$$\lambda_k:=\|\mathbf{1}_{Q_k}\|_{L^1(\rn)}
\|\psi\ast F_i\|_{L^\infty(Q_k)}\|G_i\|_{\bmo(\rn)} \quad \text{and} \quad
a_k:=\frac{(\psi\ast F_i)G_i\mathbf{1}_{Q_k}}
{\|\mathbf{1}_{Q_k}\|_{L^1(\rn)}\|\psi\ast F_i\|_{L^\infty(Q_k)}\|G_i\|_{\bmo(\rn)}}.$$
Then, for almost every $x\in\rn$,
\begin{equation}\label{trm21}
(\psi\ast F_i)(x)G_i(x)=\sum_{k\in\mathbb{Z}^n}\lambda_ka_k(x).
\end{equation}
From \cite[p.\,36, Corollary 1(4)]{g} and the fact that $G_i\in\bmo(\rn)$, it follows that
$\psi\ast G_i\in L^\infty(\rn)$,
which, together with \eqref{trm21} and the fact that $F_i\in L^1(\rn)$, further implies
that $(\psi\ast F_i)G_i\in L^1(\rn)\subset\cs'(\rn)$ and
\begin{equation}\label{serm21}
(\psi\ast F_i)G_i=\sum_{k\in\mathbb{Z}^n}\lambda_ka_k\qquad in\;\cs'(\rn).
\end{equation}
Since, for any $k\in\mathbb{Z}^n$, $|Q_k|>1$ and
\begin{align}\label{trm51}
\|a_k\|_{L^2(Q_k)}&=\frac{\|(\psi\ast F_i)G_i\|_{L^2(Q_k)}}
{\|\mathbf{1}_{Q_k}\|_{L^1(\rn)}\|\psi\ast F_i\|_{L^\infty(Q_k)}\|G_i\|_{\bmo(\rn)}}\\\noz
&\le\frac{\|G_i\|_{L^2(Q_k)}\|\psi\ast F_i\|_{L^\infty(Q_k)}}
{\|\mathbf{1}_{Q_k}\|_{L^1(\rn)}\|\psi\ast F_i\|_{L^\infty(Q_k)}\|G_i\|_{\bmo(\rn)}}\lesssim|Q_k|^{-\frac{1}{2}},
\end{align}
it follows that,
for any $k\in\mathbb{Z}^n$, $a_k$ is a local-$(1,2,0)$-atom.
Therefore, to show \eqref{lne4}, by \eqref{serm21} and  \cite[Theorem 5.1]{t}, it suffices to prove that
\begin{equation}\label{erm22}
\sum_{k\in\mathbb{Z}^n}\lambda_k
\lesssim \|F_i\|_{h^1(\rn)}
\|G_i\|_{\bmo(\rn)}.
\end{equation}
Indeed, for any fixed $x_0\in\rn$, there exists only a $k_0\in\mathbb{Z}^n$
such that $x_0 \in Q_{k_0}$. Moreover, there exists a positive
constant $c_0$ such that, for any $k\in\mathbb{Z}^n$ and $x\in Q_k$,
$Q_k\subset B(x,c_0)$.
Then, from the definition of $\{\lambda_k\}_{k\in\mathbb{Z}^n}$, we deduce that
\begin{align*}
\sum_{k\in\mathbb{Z}^n}\frac{\lambda_k}{\|\mathbf{1}_{Q_k}\|_{L^1(\rn)}}
\mathbf{1}_{Q_k}(x_0)&=\frac{\lambda_{k_0}}
{\|\mathbf{1}_{Q_{k_0}}\|_{L^1(\rn)}}
\le\lf\|\psi\ast F_i\r\|_{L^\infty(Q_{k_0})}\|G_i\|_{\bmo(\rn)}\\\noz
&\leq\sup_{y\in B(x_0,c_0)}\lf|\psi\ast F_i(y)\r|\|G_i\|_{\bmo(\rn)}\\\noz
&\leq\sup_{t\in(0,2)}\left\{\sup_{y\in B(x_0,c_0t)}
\lf|\psi_t\ast F_i(y)\r|\right\}\|G_i\|_{\bmo(\rn)}\\\noz
&\lesssim\sup_{t\in(0,1)}\left\{\sup_{y\in B(x_0,2c_0t)}
\lf|\widetilde{\psi}_t\ast F_i(y)\r|\right\}\|G_i\|_{\bmo(\rn)},
\end{align*}
where  $\widetilde{\psi}(\cdot):=\psi(\frac{\cdot}{2})$.
By this and \cite[Theorem 5.3]{shyy}, we conclude that
\begin{align*}
\sum_{k\in\mathbb{Z}^n}\lambda_k&=\sum_{k\in\mathbb{Z}^n}
\frac{\lambda_k}
{\|\mathbf{1}_{Q_k}\|_{L^1(\rn)}}|Q_k|=
\lf\| \sum_{k\in\mathbb{Z}^n}\frac{\lambda_k}
{\|\mathbf{1}_{Q_k}\|_{L^1(\rn)}}
\mathbf{1}_{Q_k}\r\|_{L^1(\rn)} \\\noz
&\lesssim \lf\|\sup_{t\in(0,1)}\left\{\sup_{y\in B(x_0,2c_0t)}
\lf|\widetilde{\psi}_t\ast F_i(y)\r|\right\}\r\|_{L^1(\rn)}\|G_i\|_{\bmo(\rn)}\lesssim
 \|f\|_{h^1(\rn)}\|G_i\|_{\bmo(\rn)},
\end{align*}
which further implies that \eqref{erm22} holds true and hence \eqref{lne4} holds true.
This finishes the proof of
\eqref{lne3}.

Now we show that
\begin{align*}
\mathbf{F}_1\cdot\mathbf{G}\in h_\ast^{\Phi}(\rn).
\end{align*}
Let $(F_1^{(1)},\,\ldots,\,F_n^{(1)}):=\mathbf{F}_1$.
By Theorem \ref{mainthm7}(i), we know that there exist two bounded bilinear operators
$$S:\, H^1(\rn)\times \bmo(\mathbb{R}^n)\to L^1(\rn)$$ and
$$T:\, H^1(\rn)\times \bmo(\mathbb{R}^n)\to H_\ast^{\Phi}(\rn)$$
such that
\begin{align*}
\mathbf{F}_1\cdot \mathbf{G}&=\dsum_{i=1}^n F_i^{(1)}\times G_i= \dsum_{i=1}^n S(F_i^{(1)},\,G_i)+ \dsum_{i=1}^nT(F_i^{(1)},\,G_i)\\
&=:\mathrm{A}(\mathbf{F}_1,\,\mathbf{G})+\mathrm{B}(\mathbf{F}_1,\,\mathbf{G})
\qquad \text{in}\;\cs'(\rn).
\end{align*}
Moreover, by \eqref{dof1}, we know that
$\mathrm{B}(\mathbf{F}_1,\,\mathbf{G})\in h^{\Phi_p}(\rn)$ and
\begin{align}\label{00a}
\lf\|\mathrm{B}(\mathbf{F}_1,\,\mathbf{G})\r\|_{h_\ast^{\Phi}(\rn)}&\le
\lf\|\mathrm{B}(\mathbf{F}_1,\,\mathbf{G})\r\|_{H_\ast^{\Phi}(\rn)}
\ls \|\mathbf{F}_1\|_{H^1(\rn;\,\rn)}
\|\mathbf{G}\|_{\bmo(\rn;\,\rn)}\\\noz
&\ls \|\mathbf{F}\|_{h^1(\rn;\,\rn)}
\|\mathbf{G}\|_{\bmo(\rn;\,\rn)}.
\end{align}
By an argument similar to that used in the estimation of
\eqref{007}, we also obtain
\begin{align*}
\lf\|\mathrm{A}(\mathbf{F}_1,\,\mathbf{G})\r\|_{h_\ast^{\Phi}(\rn)}\ls
\|\mathbf{F}\|_{h^1(\rn;\,\rn)}\|\mathbf{G}\|_{\bmo(\rn;\,\rn)}.
\end{align*}
This, combined with \eqref{dof5} and \eqref{00a}, shows that
\begin{align*}
\| \mathbf{F}\cdot
\mathbf{G}\|_{h_\ast^{\Phi}(\rn)} \le C \|\mathbf{F}\|_{h^1(\rn;\,\rn)}
\|\mathbf{G}\|_{\bmo(\rn;\,\rn)},
\end{align*} and hence
finishes the proof of Theorem \ref{sssa}.
\end{proof}

\begin{remark}
\begin{itemize}
\item[\rm (i)] In Theorem \ref{sssa}, if we replace $0$-inhomogeneous $\mathrm{curl}\,\mathbf{F}=0$
by ${\rm curl}\, \mathbf{F}\equiv0$, using Proposition \ref{weak},
we know that the corresponding conclusion of Theorem \ref{sssa}
still holds true, which is just \cite[Theorem 1.2]{cky1}.
In this case, we seal the gap appearing in the proof of \cite[Theorem 1.2]{cky1}.
\item[\rm (ii)]
In Theorem \ref{sssa}, if we replace $0$-inhomogeneous $\mathrm{curl}\,\mathbf{F}=0$, $h^1(\rn;\,\rn)$,
$\bmo(\rn;\,\rn)$ and $h_\ast^{\Phi}(\rn)$
replaced, respectively, by ${\rm curl}\, \mathbf{F}\equiv0$,
$H^1(\rn;\,\rn)$, $\BMO(\rn;\,\rn)$ and $H^{\log}(\rn)$
as in \eqref{11z},
then the corresponding conclusion of Theorem \ref{sssa}
still holds true, which is just \cite[Theorem 1.2]{BGK12}.
Observe that
$H^1(\rn;\,\rn)\subset h^1(\rn;\,\rn)$,
$\bmo(\rn;\,\rn)\subset
\BMO(\rn;\,\rn)$
and $H^{\log}(\rn)\subset h^{\log}(\rn)$.
Thus, it is quite natural to use the space $h^{\log}(\rn)$
as the target space in Theorem \ref{sssa}; however,
by \cite[Theorem 3.6(v)]{zyy}, we know that the target space $h_\ast^{\Phi}(\rn)$
in Theorem \ref{sssa} satisfies
$h_\ast^{\Phi}(\rn)\subsetneqq h^{\log}(\rn)$. In this sense, Theorem \ref{sssa} is a little bit surprising.
Moreover, by Proposition \ref{weak}, we conclude that the assumption,
$0$-inhomogeneous $\mathrm{curl}\,\mathbf{F}=0$, in Theorem \ref{sssa}
is weaker than the assumption ${\rm curl}\, \mathbf{F}\equiv0$ in
\cite[Theorem 1.5]{BGK12}, which perfectly matches
the locality of $h^1(\rn;\,\rn)$ and $\bmo(\rn;\,\rn)$.
\end{itemize}
\end{remark}

\noindent\textbf{Acknowledgements}.\quad Yangyang Zhang would like to thank Professors Liguang Liu
and Jun Cao for some helpful discussions on the topic of this article.

\bigskip

\noindent Yangyang Zhang, Dachun Yang (Corresponding author) and Wen Yuan

\smallskip

\noindent  Laboratory of Mathematics and Complex Systems
(Ministry of Education of China),
School of Mathematical Sciences, Beijing Normal University,
Beijing 100875, People's Republic of China

\smallskip

\noindent {\it E-mails}: \texttt{yangyzhang@mail.bnu.edu.cn} (Y. Zhang)

\noindent\phantom{{\it E-mails:}} \texttt{dcyang@bnu.edu.cn} (D. Yang)

\noindent\phantom{{\it E-mails:}} \texttt{wenyuan@bnu.edu.cn} (W. Yuan)


\begin{thebibliography}{10}

\bibitem{AIKM00}
K. Astala, T. Iwaniec, P. Koskela and G. Martin,
Mappings of BMO-bounded distortion,
Math. Ann. 317 (2000), 703-726.

\vspace{-0.3cm}

\bibitem{Ba77}
J. M. Ball,
Convexity conditions and existence theorems in nonlinear elasticity,
Arch. Rational Mech. Anal. 63 (1976/77), 337-403.

\vspace{-0.3cm}

\bibitem{bl}J. Bergh and J. L\"ofstr\"om,
Interpolation Spaces, An introduction,
Grundlehren der Mathematischen Wissenschaften, No. 223, Springer-Verlag, Berlin-New York, 1976.

\vspace{-0.3cm}

\bibitem{bckly}A. Bonami, J. Cao, L. D. Ky, L. Liu, D. Yang and W. Yuan,
Multiplication between Hardy spaces and their dual spaces,
J. Math. Pures Appl. (9) 131 (2019), 130-170.

\vspace{-.3cm}

\bibitem{bf} A. Bonami and J. Feuto, Products of functions in Hardy and Lipschitz or BMO spaces,
in:
Recent Developments in Real and Harmonic Analysis, 57-71,
Appl. Numer. Harmon. Anal., Birkh\"{a}user Boston, Inc., Boston, MA, 2010.

\vspace{-0.3cm}

\bibitem{BFG}
A. Bonami, J. Feuto and S. Grellier,
Endpoint for the DIV-CURL lemma in Hardy spaces,
Publ. Mat. 54 (2010), 341-358.



\vspace{-0.3cm}

\bibitem{bfgk}
A. Bonami, J. Feuto, S. Grellier and L. D. Ky,
Atomic decomposition and weak factorization in generalized Hardy spaces of closed forms,
Bull. Sci. Math. 141 (2017), 676-702.

\vspace{-0.3cm}

\bibitem{BGK12}
A. Bonami, S. Grellier and L. D. Ky,
Paraproducts and products of functions in ${\rm BMO}({\mathbb R}^n)$ and $H^1({\mathbb R}^n)$ through wavelets,
J. Math. Pures Appl. (9) 97 (2012), 230-241.


\vspace{-0.3cm}

\bibitem{BK}
A. Bonami and L. D. Ky,
Factorization of some Hardy-type spaces of holomorphic functions,
C. R. Math. Acad. Sci. Paris 352 (2014), 817-821.

\vspace{-0.3cm}

\bibitem{BIJZ07}
A. Bonami, T. Iwaniec, P. Jones and M. Zinsmeister,
On the product of functions in BMO and $H^1$,
Ann. Inst. Fourier (Grenoble) 57 (2007), 1405-1439.


\vspace{-0.3cm}

\bibitem{blyy}
A. Bonami, L. Liu, D. Yang and W. Yuan,
Pointwise multipliers of Zygmund classes on $\rn$, Submitted.


\vspace{-0.3cm}

\bibitem{Ca63}
S. Campanato,
Propriet\`a di h\"olderianit\`a di alcune classi di funzioni,
Ann. Scuola Norm. Sup. Pisa (3) 17 (1963), 175-188.

\vspace{-0.3cm}

\bibitem{Ca64}
S. Campanato,
Propriet\`a di una famiglia di spazi funzionali,
Ann. Scuola Norm. Sup. Pisa (3) 18 (1964), 137-160.

\vspace{-0.3cm}

\bibitem{c}
J. Cao, L. Liu, D. Yang and W. Yuan,
Intrinsic structures of certain Musielak-Orlicz Hardy spaces,
J. Geom. Anal. 28 (2018), 2961-2983.

\vspace{-.3cm}

\bibitem{cky1} J. Cao, L. D. Ky and D. Yang, Bilinear decompositions of products
of local Hardy and Lipschitz or BMO spaces through wavelets,
Commun. Contemp. Math. 20 (2018), 1750025, 30 pp.

\vspace{-.3cm}

\bibitem{cdy16} D.-C. Chang, G. Dafni and H. Yue, A note on nonhomogenous weighted div-curl
lemmas, in: Harmonic Analysis, Partial Differential Equations, Complex Analysis, Banach Spaces, and Operator Theory,
Vol. 1, 143-152, Assoc. Women Math. Ser., 4, Springer, [Cham], 2016.

\vspace{-.3cm}

\bibitem{cdy10} D.-C. Chang, G. Dafni and H. Yue, Nonhomogeneous div-curl decompositions for local
Hardy spaces on a domain, in: Hilbert Spaces of Analytic Functions, 153-163, CRM Proc. Lecture
Notes, 51, Amer. Math. Soc., Providence, RI, 2010.

\vspace{-.3cm}

\bibitem{cdy09} D.-C. Chang, G. Dafni and H. Yue, A div-curl decomposition for the local Hardy
space, Proc. Amer. Math. Soc. 137 (2009), 3369-3377.

\vspace{-.3cm}

\bibitem{cds05} D.-C. Chang, G. Dafni and C. Sadosky, A div-curl lemma in BMO on a domain,
in: Harmonic Analysis, Signal Processing, and Complexity, 55-65, Progr. Math., 238,
Birkh\"auser Boston, Boston, MA, 2005.

\vspace{-0.3cm}

\bibitem{CDM95}
R. R. Coifman, S. Dobyinsky and Y. Meyer,
Op\'erateurs bilin\'eaires et renormalisation,
in: Essays on Fourier Analysis in Honor of Elias M. Stein (Princeton, NJ, 1991), 146-161,
Princeton Math. Ser. 42, Princeton Univ. Press, Princeton, NJ, 1995.

\vspace{-0.3cm}


\bibitem{CLMS}
R. R. Coifman, P.-L.  Lions, Y. Meyer and S. Semmes,
Compensated compactness and Hardy spaces, J. Math. Pures Appl. (9) 72 (1993), 247-286.

\vspace{-0.3cm}

\bibitem{Da95}
G. Dafni,
Nonhomogeneous div-curl lemmas and local Hardy spaces,
Adv. Differential Equations 10 (2005), 505-526.

\vspace{-0.3cm}

\bibitem{Da88}
I. Daubechies,
Orthonormal bases of compactly supported wavelets,
Comm. Pure Appl. Math. 41 (1988), 909-996.

\vspace{-.3cm}

\bibitem{Do95}
S. Dobyinsky,
La ``version ondelettes" du th\'eor\'eme du Jacobien,
Rev. Mat. Iberoam. 11 (1995),  309-333.

\vspace{-0.3cm}

\bibitem{FJ90}
M. Frazier and B. Jawerth,
A discrete transform and decompositions of distribution spaces,
J. Funct. Anal. 93 (1990), 34-170.


\vspace{-.3cm}


\bibitem{FYL} X. Fu,  D. Yang and Y. Liang, Products of functions in $BMO(\cx)$
and $H^1_{\rm at}(\cx)$ via wavelets over spaces of homogeneous type,
J. Fourier Anal. Appl. 23 (2017), 919-990.


\vspace{-0.3cm}

\bibitem{GM01}
J.  Garc\'ia-Cuerva and J. M. Martell,
Wavelet characterization of weighted spaces,
J. Geom. Anal. 11 (2001), 241-264.


\vspace{-0.3cm}

\bibitem{GR85}
J.  Garc\'ia-Cuerva and J. Rubio de Francia,
Weighted Norm Inequalities and Related Topics,
North-Holland Mathematics Studies 116, Notas de Matem\'atica [Mathematical Notes] 104,
North-Holland Publishing Co., Amsterdam, 1985.


\vspace{-.3cm}

\bibitem{g} D. Goldberg, A local version of real Hardy spaces,
Duke Math. J. 46 (1979), 27-42.


\vspace{-0.3cm}

\bibitem{G1} L. Grafakos, Classical Fourier Analysis, third edition,
Grad. Texts in Math. 249, Springer, New York, 2014.

\vspace{-0.3cm}

\bibitem{g14} L. Grafakos, Modern Fourier Analysis, third edition,
Grad. Texts in Math. 250, Springer, New York, 2014.

\vspace{-0.3cm}

\bibitem{He}
F. H\'elein,  Regularity of weakly harmonic maps from a surface into a manifold with symmetries,
{Manuscripta Math.} 70 (1991),  203-218.


\vspace{-.3cm}

\bibitem{HW96}
E. Hern\'andez and G. Weiss,
A First Course on Wavelets, with a foreword by Yves Meyer
Studies in Advanced Mathematics,
CRC Press, Boca Raton, FL, 1996.


\vspace{-0.3cm}

\bibitem{IS}
T. Iwaniec and C. Sbordone, Quasiharmonic fields,
Ann. Inst. H. Poincar\'e Anal. Non Lin\'eaire 18 (2001),  519-572.


\vspace{-0.3cm}

\bibitem{kf} A. N. Kolmogorov and S. V. Fomin, Measure, Lebesgue Integrals, and Hilbert Space,
Academic Press, New York-London 1961.


\vspace{-0.3cm}

\bibitem{Ky13} L. D. Ky, Bilinear decompositions and commutators of singular integral operators,
Trans. Amer. Math. Soc. 365 (2013), 2931-2958.

\vspace{-0.3cm}

\bibitem{Ky14}
L. D. Ky,
New Hardy spaces of Musielak--Orlicz type and boundedness of subilinear operators,
Integral Equations Operator Theory 78 (2014), 115-150.


\vspace{-0.3cm}

\bibitem{LLL09}
M.-Y. Lee, C.-C. Lin and Y.-C. Lin,
A wavelet characterization for the dual of weighted Hardy spaces,
Proc. Amer. Math. Soc. 137 (2009), 4219-4225.


\vspace{-0.3cm}

\bibitem{Le02}
P. G. Lemari\'e-Rieusset,
Recent Developments in the Navier--Stokes Problem,
Chapman $\&$ Hall/CRC Research Notes in Mathematics 431,
Chapman $\&$ Hall/CRC, Boca Raton, FL, 2002.

\vspace{-0.3cm}

\bibitem{LY12}
Y. Liang, Y. Sawano, T. Ullrich, D. Yang and W. Yuan,
New characterizations of Besov--Triebel--Lizorkin--Hausdorff spaces
including coorbits and wavelets, J. Fourier Anal. Appl. 18 (2012), 1067-1111.



\vspace{-0.3cm}

\bibitem{LCFY} L. Liu, D.-C. Chang, X. Fu and D. Yang,
Endpoint boundedness of commutators on spaces of homogeneous type,
Appl. Anal. 96 (2017), 2408-2433.

\vspace{-.3cm}


\bibitem{Lu95}
S. Lu,
Four Lectures on Real $H^p$ Spaces,
World Scientific Publishing Co., Inc., River Edge, NJ, 1995.

\vspace{-0.3cm}

\bibitem{Me92}
Y. Meyer, Wavelets and Operators, Translated from the 1990 French original by D. H. Salinger,
Cambridge Studies in Advanced Mathematics 37,
Cambridge University Press, Cambridge, 1992.


\vspace{-0.3cm}

\bibitem{m}
Y. Meyer and  R. Coifman, Wavelets, Calder\'on--Zygmund and Multilinear Operators,
Cambridge Studies in Advanced Mathematics 48, Cambridge University Press, Cambridge, 1997.

\vspace{-0.3cm}

\bibitem{Mu88}
S. M\"uller,
Weak continuity of determinants and nonlinear elasticity,
C. R. Acad. Sci. Paris S\'er. I Math. 307 (1988), 501-506.


\vspace{-0.3cm}

\bibitem{Mu90}
S. M\"uller,
Higher integrability of determinants and weak convergence in $L^1$,
J. Reine Angew. Math. 412 (1990), 20-34.


\vspace{-0.3cm}

\bibitem{Mur78}
F. Murat,
Compacit\'e par compensation,
Ann. Scuola Norm. Sup. Pisa Cl. Sci. (4) 5 (1978), 489-507.

\vspace{-0.3cm}

\bibitem{n17} E. Nakai, Pointwise multipliers on several functions
spaces--a survey, Linear Nonlinear Anal. 3 (2017), 27-59.

\vspace{-0.3cm}

\bibitem{ns} E. Nakai and Y. Sawano,
Hardy spaces with variable expoments and generalized Campanato spaces,
J. Funct. Anal. 262 (2012), 3665-3748.


\vspace{-0.3cm}

\bibitem{NY85}
E. Nakai and K. Yabuta,
Pointwise multipliers for functions of bounded mean oscillation,
J. Math. Soc. Japan 37 (1985), 207-218.

\vspace{-0.3cm}

\bibitem{mmz} M. M. Rao and Z. D. Ren, Theory of Orlicz spaces,
Monographs and Textbooks in Pure and Applied Mathematics 146,
Marcel Dekker, Inc., New York, 1991.


\vspace{-0.3cm}

\bibitem{R} W. Rudin, Functional Analysis, Second edition,
International Series in Pure and Applied Mathematics, McGraw-Hill, Inc., New York, 1991.


\vspace{-.3cm}

\bibitem{shyy} Y. Sawano, K.-P. Ho, D. Yang and S. Yang, Hardy spaces for ball
quasi-Banach function spaces,
Dissertationes Math. (Rozprawy Mat.) 525 (2017), 1-102.

\vspace{-0.3cm}

\bibitem{tw} M. H. Taibleson and G. Weiss, The molecular characterization of certain Hardy
spaces, in: Representation theorems for Hardy spaces, pp. 67-149, Ast\'erisque, 77, Soc. Math. France,
Paris, 1980.


\vspace{-0.3cm}

\bibitem{t}
L. Tang, Weighted local Hardy spaces and their applications, Illinois J. Math, 56 (2012), 453-495.


\vspace{-0.3cm}

\bibitem{Ta79}
L. Tartar, Compensated compactness and applications to partial differential equations,
in: Nonlinear Analysis and Mechanics: Heriot--Watt Symposium, Vol. IV,  pp. 136-212,
Res. Notes in Math. 39, Pitman, Boston, Mass.--London, 1979.


\vspace{-0.3cm}


\bibitem{gg}
H. Triebel, Theory of Function Spaces. III, Monographs in Mathematics 100, Birkh\"auser Verlag, Basel, 2006.


\vspace{-.3cm}

\bibitem{wyyz} S. Wang, D. Yang, W. Yuan and Y. Zhang,
Weak Hardy-type spaces associated with ball quasi-banach function spaces
II: Littlewood--Paley characterizations and real interpolation,
J. Geom. Anal. (2019), DOI: 10.1007/s12220-019-00293-1.

\vspace{-0.3cm}

\bibitem{ylk} D. Yang, Y. Liang and L. D. Ky, Real-Variable Theory of Musielak--Orlicz
Hardy Spaces, Lecture Notes in Math. 2182, Springer, 2017.


\vspace{-.3cm}
\bibitem{yy2} D. Yang and S. Yang, Local Hardy spaces of
Musielak--Orlicz type and their applications, Sci. China Math. 55 (2012), 1677-1720.


\vspace{-0.3cm}

\bibitem{yy} D. Yang and S. Yang, Weighted local Orlicz Hardy spaces with
applications to pseudo-differential operators, Dissertationes Math. (Rozprawy Mat.) 478 (2011), 1-78.

\vspace{-.3cm}

\bibitem{y} K. Yosida, Functional Analysis, Reprint of the sixth (1980) edition,
Classics in Mathematics, Springer-Verlag, Berlin, 1995.



\vspace{-.3cm}

\bibitem{ZWYY} Y. Zhang, S. Wang, D. Yang and W. Yuan, Weak Hardy-type spaces associated with
ball quasi-Banach function spaces I: Decompositions with applications to boundedness of
Calder\'on--Zygmund operators, Sci. China Math. (2020), DOI: 10.1007/s11425-019-1645-1.

\vspace{-.3cm}

\bibitem{zyy}
Y. Zhang, D. Yang and W. Yuan,
Real-variable characterizations of
local Orlicz-slice Hardy spaces with application to bilinear decompositions,
Submitted.


\vspace{-.3cm}

\bibitem{zyyw}
Y. Zhang, D. Yang, W. Yuan and S. Wang,
Real-variable characterizations of Orlicz-slice Hardy spaces,
Anal. Appl. (Singap.) 17 (2019), 597-664.


\end{thebibliography}
\end{document}